\documentclass[reqno]{amsart}
\usepackage[utf8]{inputenc}

\usepackage[utf8]{inputenc} 
\usepackage{amssymb}
\usepackage{mathrsfs}
\usepackage{graphicx}
\usepackage{latexsym}
\usepackage{stmaryrd}
\usepackage{enumitem}
\usepackage[bookmarks,hyperfootnotes=false]{hyperref}
\usepackage{multirow}
\usepackage{xcolor}
\usepackage{caption}
\colorlet{darkishRed}{red!60!black}
\colorlet{darkishBlue}{blue!60!black}
\colorlet{darkishGreen}{green!60!black}
\hypersetup{
    draft = false,
    bookmarksopen=true,
    colorlinks,
    linkcolor={blue!60!black},
    citecolor={green!60!black},
    urlcolor={blue!60!black}
}
\usepackage[nameinlink, capitalise, noabbrev]{cleveref}
\crefformat{enumi}{#2#1#3}
\crefformat{equation}{#2(#1)#3}
\usepackage[msc-links]{amsrefs} 
\usepackage{doi}

\renewcommand{\PrintDOI}[1]{\doi{#1}}

\let\setminus=\smallsetminus
\usepackage{tikz}
\usepackage{tikz-cd}
\usetikzlibrary{calc,through,intersections,arrows, trees, positioning, decorations.pathmorphing, cd}
\usepackage{relsize}
\usepackage{comment}
\usepackage{svg}
\usepackage{mathtools}
\usepackage{nccmath}
\usepackage{pifont}
\usepackage{comment}

\usepackage{geometry}
\let\setminus=\smallsetminus
\renewcommand{\leq}{\leqslant}
\renewcommand{\geq}{\geqslant}

\usepackage{subcaption}
\def\namedlabel#1#2{\begingroup
   \def\@currentlabel{#2}%
   \label{#1}\endgroup
}

\let\eps=\varepsilon
\let\rho=\varrho
\let\phi=\varphi

\newcommand{\crit}{\normalfont\text{crit}}

\newcommand{\torso}{\normalfont\text{torso}}
\newcommand{\cdeg}{\Delta}
\newcommand{\Dom}{\normalfont{\text{Dom}}}
\newcommand{\dom}{\normalfont{\text{dom}}}

\DeclareMathOperator{\interior}{int}

\newcommand{ \N } { \mathbb{N} }

\newcommand{\defn}[1]{{\color{darkishRed}{\emph{#1}}}}
\newcommand{\defnm}[1]{{\color{darkishRed}{#1}}}

\newcommand{\abs}[1]{\lvert#1\rvert}

\newcommand{\COMMENT}[1]{}

\makeatletter

\def\calCommandfactory#1{%
   \expandafter\def\csname c#1\endcsname{\mathcal{#1}}}
\def\frakCommandfactory#1{%
   \expandafter\def\csname frak#1\endcsname{\mathfrak{#1}}}

\newcounter{ctr}
\loop
  \stepcounter{ctr}
  \edef\X{\@Alph\c@ctr}
  \expandafter\calCommandfactory\X
  \expandafter\frakCommandfactory\X
  \edef\Y{\@alph\c@ctr}
  \expandafter\frakCommandfactory\Y
\ifnum\thectr<26
\repeat

\setenumerate{label={\normalfont (\roman*)}}

\usepackage[presets={vec-cev,abc,ABC,cAcBcC}]{letterswitharrows}

\newtheorem{theorem}{Theorem}[section] 
\newtheorem{proposition}[theorem]{Proposition}

\newtheorem{lemma}[theorem]{Lemma}

\newtheorem{mainresult}{Theorem} 
\crefname{mainresult}{Theorem}{Theorems}

\newtheorem{LEM}[theorem]{Lemma}
\crefname{LEM}{Lemma}{Lemmas}
\newtheorem{THM}[theorem]{Theorem}
\crefname{THM}{Theorem}{Theorems}
\newtheorem{COR}[theorem]{Corollary}

\newtheorem{PROP}[theorem]{Proposition}
\crefname{PROP}{Proposition}{Propositions}

\theoremstyle{definition}

\newtheorem{claim}{Claim}
\crefname{claim}{Claim}{Claims}
\AtEndEnvironment{proof}{\setcounter{claim}{0}}

\newenvironment{claimproof}{\noindent\textit{Proof.}}{\hfill\ensuremath{\blacksquare}\medskip}
\usepackage{etoolbox}

\newlist{thmlist}{enumerate}{1}
\setlist[thmlist]{label=(\roman{thmlisti}), ref=\thethm.(\roman{thmlisti}),noitemsep}

\def\td{tree-decom\-posi\-tion}
\def\tds{tree-decom\-posi\-tions}

\newtheorem{innercustomthm}{}

\newenvironment{customthm}[1]
  {\renewcommand\theinnercustomthm{#1}\innercustomthm}
  {\endinnercustomthm}

\theoremstyle{definition}
\newtheorem{example}[theorem]{Example}

\newtheorem{construction}[theorem]{Construction}

\theoremstyle{remark}

\makeatletter
\newcommand\footnoteref[1]{\protected@xdef\@thefnmark{\ref{#1}}\@footnotemark}
\makeatother

\def\bigskip{\vspace{14pt}}

\title{Tangle-tree duality in infinite graphs}
\author{Sandra Albrechtsen}
\address{University of Hamburg, Department of Mathematics, Bundesstraße 55 (Geomatikum), 20146 Hamburg, Germany\\
(now at Leipzig University)}
\email{sandra@albrechtsen-mail.de}

\subjclass{05C83, 05C63, 05C40, 05C05}
\keywords{Tangle-tree duality, tree-decomposition, tree-width, infinite graphs, tree of tangles, bramble, branch-width.}

\begin{document}

\begin{abstract}
    We extend Robertson and Seymour's tangle-tree duality theorem to infinite graphs. 
\end{abstract}

\maketitle

\section{Introduction}

\emph{Tree-decompositions} are a central object in structural graph theory. They were not only a crucial tool in the Graph Minor Project of Robertson and Seymour \cite{GM}, but also attracted attention as several computationally hard problems can be solved efficiently on graphs of small tree-width.
Because of this, the question arose which graphs have \emph{small tree-width}, that is, admit a \td\ into bags that all contain only few vertices, and, conversely, what kind of substructures prevent a graph from having small tree-width. 

There are a number of substructures, e.g.\ large grid or clique minors, or $k$-blocks for large $k$, that are known to force a graph to have large tree-width. While these objects differ in their concrete shape, they have one thing in common: they witness high cohesion somewhere in the graph.

In their Graph Minors Project \cite{GM}, Robertson and Seymour introduced \emph{tangles} as a unified way to capture all such highly cohesive substructures in a graph.
A $k$-tangle in a graph~$G$ is a certain orientation of all its separations of order less than $k$. 
The idea is that every highly cohesive substructure of~$G$ will lie mostly on one side of such a low-order separation, and therefore orient it towards that side. All these orientations, collectively, are then called a tangle.

One of the two major theorems in Robertson and Seymour's  original work on tangles is the following duality between tangles of high order and small tree-width \cite{GMX}, rephrased here in the terminology of \cite{DiestelBook16noEE}: 

\begin{mainresult}\label{main:TTDFiniteGraphs} 
For every finite graph $G$ and $k \in \N$, exactly one of the following assertions holds:
\begin{enumerate}[label=\rm{(\roman*)}]
    \item\label{itm:TTDFiniteGraphs:Tangle} There exists a $k$-tangle in $G$.
    \item\label{itm:TTDFiniteGraphs:Tree} There exists an $S_k$-tree over $\cT^*$.
\end{enumerate}
\end{mainresult}

\noindent \cref{main:TTDFiniteGraphs} is known as the \emph{tangle-tree duality theorem}. For a definition of $S_k$-trees over~$\cT^*$ we refer the reader to \cref{sec:Prelims}; see also \cref{fig:Intro:Finite}. Note that a finite graph $G$ has an $S_k$-tree over~$\cT^*$ if and only if~$G$ has branch-width $<k$.

\cref{main:TTDFiniteGraphs} implies an approximate duality for tangles and tree-width: every graph with a $k$-tangle has tree-width at least $k-1$, while a graph with a tree as in \ref{itm:TTDFiniteGraphs:Tree} has tree-width at most~$\lfloor 3k/2\rfloor-1$ \cite{GMX}*{(5.1)}. 
\medskip

\begin{figure}[ht]
    \centering
    \includegraphics[width=0.5\linewidth]{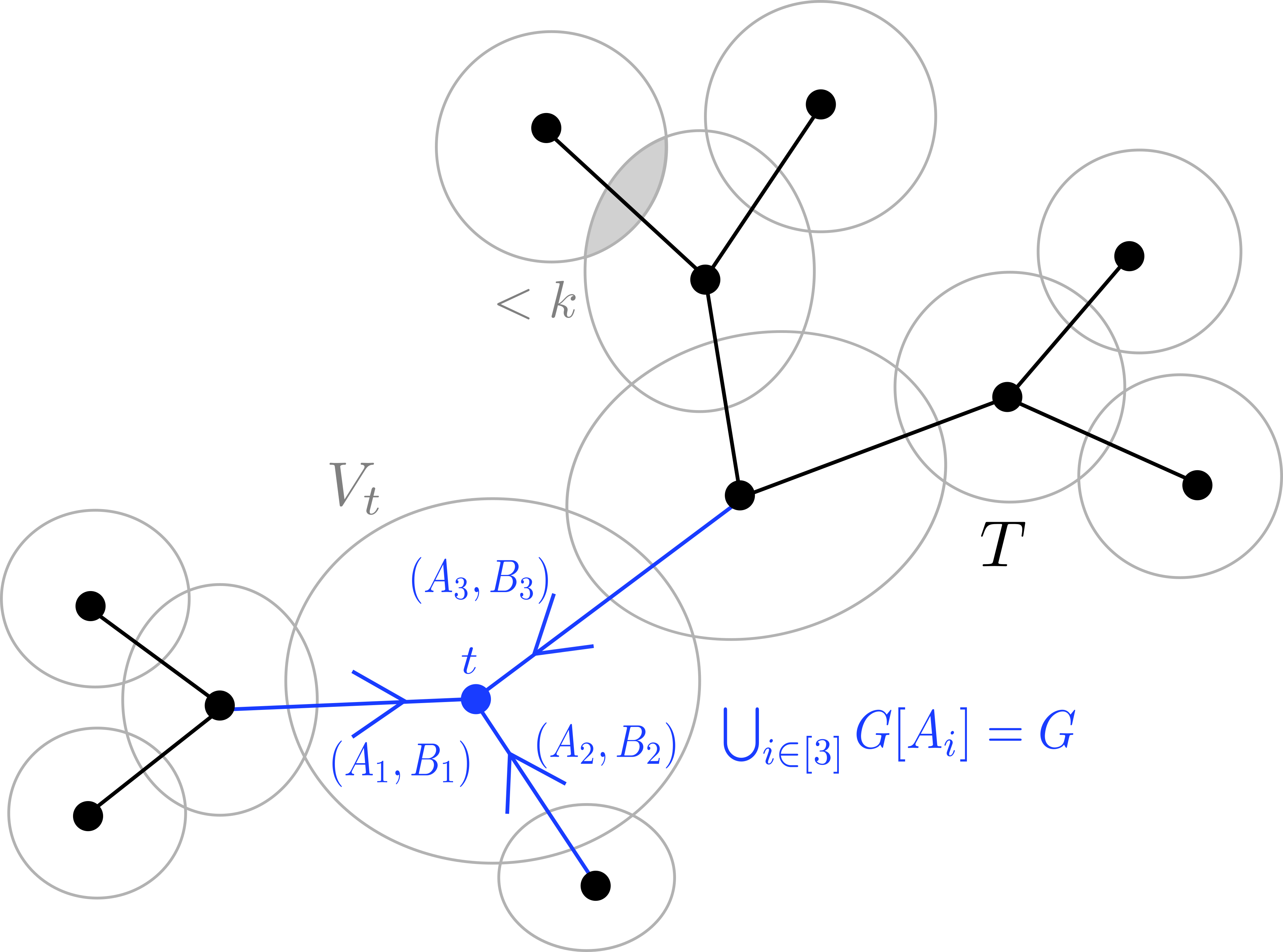}
    \caption{An $S_k$-tree is a pair $(T, \alpha)$ of a tree $T$ and an order-preserving map $\alpha$ which maps, for every edge $e$ of $T$, its orientations $\ve, \ev$ to the two orientations $(A,B), (B,A)$ of a separation $\{A,B\}$ of $G$ of order $|A \cap B|<k$ (cf.\ \cref{subsec:STrees}). 
    If $(T, \alpha)$ is over $\cT^*$, then $T$ has maximum degree~$3$, and for every node $t$ of $T$, the $A_i$-sides of the separations $(A_i,B_i) := \alpha(t't)$, for $t' \in N_T(t)$, cover $G$.\\
    $(T, \alpha)$ induces a \td\ with decomposition tree $T$ whose bags $V_t$ are the intersections $\bigcap_{i \in [3]} B_i$ of the $B_i$-sides of the separations $(A_i,B_i) := \alpha(t't)$, for $t' \in N_T(t)$. This \td\ has adhesion $<k$, and width $\leq 3k-4$.}
    \label{fig:Intro:Finite}
\end{figure}

The definition of a tangle extends verbatim to infinite graphs. There are several papers that extend results about tangles in finite graphs to infinite ones, or which deal with new questions that arise from tangles in infinite graphs only \cites{ToTinfOrder, EndsAndTangles, CanonicalTreesofTDs, InfiniteSplinters, PRToTsInLocFinGraphs,KPCompactification,KPEndsTanglesCVSets}. 
In this paper we contribute to this a duality statement about tangles and small tree-width in infinite graphs, which extends \cref{main:TTDFiniteGraphs} to infinite graphs. 

The first thing to notice is that, in infinite graphs, high-order tangles no longer force the tree-width up.
Indeed, every infinite graph $G$ contains a tangle of infinite order~\cite{EndsAndTangles}\COMMENT{Proposition 4.2}, an orientation of all the finite-order separations of $G$. By restricting it to only those oriented separations that have order less than $k$, every such tangle induces a $k$-tangle for every $k \in \N$. Hence, every infinite graph has a $k$-tangle for every $k \in \N$.
However, there are infinite graphs, e.g.\ infinite trees, that have small tree-width.
Thus, in contrast to finite graphs, high-order tangles in infinite graphs are in general not an obstruction to small tree-width. 
Specifically, infinite locally finite trees are an example of graphs that have both a $3$-tangle and an $S_3$-tree over $\cT^*$, thus witnessing that \cref{main:TTDFiniteGraphs} fails for infinite graphs.

So what is the difference between finite and infinite graphs that causes high-order tangles to force large tree-width in finite graphs but not in infinite graphs?
In finite graphs, tangles arise \emph{only} from highly connected substructures (which may be fuzzy) as indicated earlier. 
In infinite graphs, however, there are also tangles that arise from infinite phenomena of the graph that do not reflect high local cohesion.

Let us first consider locally finite graphs. Every infinite locally finite, connected graph $G$ has an \emph{end}, an equivalence class of rays in $G$ where two rays are equivalent if they cannot be separated by deleting finitely many vertices. Every end induces an infinite tangle by orienting every finite-order separation to the side which contains a tail of one (equivalently each) of its rays \cite{EndsAndTangles}. 
The \emph{degree} of an end is the maximum number of disjoint rays in it.

Ends of large degree do force the tree-width up:
It is not difficult to see that every graph with an end of degree at least~$k$ has tree-width at least~$k$.
But this is sharp in the sense of \cref{main:TTDFiniteGraphs}:
For every $k \in \N$, there exists a locally finite graph (e.g.\ the rectangular $(k-1) \times \infty$ grid) whose single end has degree $k-1$, and that has an $S_k$-tree over~$\cT^*$.
In particular, ends of small degree do not force the tree-width up.
So if we want to extend \cref{main:TTDFiniteGraphs} to locally finite graphs in a way that retains its duality between tree structure on the one hand and the existence of high local cohesion on the other hand, we need to adjust \cref{itm:TTDFiniteGraphs:Tangle} to ban tangles that are induced by ends of small degree. 

We also have to adjust \cref{itm:TTDFiniteGraphs:Tree} of \cref{main:TTDFiniteGraphs}, for a different reason. 
Since no infinite graph has a finite $S_k$-tree over $\cT^*$, we have to allow infinite $S_k$-trees in \cref{itm:TTDFiniteGraphs:Tree} when we extend \cref{main:TTDFiniteGraphs} to infinite graphs.
But this creates another problem. For example, consider the graph~$G$ which is obtained from a ray on vertex set $\N$ by gluing a large clique $K$ on to $0$. 
Then $G$ has an infinite $S_4$-tree $(R, \alpha)$ over $\cT^*$, where~$R$ is the natural ray on vertex set $\N$ and $\alpha : \vE(R) \rightarrow \vS_4$ with $\alpha(i, i+1) = (\{0, \dots, i+1\}, \N_{\geq i} \cup V(K))$. But $G$ has large tree-width (and a high-order tangle), as witnessed by the large clique~$K$. 
The problem is that, in contrast to finite $S_k$-trees, $(R, \alpha)$ does not induce a \td. 

Since it is our aim to extend \cref{main:TTDFiniteGraphs} to infinite graphs in a way that retains its duality between tree structure and the existence of high local cohesion, we need to exclude such $S_k$-trees from \cref{itm:TTDFiniteGraphs:Tree}.
This will be formalised by `weakly exhaustive' $S_k$-trees in \cref{sec:Prelims}: every weakly exhaustive $S_k$-tree induces a \td\ (cf.\ \cref{fig:Intro:Infinite}).

Our tangle-tree duality theorem for locally finite graphs then reads as follows:

\begin{mainresult} \label{main:TTDLocFinGraphs}
    For every locally finite, connected graph $G$ and $k \in \N$, exactly one of the following assertions holds:
\begin{enumerate}[label=\rm{(\roman*)}]
    \item\label{itm:TTDLocFin:Tangle} There exists a $k$-tangle in $G$ that is not induced by an end of degree $<k$.
    \item\label{itm:TTDLocFin:Tree} There exists a weakly exhaustive $S_k$-tree over $\cT^*$.
\end{enumerate}
\end{mainresult}

\noindent We remark that a countable graph $G$ has an $S_k$-tree as in \ref{itm:TTDLocFin:Tree} if and only if $G$ has branch-width $<k$.
\medskip

Let us now consider arbitrary infinite graphs. There is another type of tangle that can occur in infinite graphs which also does not reflect any highly cohesive substructure. 
For example, let $G$ be the (infinite) star with center $c$ and $\N$ the set of leaves. Then all `non-trivial' $1$-separations of $G$ are of the form $\{A,B\}$ with $A \cap B = \{c\}$ and $A\setminus B$ and $B\setminus A$ a bipartition of $\N$. Let $\beta$ be a non-principal ultra filter on $\N$, and orient every non-trivial $1$-separation of $G$ towards its side in $\beta$.
This is a $2$-tangle\footnote{For this, we also need to orient all `trivial' $1$-separations (of the form $\{V(G), \{v\}\}$, for $v \in V(G)$) towards their side~$V(G)$.} in $G$, since $\N$ is not a union of three subsets not in~$\beta$. 

More generally, a $k$-tangle is \emph{principal} if for every set $X$ of fewer than $k$ vertices it orients a separation of the form $\{V(G) \setminus V(C), V(C) \cup X\}$, with~$C$ a component of $G - X$, towards the side $V(C) \cup X$. 
As in our example, an infinite graph $G$ contains a non-principal $k$-tangle if there is a set~$X$ of fewer than $k$ vertices of $G$ whose deletion separates $G$ into infinitely many components; and every such tangle, one for each non-principal ultrafilter on the set of components of $G-X$, orients all the separations of the form $\{V(C) \cup X, V(G)\setminus V(C)\}$, for components~$C$ of $G-X$, towards their side $V(G)\setminus V(C)$~\cite{EndsAndTangles}.\footnote{Note that such tangles cannot exist in connected locally finite graphs, where deleting finitely many vertices never leaves infinitely many components.} As we have seen, such non-principal tangles do not force a graph to have large tree-width, and hence give rise to counterexamples to \cref{main:TTDFiniteGraphs,main:TTDLocFinGraphs} for arbitrary infinite graphs. 
Hence, for graphs that are not locally finite, we shall have to adjust \cref{itm:TTDLocFin:Tangle} again, to ban non-principal tangles.

We shall have to adjust \cref{itm:TTDLocFin:Tangle} in another way too. A vertex $v$ of $G$ \emph{dominates} an end $\eps$ of $G$ if no finite set of vertices other than $v$ separates $v$ from a ray in $\eps$. 
The \emph{combined degree} of an end is the sum of its degree and the number of its dominating vertices.\footnote{Note that in locally finite graphs a vertex cannot dominate an end, so the combined degree of an end is simply its degree.} 
Similarly to ends of large degree, also ends of large combined degree force the tree-width up: It is not difficult to see that every graph with an end of combined degree~$k$ has tree-width at least $k$. Hence, we need to adjust \cref{itm:TTDLocFin:Tangle} as follows:

\begin{mainresult} \label{main:TTDCountableGraphs}
    For every countable graph $G$ and $k \in \N$, exactly one of the following assertions holds:
    \begin{enumerate}[label=\rm{(\roman*)}]
    \item\label{itm:TTDCountable:Tangle} There exists a principal $k$-tangle in $G$ that is not induced by an end of combined degree $<k$.
    \item\label{itm:TTDCountable:Tree} There exists a weakly exhaustive $S_k$-tree over $\cT^*$.
\end{enumerate}
\end{mainresult}

Note that we restricted the graphs in \cref{main:TTDCountableGraphs} to those that are countable. We did so for a reason: There is no uncountable graph $G$ that has a weakly exhaustive $S_k$-tree over $\cT^*$ for any $k \in \N$. Indeed, let $(T, \cV)$ be the \td\ induced by any weakly exhaustive $S_k$-tree over $\cT^*$. Then the definition of $\cT^*$ ensures that the bags $V_t$ of $(T, \cV)$ all have size $\leq 3k-3$, and that the tree $T$ has maximum degree at most $3$, and hence is countable. But then $G$ is countable.

However, there are uncountable graphs, e.g.\ stars with uncountably many leaves, that have no $3$-tangles as in \cref{itm:TTDCountable:Tangle} of \cref{main:TTDCountableGraphs}, and that even have tree-width $1$.
We could now try to update \cref{itm:TTDCountable:Tangle} again, so that our duality theorem always outputs \cref{itm:TTDCountable:Tangle} if the graph is uncountable; but
then \cref{itm:TTDCountable:Tangle} would no longer capture high local cohesion in a graph, which remains our aim. So we need to adjust~\cref{itm:TTDCountable:Tree}.

As indicated earlier, the definition of $S_k$-trees over $\cT^*$ is too restrictive to capture uncountable graphs of small tree-width, as those $S_k$-trees have maximum degree at most $3$. Hence, we will allow the $S_k$-tree in~\cref{itm:TTDInf:Tree} to have infinite-degree nodes. For this, we allow $S_k$-trees over $\cT^* \cup \cU^\infty_k$ rather than just~$\cT^*$, where $\cU^\infty_k$ is the set of all infinite stars of separations of order $<k$ whose interior has size~$<k$.\footnote{Formally, $\cU_k^\infty := \{\sigma = \{(A_i, B_i) : i \in I\} \subseteq \vS_k : \sigma \text{ is a star, } |\bigcap_{i \in I} B_i| < k \text{ and } |\sigma| = \infty\}$; see \cref{sec:Prelims} for details.} 
Now a \td\ induced by a weakly exhaustive $S_k$-tree over $\cT^* \cup \cU^\infty_k$ may have nodes $t$ of infinite degree as long as their bags $V_t$ have size less than $k$.
This modification of \cref{itm:TTDCountable:Tree} will be in line with our aim that \cref{itm:TTDInf:Tree} describes graphs that have no large highly cohesive substructures: such graphs may have non-principal tangles, and these can now happily live in the $S_k$-tree in its nodes of infinite degree.

Our tangle-tree duality theorem for arbitrary graphs now reads as follows (see \cref{fig:Intro:Infinite}):

\begin{mainresult}\label{main:TTDInfGraphs}
For every graph $G$ and $k \in \N$, exactly one of the following assertions holds:
\begin{enumerate}[label=\rm{(\roman*)}]
    \item\label{itm:TTDInf:Tangle} There exists a principal $k$-tangle in $G$ that is not induced by an end of combined degree $<k$.
    \item\label{itm:TTDInf:Tree} There exists a weakly exhaustive $S_k$-tree over $\cT^* \cup \cU_k^\infty$.
\end{enumerate}
\end{mainresult}

\noindent We remark that \cref{main:TTDLocFinGraphs,main:TTDCountableGraphs} are simple applications of \cref{main:TTDInfGraphs} (see \cref{sec:TTDInfGraphs} for details).
In particular, \cref{main:TTDInfGraphs} contains \cref{main:TTDFiniteGraphs} as a special case. Indeed, in finite graphs all tangles are principal, so a finite graph satisfies \cref{itm:TTDInf:Tangle} of \cref{main:TTDInfGraphs} if and only if it satisfies~\cref{itm:TTDFiniteGraphs:Tangle} of \cref{main:TTDFiniteGraphs}. Moreover, in finite graphs the set $\cU^\infty_k$ is empty, and every finite $S_k$-tree over~$\cT^*$ is weakly exhaustive. So a finite graph satisfies \cref{itm:TTDInf:Tree} of \cref{main:TTDInfGraphs} if and only if it satisfies \cref{itm:TTDFiniteGraphs:Tree} of \cref{main:TTDFiniteGraphs}. 

Moreover, similarly to \cref{main:TTDFiniteGraphs},
every graph with a tangle as in \cref{itm:TTDInf:Tangle} has tree-width at least $k-1$, while every graph with a tree as in \cref{itm:TTDInf:Tree} has tree-width at most $\lfloor 3k/2\rfloor -1$.
\medskip

\begin{figure}[ht]
    \centering
    \includegraphics[width=0.6\linewidth]{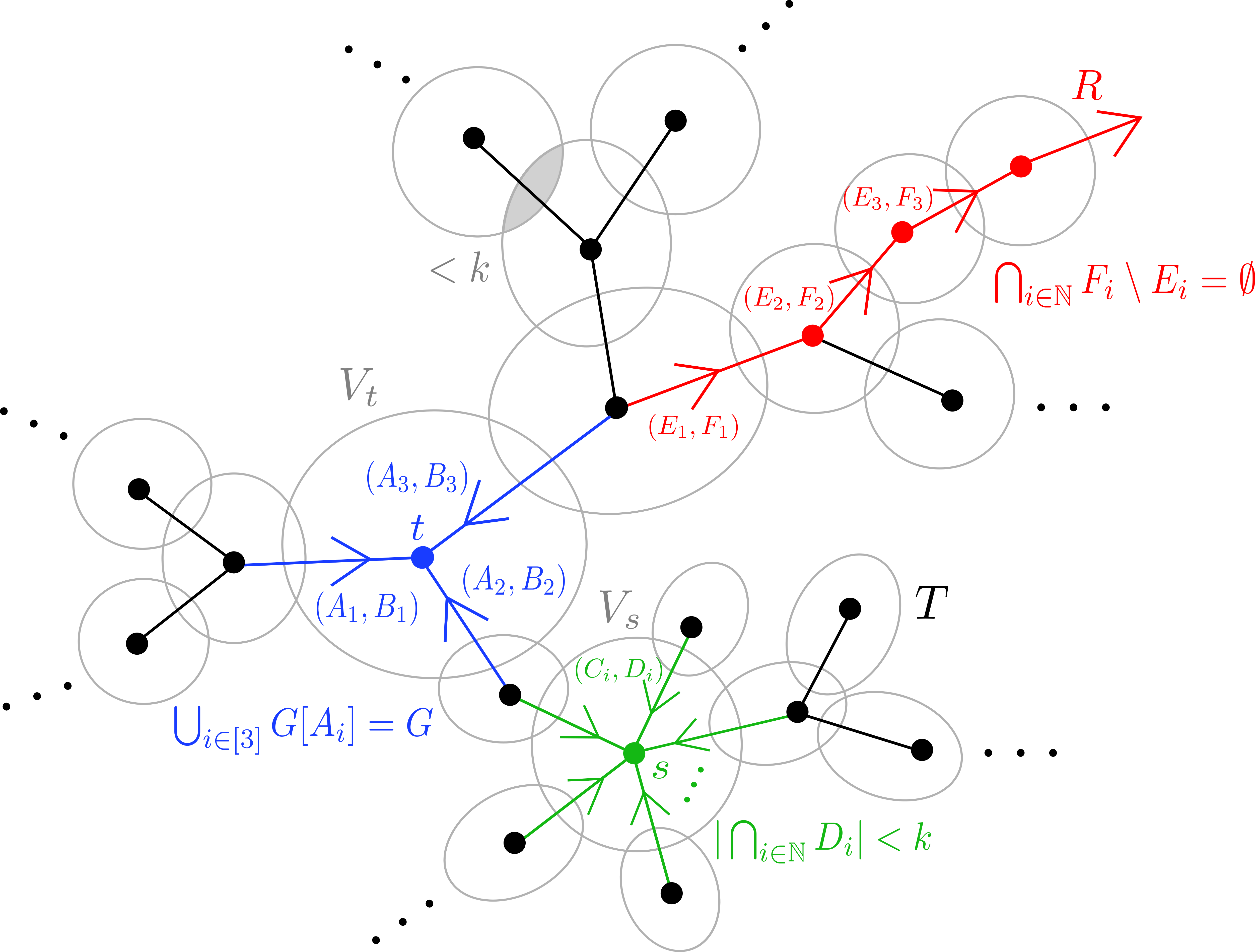}
    \caption{In addition to nodes of degree $\leq 3$, an $S_k$-tree $(T, \alpha)$ over $\cT^* \cup \cU^\infty_k$ can also have nodes $s$ of infinite degree, but then the intersection $\bigcap_{i \in \N} D_i$ of the $D_i$-sides of the separations $(C_i,D_i) := \alpha(s's)$ for $s' \in N_T(s)$ must have size $<k$. \\ 
    If $(T, \alpha)$ is weakly exhaustive, i.e.\ $\bigcap_{i \in \N} F_i \setminus E_i$ is empty for all rays $R = v_1v_2\dots$ in $T$ with $\alpha(v_iv_{i+1}) = (E_i,F_i)$, then $(T, \alpha)$ induces a \td\ (of adhesion $<k$ and width $\leq 3k-4$).} 
    \label{fig:Intro:Infinite}
\end{figure}

Diestel and Oum \cite{TangleTreeAbstract} generalized \cref{main:TTDFiniteGraphs} to so-called `$\cF$-tangles'.
The `standard' $k$-tangles are orientations of the separations of a graph $G$ of order $<k$ that \emph{avoid} the set $\cT^* \subseteq 2^{\vS_k}$: a $k$-tangle does not contain an element of $\cT^*$ as a subset. This set $\cT^*$ can be replaced by more general sets of separations, leading to the more general notion of $\cF$-tangles.

We will in fact prove \cref{main:TTDInfGraphs} more generally for $\cF$-tangles (see \cref{main:TTDInfGraphsFTangles} in \cref{sec:TTDInfGraphs} for the precise statement). As an application, we obtain the following exact characterization of graphs that have tree-width $k \in \N$, which generalizes a result of Diestel and Oum \cite{TangleTreeGraphsMatroids}:

\begin{mainresult} \label{main:TangleBrambleTreeDualityInfGraphs}
    The following assertions are equivalent for all graphs $G$ and $k \in \N$:
    \begin{enumerate}[label=\rm{(\roman*)}]
        \item \label{itm:TBTDuality:Tangle} $G$ has a $\cU_k$-tangle of order $k$ that is not induced by an end of combined degree $<k$.
        \item \label{itm:TBTDuality:Bramble} $G$ has a finite bramble of order at least $k$.
        \item \label{itm:TBTDuality:Sktree} $G$ has no weakly exhaustive $S_k$-tree over $\cU_k$.
        \item \label{itm:TBTDuality:TW} $G$ has tree-width at least $k-1$.
    \end{enumerate}
\end{mainresult}

\noindent (See \cref{sec:Prelims} for definitions.) The equivalence of \cref{itm:TBTDuality:Bramble} and \cref{itm:TBTDuality:TW} yields a generalization of the `bramble-treewidth duality theorem' of Seymour and Thomas \cite{ST1993GraphSearching} (see \cref{thm:BrambleTreeDualityInfGraphs} in \cref{sec:BrambleTWDuality}), which also includes their finite version as a corollary (without using it in the proof).
\medskip

The other major theorem about tangles which Robertson and Seymour \cite{GMX} proved is the \emph{tree-of-tangles theorem}. 
Recall that a separation $\{A,B\}$ of a graph $G$ \emph{distinguishes} two tangles in $G$ if they orient $\{A,B\}$ differently. It distinguishes them \emph{efficiently} if they are not distinguished by any separation of smaller order.

The tree-of-tangles theorem for fixed $k \in \N$ asserts that every finite graph $G$ has a \td\ $(T, \cV)$ which efficiently distinguishes all its $k$-tangles: for every pair $\tau, \tau'$ of $k$-tangles in $G$, there is an edge~$e$ of $T$ such that the separation induced by $e$ distinguishes $\tau$ and $\tau'$ efficiently. 

Following \cite{ToTinfOrder}, we call two $k$-tangles \emph{combinatorially} distinguishable if there is a finite set $X \subseteq V(G)$ and a component $C$ of $G-X$ such that $\{V(C) \cup X, V(G-C)\}$ distinguishes them. In particular, if two $k$-tangles are combinatorially indistinguishable, then they are both non-principal.
For instance, in the example mentioned above, where $G$ is the edgeless graph on vertex set $\N$, no two $1$-tangles in $G$ are combinatorially distinguishable.

We show that if a graph $G$ has no $k$-tangle as in \cref{itm:TTDInf:Tangle} of \cref{main:TTDInfGraphs}, then it has an $S_k$-tree as in \cref{itm:TTDInf:Tree} which additionally distinguishes all the combinatorially distinguishable $k$-tangles that $G$ may have: those that are not of the form as in \cref{itm:TTDInf:Tangle}:

\begin{mainresult} \label{main:TTDPlusToT}
    Let $G$ be a graph and $k \in \N$. Suppose that all principal $k$-tangles in $G$ are induced by ends of combined degree $< k$. Then $G$ has a weakly exhaustive $S_k$-tree $(T, \alpha)$ over $\cT^* \cup \cU^\infty_k$ such that
    \begin{enumerate}[label=\rm{(\roman*)}]
        \item \label{itm:TTDPlusToT:Ends} every end of $G$ lives in an end of $T$, and conversely every end of $T$ is home to an end of $G$, 
        \item \label{itm:TTDPlusToT:Ultrafilter} every non-principal $k$-tangle lives at a node $t$ of $T$ with $\{\alpha(t',t) : t' \in N_T(t)\} \in \cU^\infty_k$, and
        \item \label{itm:TTDPlusToT:ToT} for every pair $\tau, \tau'$ of combinatorially distinguishable $k$-tangles in $G$ there is an edge $e$ of $T$ such that $\alpha(e)$ distinguishes $\tau$ and $\tau'$ efficiently.
    \end{enumerate}
\end{mainresult}

\noindent In particular, \cref{itm:TTDPlusToT:ToT} ensures that no two ends of $G$ live in the same end of $T$, and that no two combinatorially distinguishable non-principal $k$-tangles live at the same node of $T$. 
We remark that it is not possible to strengthen \cref{itm:TTDPlusToT:ToT} so that all $k$-tangles are efficiently distinguished by a separation of the form $\alpha(e)$ for an edge $e$ of $T$ \cite{ToTinfOrder}*{Corollary~3.4}.

We will obtain \cref{main:TTDPlusToT} as a corollary of a more general theorem (see \cref{thm:RefToTsInfGraphs} in \cref{sec:RefiningEssStars}) that yields a \td\ with similar properties as the $S_k$-tree in \cref{main:TTDPlusToT} even if $G$ has other $k$-tangles than those allowed in the premise of \cref{main:TTDPlusToT}.
\medskip

This paper is organised as follows. We first give a brief introduction to infinite graphs and their tangles in \cref{sec:Prelims,sec:EndsAndCVs}. 
In \cref{sec:RefiningInessStars}, we first sketch the proof of \cref{main:TTDInfGraphs} briefly, and then prove \cref{lem:ReflemForInfGraphsInessStars}, which is one of the two main ingredients to the proof of \cref{main:TTDInfGraphs}. In \cref{sec:TTDInfGraphs}, we prove \cref{main:TTDInfGraphs}, and then derive \cref{main:TTDLocFinGraphs,main:TTDCountableGraphs} from it.
In \cref{sec:BrambleTWDuality} we deduce \cref{main:TangleBrambleTreeDualityInfGraphs}.
In \cref{sec:RefiningEssStars}, we use the tools developed in \cref{sec:RefiningInessStars} to show a `refined' version of the \emph{tree-of-tangles theorem}, which generalizes a result of \cite{SARefiningEssParts} to infinite graphs, and which contains \cref{main:TTDPlusToT} as a special case.

\section{Preliminaries}\label{sec:Prelims}

We mainly follow the notions from \cite{DiestelBook16noEE}. In what follows, we recap some important definitions which we need later.
All graphs in this paper may be infinite unless otherwise stated. 

\subsection{Infinite graphs}\label{subsec:InfiniteGraphs}

A \defn{ray} in a graph is a one-way infinite path. A graph is \defn{rayless} if it contains no ray. A graph is \defn{tough} if deleting finitely many vertices never leaves infinitely many components. 

The following theorem was first proved by Polat \cite{PolatEME1}; see \cite{LinkedTDInfGraphs}*{Theorem~2.5} for a short proof.

\begin{THM} 
\label{prop:RaylessToughGraphsAreFinite}
    Every tough, rayless graph is finite.
\end{THM}

\subsection{Separations}

Let $G$ be any graph. A \defn{separation} of $G$ is a set $\{A, B\}$ of subsets of $V(G)$ such that $A \cup B = V(G)$ and there is no edge in $G$ between $A\setminus B$ and $B \setminus A$. A separation $\{A,B\}$ of $G$ is \defn{proper} if neither $A$ nor $B$ equals $V(G)$.  
The \defn{order} $|\{A,B\}|$ of a separation $\{A, B\}$ is the size $\abs{A \cap B}$ of its \defn{separator} $A \cap B$.
For every $k \in \N \cup \{\aleph_0\}$, we define \defn{$S_k$} to be the set of all separations of $G$ of order $< k$. 

The \defn{orientations} of a separation~$\{A, B\}$ are the \defn{oriented separations} $(A,B)$ and $(B,A)$. We refer to~$A\, (\setminus B)$ as the \defn{(strict) small side} of $(A,B)$ and to $B\, (\setminus A)$ as the \defn{(strict) big side} of $(A,B)$.
Given a set~$S$ of separations of $G$, we write $\defnm{\vS} := \{(A, B) : \{A, B\} \in S\}$ for the set of all their orientations. We will use terms defined for unoriented separations also for oriented ones and vice versa if that is possible without causing ambiguities. Moreover, if the context is clear, we will simply refer to both oriented and unoriented separations as `separations'.
If we do not need to know about the sides of a separation, we sometimes denote separations with~$s$, and their orientations with $\vs, \sv$. Note that there are no default orientations: Once we denoted one orientation by $\vs$ the other one will be $\sv$, and vice versa.

A separation $(A,B)$ of $G$ is \defn{tight on the small side} (\defn{tight on the big side}) if the neighbourhood in $G$ of some component of $G[A\setminus B]$ ($G[B\setminus A]$) equals $A \cap B$.
Moreover, $\{A,B\}$ is \defn{tight} if $(A,B)$ is tight on the left and right side. 

The oriented separations of a graph $G$ are partially ordered by $(A, B) \leq (C,D)$ if $A \subseteq C$ and $B \supseteq D$.
A separation $\vr$ of $G$ is \defn{trivial} in $\vS \subseteq \vS_{\aleph_0}$ if there exists $s \in S$ such that $\vr < \vs$ as well as $\vr < \sv$. The separations of $G$ that are trivial in some $S\subseteq\vS_{\aleph_0}$ are those of the form $\vr =(X,V(G)) \in \vS$ for which there exists $s = \{A,B\} \in S_{\aleph_0}\setminus \{r\}$ with $X \subseteq A \cap B$.

An infinite increasing sequence $((A_i, B_i))_{i \in \N}$ of separations of a graph $G$ is \defn{weakly exhaustive} if the intersection of their strict big sides is empty, i.e.\ if $\bigcap_{i \in \N} (B_i \setminus A_i) = \emptyset$.

A set $\sigma \subseteq \vS_{\aleph_0} \setminus \{(V,V)\}$ of separations is called a \defn{star} if for any $(A,B), (C,D) \in \sigma$ it holds that $(A,B) \leq (D,C)$. 
The \defn{interior} of a star $\sigma \subseteq \vS_{\aleph_0}$ is the intersection $\interior(\sigma) := \bigcap_{(A, B) \in \sigma} B$, and the \defn{torso} \defn{of $\sigma$}, denoted by $\torso(\sigma)$, is the graph that is obtained from $G[\interior(\sigma)]$ by adding an edge $\{v,u\}$ whenever $v \neq u \in \interior(\sigma)$ lie together in some separator of a separation in $\sigma$.

The partial order on $\vS_{\aleph_0}$ also relates the proper stars in $\vS_{\aleph_0}$: If $\sigma, \tau \subseteq \vS_{\aleph_0}$ are stars of proper separations, then $\sigma \leq \tau$ if and only if for every $\vs \in \sigma$ there exists some $\vr \in \tau$ such that $\vs \leq \vr$. Note that this relation is again a partial order \cite{TreeSets}.

Two separations $\{A,B\}$ and $\{C,D\}$ of $G$ are \defn{nested} if they have orientations which can be compared; otherwise they \defn{cross}.

For any pair of separations $(A, B)$ and $(C,D)$ also their \defn{infimum} $(A, B) \wedge (C,D) := (A \cap C, B \cup D)$ and their \defn{supremum} $(A, B) \vee (C,D) := (A \cup C, B \cap D)$ are separations of $G$; we call $\{A \cap C, B \cup D\}$, $\{A \cup C, B \cap D\}$, $\{B \cap C, A \cup D\}$ and $\{B \cup C, A \cap D\}$ the \defn{corner separations} of $\{A, B\}$ and $\{C,D\}$.

\begin{LEM}{\cite{DiestelBook16noEE}*{Lemma 12.5.5}}\label{lem:Fishlemma}
    Let $r, s$ be two crossing separations of a graph $G$. Every separation of $G$ that is nested with both $r$ and $s$ is also nested with all four corner separations of $r$ and $s$.
\end{LEM}

Moreover, it is easy to check that if two separations $\{A,B\}, \{C,D\}$ of a graph cross, then 
\[
\abs{(A \cap C) \cap (B \cup D)} + \abs{(A \cup C) \cap (B \cap D)} = \abs{A \cap B} + \abs{C \cap D}.
\]
Here `$\leq$' is the important part, which is called \defn{submodularity}.

\subsection{Profiles and tangles}

An \defn{orientation} of a set $S \subseteq S_{\aleph_0}$ is a set $O \subseteq \vS$ which contains, for every $\{A,B\} \in S$, exactly one of its orientations $(A,B)$ and $(B,A)$. It is \defn{consistent} if it does not contain both $(B,A)$ and $(C,D)$ whenever $(A,B) < (C,D)$ for distinct $\{A,B\}, \{C,D\} \in S$. An orientation is \defn{regular} if it does not contain $(V(G), A)$ for any subset $A \subseteq V(G)$.

An orientation $O$ of $S$ \defn{lives} in a star $\sigma \subseteq \vS$ (or equivalently $\sigma$ \defn{is home} to $O$) if $\sigma \subseteq O$.
If $\cO$ is a set of consistent orientations of $S$, we call a star $\sigma \subseteq \vS$ \defn{essential (for $\cO$)} if some~$O \in \cO$ lives in $\sigma$. Otherwise $\sigma$ is called \defn{inessential (for $\cO$)}. 

A separation $\{A,B\} \in S$ \defn{distinguishes} two orientations of $S$ if they orient~$\{A,B\}$ differently. $\{A,B\}$ distinguishes them \defn{efficiently} if they are not distinguished by any separation of smaller order. A set of separations $N \subseteq S$ \defn{(efficiently) distinguishes} a set $\cO$ of consistent orientations of $S$ if any two distinct orientations in $\cO$ are (efficiently) distinguished by some separation in $N$.
\smallskip

Let $S \subseteq S_{\aleph_0}$, and let $\cF$ be a set of subsets of~$\vS_{\aleph_0}$. We call an orientation~$O$ of~$S$ an \defn{$\cF$-tangle of~$S$} if~$O$ is consistent and \defn{avoids}~$\cF$, i.e.\ if~$O$ does not contain any element of $\cF$ as a subset. 

For some $k \in \N \cup \{\aleph_0\}$, a \defn{$k$-tangle} in $G$ (or \defn{tangle of order $k$}) is an orientation $\tau$ of $S_k$ with the property that no three small sides  of separations in $\tau$ cover $G$, i.e.\ $G[A_1] \cup G[A_2] \cup G[A_3] \neq G$ for all $(A_1, B_1), (A_2, B_2), (A_3, B_3) \in \tau$. Thus, every $k$-tangle avoids 
\[
\defnm{\cT_k} := \big\{\{(A_1, B_1), (A_2, B_2), (A_3, B_3)\} \subseteq \vS_k : G[A_1] \cup G[A_2] \cup G[A_3] = G\big\},
\]
and hence the $k$-tangles in $G$ are precisely the $\cT_k$-tangles of $S_k$.
We denote by~\defn{$\cT^*_k$} the set of all stars in~$\cT_k$, and we abbreviate $\cT_{\aleph_0}, \cT^*_{\aleph_0}$ by $\defnm{\cT}, \defnm{\cT^*}$, respectively. 
\smallskip

For some $k \in \N$, a consistent orientation $P$ of $S_k$ is a \defn{$k$-profile in $G$} if for all $(A,B), (C,D) \in P$, their oriented corner $(B \cap D, A \cup C)$ does not lie in $P$ (see \cref{subfig:Profile1}). Note that this can be satisfied in two ways: either by $(A \cup C, B \cap D) \in P$, or because $\{A\cup C, B \cap D\}$ has order $\geq k$, and is hence not oriented by $P$ at all.
Note that every $k$-tangle is a $k$-profile.

\begin{figure}[ht]
    \centering
    \begin{subfigure}{0.4\linewidth}
        \centering
        \includegraphics[width=0.6\linewidth]{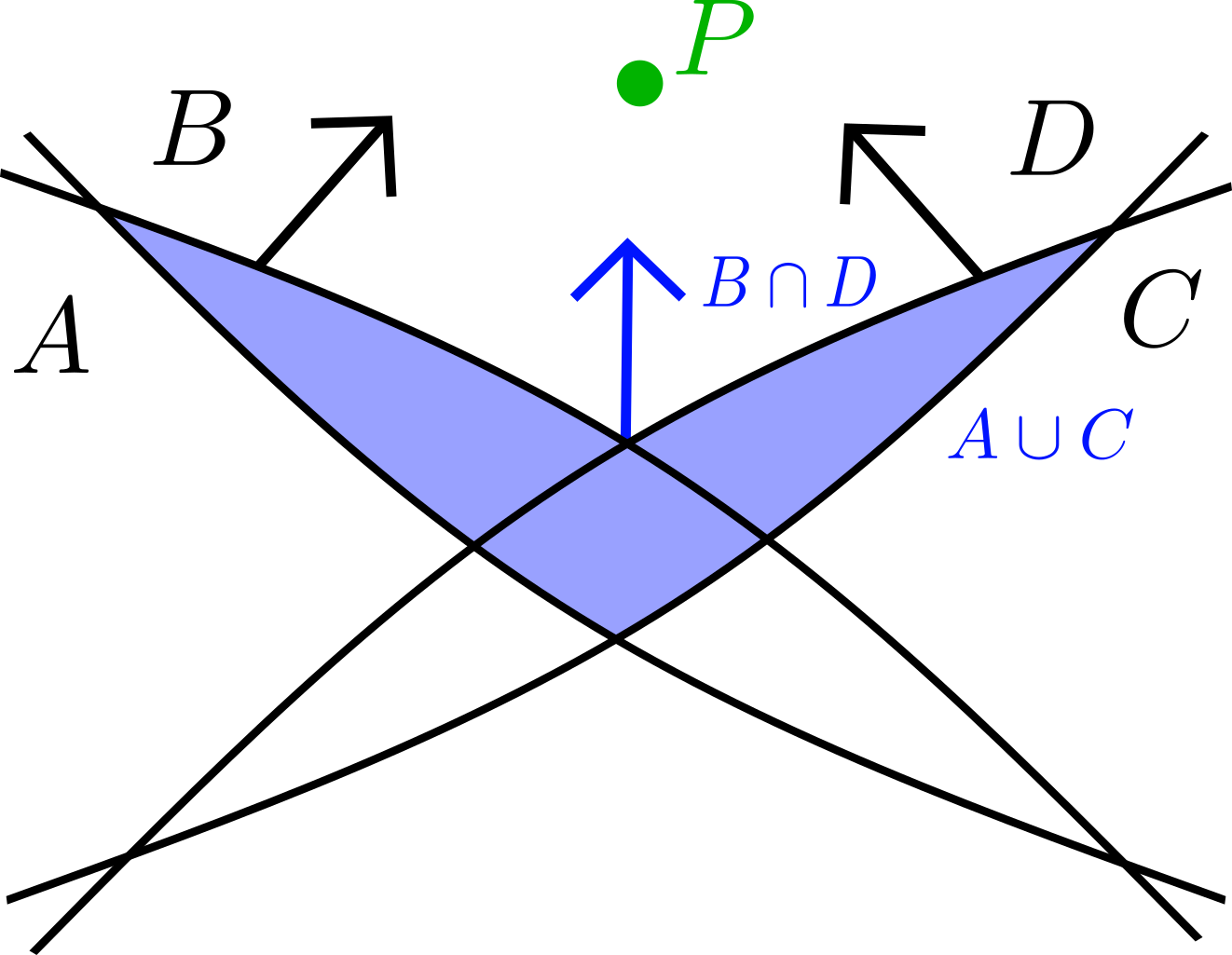}
        \caption{}
        \label{subfig:Profile1}
    \end{subfigure}
    \begin{subfigure}{0.4\linewidth}
        \centering
        \includegraphics[width=0.6\linewidth]{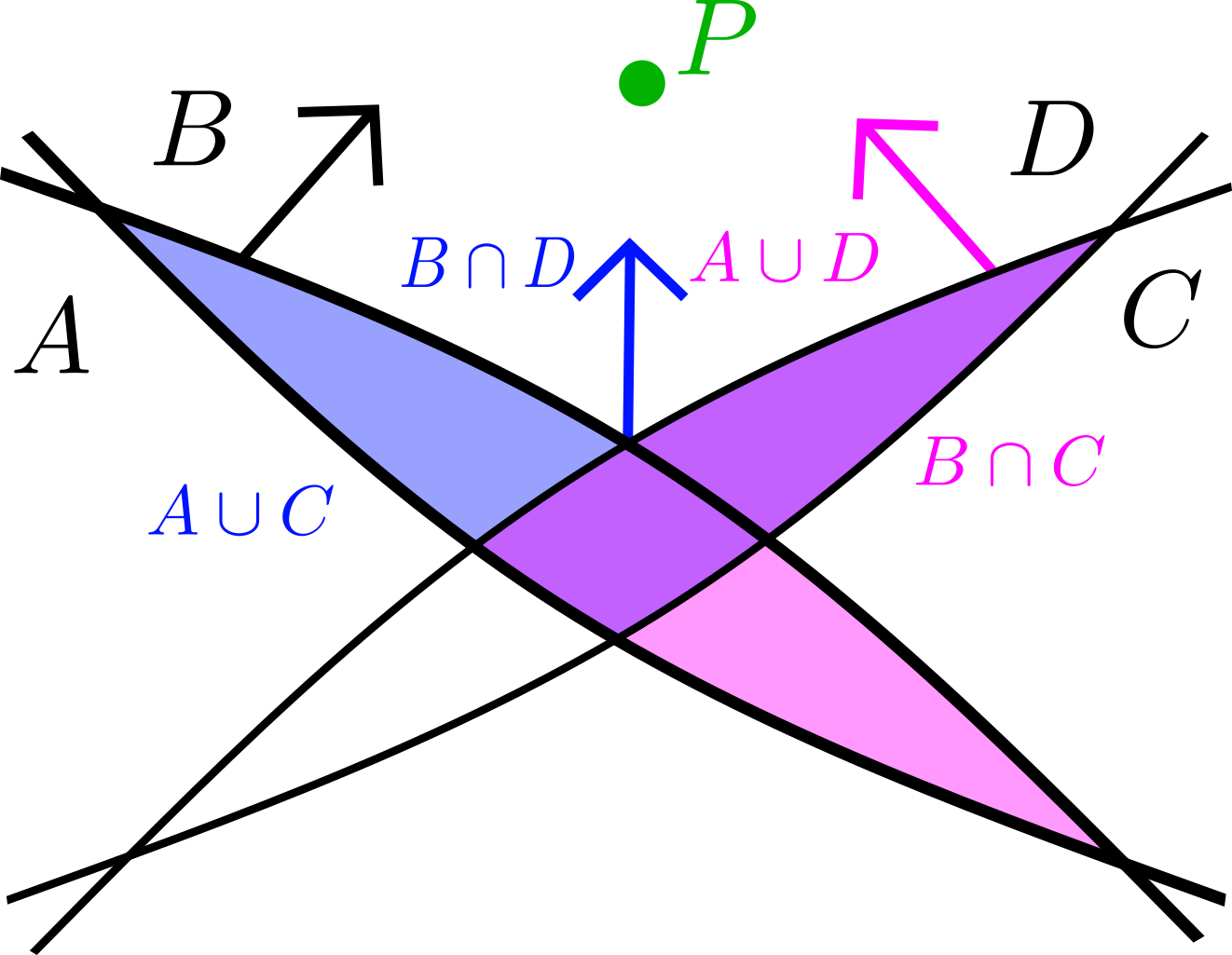}
        \caption{}
        \label{subfig:Profile2}
    \end{subfigure}
    \caption{Depicted are separations $\{A,B\}, \{C,D\} \in S_k$, their corners $\{A \cup C, B \cap D\}$ (in blue) and $\{A \cup D, B \cap C\}$ (right, in pink), and an orientation $P$ of $S_k$. In the left figure, if $P$ is a profile, then $(A \cup C, B \cap D) \in P$. In the right figure, if $P$ avoids $\cP_k$, then again $(A \cup C, B \cap D) \in P$ because the star $\sigma := \{(A,B), (B \cap C, A \cup D), (B \cap D, A \cup C)$ is in $\cP_k$. If the coloured area in the right figure has size $<k$, then $\sigma$ is contained in $\cP'_k$.}
    \label{fig:Profile}
\end{figure}

One can represent $k$-profiles as $\cF$-tangles for a set $\cF$ of stars as follows. Let $(A,B), (C,D) \in S_k$, and assume that their supremum $(A \cup C, B \cap D)$ has order $<k$. Additionally, by submodularity, at least one of their `opposite' corners $(A \cap D, B \cup C)$ and $(B \cap C, A \cup D)$ has order $<k$. Hence, this separation together with $(B \cap D, A \cup C)$ and one of $(A,B)$ or $(C,D)$ forms a star in $\vS_k$ (see \cref{subfig:Profile2}). An orientation of $S_k$ is a $k$-profile in $G$ if and only if it contains no such star, which is precisely the case if it avoids the set
\[
\defnm{\cP_k} := \big\{\{(A,B), (B \cap C, A \cup D), (B \cap D, A \cup C)\} \subseteq \vS_k : (A,B), (C,D) \in \vS_{k}\big\}
\]
of all these stars. Hence, the $k$-profiles in $G$ are precisely the $\cP_k$-tangles of $S_k$ (cf.\ \cite{ProfileDuality}*{Lemma 11}).

We also need the following subset of $\cP_k$, which contains precisely the stars in $\cP_k$ whose interior has size~$<k$ (i.e.\ the coloured area in \cref{subfig:Profile2} has size $<k$):
\[
\defnm{\cP'_k} := \big\{\sigma \in \cP_k : |\interior(\sigma)| < k\big\}
\]
Tangles and profiles of unspecified order are referred to as \defn{tangles / profiles in $G$}.

\begin{lemma}{\cite{InfiniteSplinters}*{Lemma~6.1}} \label{lem:EffYieldsTight}
    Let $P, P'$ be two regular profiles in a graph $G$. If $\{A, B\}$ is a separation of finite order that efficiently distinguishes $P$ and $P'$, then $\{A,B\}$ is tight.
\end{lemma}

An orientation of some $S_k$ is \defn{principal} if it contains for every set $X$ of fewer than $k$ vertices a separation of the form $(V(G-K), V(K) \cup X)$ where $K$ is a component of $G - X$. It is easy to check that the regular, principal $\cP'_k$-tangles of $S_k$ avoid 
\[
\defnm{\cU_k} := \big\{\sigma \subseteq \vS_k : \sigma \text{ is a star and } |\interior(\sigma)| < k\big\}. 
\]
Write $\defnm{\cU^\infty_k} := \{\sigma \in \cU_k : |\sigma| = \infty\}$ for the set of all infinite stars in $\cU_k$.

\begin{LEM} \label{lem:NonPrincipalPkTanglesAreInducedByUFTangles}
    Let $\tau$ be a regular, non-principal $\cP'_k$-tangle of order $k \in \N$ in a graph $G$, and let $\sigma \subseteq \vS_k$ be a finite star with finite interior. Then $\sigma \not\subseteq \tau$.
\end{LEM}

\begin{proof}
    Suppose for a contradiction that $\sigma \subseteq \tau$. 
    Since $\tau$ is non-principal, there is some $X \subseteq V(G)$ such that $(V(K) \cup X, V(G-K)) \in \tau$ for all components~$K$ of $G-X$; note that $|X| < k$ as $\tau$ has order $k$.
    By the consistency of $\tau$, we have $\big(V\big(\bigcup \cK_{(A,B)}\big) \cup X, V\big(G-\bigcup \cK_{(A,B)}\big)\big) \in \tau$ for all $(A,B) \in \sigma$ where~$\cK_{(A,B)}$ is the set of all components of $G-X$ that are contained in $G[A\setminus B]$.
    Since $|\sigma|$ is finite, inductively applying that $\tau$ is a $\cP'_k$-tangle yields that $\big(V\big(\bigcup \cK\big) \cup X, V\big(G-\bigcup\cK\big)\big) \in \tau$ where $\cK := \bigcup_{(A,B) \in \sigma} \cK_{(A,B)}$.
    As $\interior(\sigma)$ is finite, at most finitely many components of $G-X$ are not in $\cK$, so the same inductive argument yields that $(V(G), X) \in \tau$, which contradicts that $\tau$ is regular.
\end{proof}

\subsection{Nice sets of stars}

Let $S \subseteq S_{\aleph_0}$, $\vr \in \vS_{\aleph_0}$, and set $S_{\geq \vr} := \{x \in S : \vx \geq \vr \text{ or } \xv \geq \vr\}$. 
Further, let $\vs \in \vS$. We say that \defn{$\vs$ emulates $\vr$ in $\vS$} if $\vs \geq \vr$ and for every $\vx \in \vS$ with $\vx \geq \vr$ it holds that $\vs \vee \vx \in \vS$.
Given a set $\cF$ of stars in $\vS_{\aleph_0}$, we say that \defn{$\vs$ emulates $\vr$ in $\vS$ for $\cF$} if $\vs$ emulates $\vr$ in $\vS$ and for every star $\sigma \in \cF$ with $\sigma \subseteq \vS_{\geq \vr} \setminus \{\rv\}$ that contains an element $\vt \geq \vr$ it holds that~$\{\vs \vee \vt\} \cup \{\sv \wedge \vt' : \vt' \in \sigma\setminus \{\vt\}\}$ is again a star in $\cF$.

A set $\cF$ of stars in $\vS_{\aleph_0}$ is \defn{closed under shifting in $\vS$} if whenever $\vs \in \vS$ emulates some non-trivial $\vr \in \vS \setminus \{(V(G),V(G))\}$ in $S$ with $\{\rv\} \notin \cF$, then it also emulates $\vr$ in $\vS$ for $\cF$.
Further, if $S = S_k$ for some $k \in \N$, then $\cF$ is \defn{strongly closed under shifting in $\vS_k$} if whenever $\vs \in \vS_k$ emulates some $\vr \in \vS_{2k-1}$ in $S_k$, then it also emulates $\vr$ in $\vS_k$ for $\cF$.

\begin{PROP}\label{prop:ShiftingForbiddenStars}
    Let $G$ be a graph, $k \in \N$, and let $\cF$ be a set of stars in $\vS_k$ that is strongly closed under shifting. Further, let~$P$ be a $k$-profile in $G$, let $\sv \in P$, and let $\{\vt\} \cup \{\vr_i : i \in I\} \in \cF$ be such that $\tv \in P$. Suppose further that $\vs \vee \vx \in \vS_k$ for all $\xv \in P$. Then $\{\vs \vee \vt\} \cup \{\sv \wedge \vr_i : i \in I\}$ is a star in $\cF$.
\end{PROP}

\begin{proof}
    We show that $\vs$ emulates $\vs \wedge \vt$ in $S_k$ from which the assertion follows as $\cF$ is strongly closed under shifting and~$\abs{\vs \wedge \vt} \leq \abs{s} + \abs{t} < 2k-1$.
    For this let $\vx \in \vS_k$ with $\vx \geq \vs \wedge \vt$ be given. 
    Then $\xv \in P$ by \cite{EberenzMaster}*{Theorem~1}.
    So by the assumptions on $s$,
    it follows that $\vs \vee \vx \in S_k$.
\end{proof}

A set $\cF$ of stars in $\vS_k$ is called \defn{nice} if $\cF$ is strongly closed under shifting in $\vS_k$, $\{(V(G), A)\} \in \cF$ for all subsets $A \subseteq V(G)$ of fewer than $k$ vertices, and $\cP'_k \subseteq \cF$.

\begin{LEM} \label{lem:UkStronglyClosedUnderShifting}
    For every $k > 0$, the sets $\cU_k$ and $\cT^*_k$ are nice.
\end{LEM}

\begin{proof}
    By definition, $\cP'_k \subseteq \cU_k \cap \cT^*_k$, and we have $\{(V(G), A)\} \in \cU_k \cap \cT^*_k$ for all sets~$A$ of fewer than~$k$ vertices. The proofs that $\cU_k$ and $\cT^*_k$ are strongly closed under shifting in $\vS_k$ are analogous to the proofs that they are closed under shifting (see the proofs of \cite{TangleTreeGraphsMatroids}*{Lemma~6.1 \& Theorem~4.1}).
\end{proof}

\subsection{Tree-decompositions and \emph{S}-trees} \label{subsec:STrees}

A \defn{tree-decomposition} of a graph $G$ is a pair $(T, \cV)$ of a tree~$T$ together with a family $\cV = (V_t)_{t \in V(T)}$ of subsets of $V(G)$ such that $\bigcup_{t \in T} G[V_t] = G$, and such that for every vertex~$v \in V(G)$, the graph $T[\{t\in T : v\in V_t\}]$ is connected. We call the sets~$V_t \in \cV$ the \defn{bags} and their induced subgraphs $G[V_t]$ the \defn{parts} of $(T, \cV)$. The sets $V_e := V_t \cap V_{t'}$ for edges $e = \{t, t'\} \in E(T)$ are the \defn{adhesion sets} of $(T, \cV)$.
We say that $(T, \cV)$ has \defn{adhesion} $< k \in \N$ if all its adhesion sets have size $<k$. Further, $(T, \cV)$ has \defn{width} $< k \in \N$ if all its bags have size $\leq k$. 
A graph $G$ has \defn{tree-width} $<k$ if it admits a \td\ of width $<k$.

In a \td\ $(T,\cV)$ of $G$ every (oriented) edge $\ve = (t_0,t_1)$ of $T$ induces a separation of $G$ as follows: For $i = 0,1$, write $T_i$ for the component of $T - e$ containing $t_i$. Then $(U_{t_0}, U_{t_1})$ is a separation of $G$ where $U_{t_i} := \bigcup_{s \in V(T_i)} V_s$ for $i = 0,1$ \cite{DiestelBook16noEE}*{Lemma~12.3.1}. We say that $(U_{t_0}, U_{t_1})$ and $\{U_{t_0}, U_{t_1}\}$ are \defn{induced} by $\ve$ and $e$, respectively.
It is easy to check that the separations induced by the edges of an oriented ray in $T$ form a weakly exhaustive increasing sequence.
Moreover, the set $\sigma_t := \{(U_{t'}, U_{t}) : (t', t) \in \vE(T)\}$ induced by the inwards oriented edges incident with a node $t$ of $T$ is a star, which we call the \defn{star associated with $t$}. The \defn{torso} of a bag $V_t$ is the graph $\torso(\sigma_t)$.

We say that a consistent orientation $O$ of $S_k$, for some $k \in \N \cup \{\aleph_0\}$, \defn{lives} at a node $t$ of $T$, or in the bag $V_t$, if~$\sigma_t \subseteq O$. 
Further, given a set $\cO$ of consistent orientations of $S_k$, we call a node $t$ of $T$ and its bag~$V_t$ \defn{essential (for $\cO$)} if there is an orientation in $\cO$ that lives at~$t$ and otherwise \defn{inessential (for~$\cO$)}. 
A \td~$(T, \cV)$ \defn{(efficiently) distinguishes} two profiles if there is an edge $\{t_0, t_1\} \in E(T)$ such that $\{U_{t_0}, U_{t_1}\}$ (efficiently) distinguishes them.

Let $S \subseteq S_{\aleph_0}$. An \defn{$S$-tree} is a pair $(T,\alpha)$ of a tree $T$ and a map $\alpha: \vE(T) \rightarrow \vS$ from the oriented edges~$\vE(T)$ of $T$ to $\vS$ such that $\alpha(\ev) = (B,A)$ whenever $\alpha(\ve) = (A,B)$ and such that $\alpha(\ve) \leq \alpha(\vf)$ whenever\footnote{Here, $\ve \leq \vf$ for two oriented edges $\ve = (s,t)$ and $\vf=(s',t')$ of a tree $T$ if the unique path from $e$ to $f$ in $T$ starts in $t$ and ends in $s'$.} $\ve \leq \vf$. 
If the tree $T$ is finite, then we call $(T, \alpha)$ a \defn{finite $S$-tree}. 
If $x \in V(T)$ is a leaf of $T$ and $t \in V(T)$ its unique neighbour, then we call $\alpha(x,t) \in \vS$ a \defn{leaf separation (of $T$)}.

For every node $t \in V(T)$, the set $\sigma_t := \{\alpha(t',t) : (t',t) \in \vE(T)\}$ is a star, which we call the \defn{star associated with $t$ (in $T$)}.
An $S$-tree $(T, \alpha)$ is \defn{over} a set $\cF \subseteq 2^{\vS}$ (of stars) if $\sigma_t \in \cF$ for every node $t \in V(T)$.
Further, we call $(T, \alpha)$ \defn{weakly exhaustive} if for every ray $R = r_1 r_2 \dots$ in $T$ the sequence $(\alpha(r_i,r_{i+1}))_{i \in \N}$ is weakly exhaustive.

\subsection{Tangle-tree duality}

For the proof of our main result, \cref{main:TTDInfGraphs}, we need the abstract version \cite{TangleTreeAbstract} of the finite tangle-tree duality theorem, \cref{main:TTDFiniteGraphs}.
In order to state that theorem formally, we need two further definitions.

Given a graph $G$ and a set $\cF$ of stars in $\vS_{\aleph_0}$, a set $S \subseteq S_{\aleph_0}$ is \defn{$\cF$-separable} if for every two non-trivial $\vr, \vr' \in \vS \setminus \{(V(G),V(G))\}$ with $\vr \leq \vr'$ and $\{\rv\}, \{\vr'\} \notin \cF$, there exists $\vs \in \vS$ such that $\vs$ and~$\sv$ emulate $\vr$ and $\rv'$, respectively, in $\vS$ for $\cF$.
Further, $\cF$ is \defn{standard} for $\vS$ if $\{\rv\} \in \cF$ for all $\vr \in \vS$ that are trivial in~$\vS$.

The following theorem is a variant of the tangle-tree duality theorem. It follows immediately by applying \cite{TangleTreeAbstract}*{Theorem~4.3} to the set $S$ below. We remark that while $S$ needs to be finite, $G$ may be infinite.

\begin{THM}
\label{thm:TTDinASS}
    Let $G$ be any graph, $S \subseteq S_{\aleph_0}$ finite, and let $\cF$ be a set of stars in~$\vS_{\aleph_0}$, standard for~$\vS$. If $\vS$ is $\cF$-separable, exactly one of the following assertions holds:
    \begin{enumerate}
        \item There exists an $\cF$-tangle of $S$. 
        \item There exists an $S$-tree over $\cF$.
    \end{enumerate}
\end{THM}

\section{Ends, critical vertex sets and tangles} \label{sec:EndsAndCVs}

\subsection{Ends}

An \defn{end} of a graph~$G$ is an equivalence class of rays in~$G$ where two rays are equivalent if they are joined by infinitely many disjoint paths in~$G$ or, equivalently, if for every finite set~$X \subseteq V(G)$ they have a tail in the same component of $G - X$.

It is easy to see that every end $\eps$ of $G$ \defn{induces} an $\aleph_0$-tangle, denoted by $\tau_\eps$, by orienting every finite-order separation of~$G$ to that side which contains a tail of some (equivalently every) ray in $\eps$ \cite{EndsAndTangles}*{\S1}.
We say that an orientation~$O$ of~$S_k$, for some $k \in \N$, is \defn{induced} by an end~$\eps$ of~$G$ if $\tau \subseteq \tau_\eps$. 

It is easy to check that a tangle $\tau_\eps$ induced by an end $\eps$ of $G$ cannot contain any stars with finite interior (cf.\ \cite{EndsAndTangles}*{Discussion preceding Lemma 1.6}). Hence, the following observation about ends follows easily.

\begin{PROP}\label{prop:EndsInduceFTangles}
    Let $G$ be any graph, $k \in \N$, and let $\cF$ be a set of stars in $\vS_k$ all of which have finite interior. Then every end of $G$ induces an $\cF$-tangle of $S_k$. \qed
\end{PROP}

We say that an end $\eps$ of a graph $G$ \defn{lives} in a star $\sigma \subseteq \vS_{\aleph_0}$ (or equivalently~$\sigma$ \defn{is home to} $\eps$) if $\sigma \subseteq \tau_\eps$. 

Let $(T, \cV)$ be a \td\ of $G$ whose adhesion sets are all finite, and let $\eps$ be an end of $G$.
Let $O \subseteq \vE(T)$ consist of those oriented edges $(t,s)$ of $T$ such that $(U_t, U_s) \in \tau_\eps$. Then the orientation $O$ of $E(T)$ points towards a node of $T$ or to an end of $T$. We say that $\eps$ \defn{lives} at that node or in that end, respectively.

A vertex $v$ of $G$ \defn{dominates} an end $\eps$ of $G$ if $v \in B$ for every finite-order separation $(A,B) \in \tau_\eps$. We denote by \defn{$\Dom(\eps)$} the set of vertices of $G$ that dominate $\eps$, and by \defn{$\dom(\eps)$} the cardinality of $\Dom(\eps)$.

The \defn{degree} of an end $\eps$, which we denote by $\defnm{\deg(\eps)} \in \N \cup \{\infty\}$, is the maximum\footnote{Note that if an end contains $k$ disjoint rays, for every $k \in \N$, then it also contains infinitely many disjoint rays \cite{DiestelBook16noEE}.} number of pairwise disjoint rays in~$\eps$. The \defn{combined degree} of $\eps$ is $\defnm{\Delta(\eps)} := \deg(\eps) + \dom(\eps)$.

The following lemma describes when an end has small combined degree.

\begin{LEM}{\cite{kConnectedSetsInInfGraphs}*{Corollary~5.8}} \label{prop:DegreeOfAnEndWitnessedBySeps}
    Let $\eps$ be an end of a graph $G$ and $k \in \N$. Then the following assertions hold:
    \begin{enumerate}
        \item \label{itm:DegreeWitnessingSeq:SeqImpliesDeg} If $\tau_\eps$ contains a weakly exhaustive increasing sequence of separations of order $\leq k$, then $\Delta(\eps) \leq k$.
        \item \label{itm:DegreeWitnessingSeq:DegImpliesSeq} If $\Delta(\eps) = k$, then $\tau_\eps$ contains a weakly exhaustive increasing sequence $((A_i, B_i))_{i \in \N}$ of separations of order $k$ such that $G[B_i \setminus A_i]$ is connected and $(A_i \cap B_i) \cap (A_j \cap B_j) \subseteq \Dom(\eps)$ for all $i \neq j \in \N$.
    \end{enumerate}
\end{LEM}

We need the following two lemmas, which describe the interaction between the ends of a torso and the ends of the underlying graph. The first lemma essentially says that every end of a torso stems from an end of the underlying graph.

\begin{PROP}{\cite{LinkedTDInfGraphs}*{Proposition~6.1}} \label{prop:RayInTorsoExtends}
    Let $\sigma$ be a star of finite-order separations of some graph~$G$ that are tight on the small side. Then there exists for every ray $R$ in $\torso(\sigma)$ a ray $R'$ in $G$ that meets $V(R)$ infinitely often.
\end{PROP}

The next lemma describes a condition that is sufficient to ensure that all ends of a torso have small degree.

\begin{LEM} \label{lem:DegreeOfEndsInTorso}
    Let $\sigma$ be a star of separations of order $<k \in \N$ of some graph $G$ such that every separation in~$\sigma$ distinguishes some pair of $k$-profiles in $G$ efficiently. Then for every end $\eps'$ of $\torso(\sigma)$ there exists an end\footnote{This end $\eps$ is in fact unique, but we do not need this.} $\eps$ of $G$ such that $\sigma \subseteq \tau_\eps$, $\Delta(\eps') = \Delta(\eps)$ and for every $(A,B) \in \tau_\eps$ we have $(A \cap \interior(\sigma), B \cap \interior(\sigma)) \in \tau_{\eps'}$ if $\{A \cap \interior(\sigma), B \cap \interior(\sigma)\}$ is a separation of $\torso(\sigma)$.
\end{LEM}

\begin{proof}
    Let $R'$ be some $\eps'$-ray. Since all separations in $\sigma$ efficiently distinguish some pair of $k$-profiles, they are tight on the small side by \cref{lem:EffYieldsTight}, so by \cref{prop:RayInTorsoExtends} there is a ray~$R$ in $G$ such that $|V(R) \cap V(R')| = \infty$. Let $\eps$ be the end of $G$ to which~$R$ belongs. Then $\sigma \subseteq \tau_\eps$, because~$V(R) \cap V(R') \subseteq \interior(\sigma)$ is infinite, and thus $R$ has a tail in $G[B\setminus A]$ for every $(A,B) \in \sigma$. Moreover, for all $(A,B) \in \tau_\eps$ that induce a separation $\{A \cap \interior(\sigma), B \cap \interior(\sigma)\}$ of $\torso(\sigma)$, we have $(A \cap \interior(\sigma), B \cap \interior(\sigma)) \in \tau_{\eps'}$ for the same reason.

    Since connected subgraphs of $G$ induce connected subgraphs of $\torso(\sigma)$, every $\eps$-ray induces an $\eps'$-ray and $\Dom(\eps) \subseteq \Dom(\eps')$.
    Hence, $\Delta(\eps) \leq \Delta(\eps')$.
    We claim that also $\Delta(\eps') \leq \Delta(\eps)$, which concludes the proof. 
    So suppose for a contradiction that $\Delta(\eps') > \Delta(\eps)$. Let $U \subseteq \interior(\sigma)$ be a set of size $n := \Delta(\eps) + 1 \leq \Delta(\eps')$ and $\cP := \{P_x : x \in U\}$ a family of~$n$ pairwise disjoint paths/rays in $\torso(\sigma)$ such that $P_x$ is either an $\eps'$-ray that starts in $x$ or the trivial path whose single vertex $x$ lies in $\Dom(\eps')$.
    As $U$ is finite, there exists by \cref{prop:DegreeOfAnEndWitnessedBySeps}~\cref{itm:DegreeWitnessingSeq:DegImpliesSeq} a separation $(A, B) \in \tau_\eps$ of order $\Delta(\eps)$ such that $U \subseteq A$ and $G[B\setminus A]$ is connected.

    Set $\rho := \{(C,D) \in \sigma : \{A,B\}, \{C,D\} \text{ cross}\}$ and $\rho' := \{(C,D) \in \rho : C\cap D \cap (A\setminus B) \neq \emptyset\}$. Let us first show that $\rho'$ is finite. 
    For this, it suffices to show that each of the pairwise disjoint strict small sides $C\setminus D$ of $(C,D) \in \rho'$ meets the finite set $A \cap B$.
    Since $(C,D)$ is tight on the small side, there is a component $K \subseteq G[C\setminus D]$ such that $N_G(K) = C \cap D$. In particular, since $C \cap D \cap (A\setminus B) \neq \emptyset$, we have that $A \cap B$ meets~$K$ if also $C \cap D \cap (B\setminus A) \neq \emptyset$. So we may assume that $B\setminus A$ avoids $C \cap D$. Then either $B\setminus A \subseteq C\setminus D$ or $B\setminus A \subseteq D\setminus C$ as $G[B\setminus A]$ is connected, which implies that $\{A,B\}, \{C,D\}$ are nested and contradicts $(C,D) \in \rho$. 

    Now set $(\bar{A}, \bar{B}) := (A, B) \wedge \bigwedge_{(C,D) \in \rho\setminus \rho'} (D,C)$ and $(\tilde{A}, \tilde{B}) := (\bar{A}, \bar{B}) \wedge \bigwedge_{(C,D) \in \rho'} (D,C)$. Then $|\bar{A} \cap \bar{B}| = |A \cap B|$ because $(A\setminus B) \cap (C \cap D) = \emptyset$ for all $(C,D) \in \rho\setminus \rho'$, and \mbox{$|\tilde{A} \cap \tilde{B}| \leq |(\bar{A} \cap 
    \bar{B}) \cup \bigcup_{(C,D) \in \rho'} (C \cap D)| < \infty$} because $\rho'$ is finite. So $\tau_\eps$ contains an orientation of $\{\bar{A}, \bar{B}\}$, $\{\tilde{A}, \tilde{B}\}$; since $\tau_{\eps}$ is consistent, we find $(\tilde{A}, \tilde{B}) \leq (\bar{A}, \bar{B}) \leq (A,B) \in \tau_\eps$. Since $U \subseteq \interior(\sigma) \subseteq D$ for all $(C, D) \in \rho$, we also have $U \subseteq \tilde{A}$.
    Moreover, by \cref{lem:Fishlemma}, $(\tilde{A}, \tilde{B})$ is nested with $\sigma$, and $(\bar{A}, \bar{B})$ crosses at most those finitely many separations in $\sigma$ that are contained in $\rho'$. It follows by \cite{SARefiningEssParts}*{Claim~1 in the proof of Lemma~4.3}\footnote{In \cite{SARefiningEssParts} this is shown for finite graphs $G$, but the proof only uses the finiteness of $G$ to conclude that $(\bar{A}, \bar{B})$ crosses at most finitely many separations in $\sigma$, which we have proved separately.} applied to $\tau_\eps$, $\sigma$ and $(\tilde{A}, \tilde{B}) \leq (\bar{A}, \bar{B})$ that there is a separation $(A',B') \in \tau_\eps$ of order $\leq |\bar{A} \cap \bar{B}| = |A\cap B| = \Delta(\eps)$ such that $(\tilde{A}, \tilde{B}) \leq (A',B')$, and thus $U \subseteq \tilde{A} \subseteq A'$. Since $(A,B) \in \tau_\eps$, its big side $G[B]$ contains a tail of $R$, and hence $V(R'') \subseteq B$ for some tail $R''$ of $R'$. By the choice of $P_x$, there exists an infinite family of $P_x$--$R''$ paths in $\torso(\sigma)$ that meet at most in their endvertices on $P_x$. In particular, since $(A',B')$ is nested with $\sigma$ and thus induces a separation of $\torso(\sigma)$, we find $x \in B'$ if $P_x$ is a trivial path, or $V(P'_x) \subseteq B'$ for some tail $P'_x$ of $P_x$ if $P_x$ is a ray. But since each $P_x$ is connected, it follows that $V(P_x) \cap (A' \cap B') \neq \emptyset$. Since the $P_x$ are pairwise disjoint, this implies that $|A' \cap B'| \geq \Delta(\eps) + 1$, a contradiction.   
\end{proof}

\subsection{Critical vertex sets} \label{subsec:UltrafilterTangles}

Given a set $X$ of vertices of a graph $G$, a component $K$ of $G-X$ is \defn{tight at $X$ in $G$} if $N_G(K) = X$. By slight abuse of notation, we
will refer to such $K$ as \defn{tight} components of $G-X$. We write \defn{$\cC_X$} for the set of
components of $G-X$ and $\defnm{\breve{\cC}_X} \subseteq \cC_X$ for the set of all tight components of $G-X$. 
A \defn{critical vertex set} of $G$ is a finite set $X \subseteq V(G)$ such that $\breve{\cC}_X$ is infinite \cite{KPEndsTanglesCVSets}.
The collection of all the critical vertex sets of~$G$ is denoted by \defn{$\crit(G)$}. 

\begin{LEM}{\cite{LinkedTDInfGraphs}*{Lemma~2.6}}\label{lem:TorsosAndCriticalVertexSets}
    Let~$\sigma$ be a star of finite-order separations of a graph~$G$ that are tight on the small side. 
    Then $\crit(\torso(\sigma)) \subseteq \crit(G)$. 
\end{LEM}

In the previous subsection we have seen that every end induces an $\aleph_0$-tangle. The following lemma asserts that also graphs with critical vertex sets have $\aleph_0$-tangles: for every critical vertex set $X$, every free ultrafilter on $\cC_X$ induces an $\aleph_0$-tangle\footnote{These $\aleph_0$-tangles are called \defn{ultrafilter tangles} \cite{EndsAndTangles}.}.

\begin{lemma}{\cite{EndsAndTangles}*{Lemmas 3.4 \& 3.7}} \label{lem:UltrafilterInduceTangles}
    Given some finite set $X$ of vertices of a graph $G$, for each free ultrafilter $U$ on $\cC_X$ there exists a (non-principal) $\aleph_0$-tangle $\tau$ in $G$ such that, for all $\cK \subseteq \cC_X$, we have $\big(V\big(G-\bigcup \cK\big), V\big(\bigcup\cK\big) \cup X\big) \in \tau$ if and only if $\cK \in U$. In particular, then
    \[
    \tau = \{(A, B) \in \vS_{\aleph_0}\; |\; \exists\, \cK \in U : \medmath{\bigcup}\, \cK \subseteq G[B]\}.
    \]
\end{lemma}


Even though the $\aleph_0$-tangles described in \cref{lem:UltrafilterInduceTangles} are non-principal, critical vertex sets still induce principal $k$-tangles, but only for $k \in \N$ that are not greater than their size.

\begin{LEM} \label{lem:TanglesAtLargeCritVS}
    Let $G$ be any graph, $k \in \N$, and let $\cF$ be a set of stars in $\vS_k$ all of which have finite interior. Further, let $X$ be a critical vertex set of $G$ of size $\geq k$. Then $\tau := \{(A,B) \in \vS_k : X \subseteq B\}$ is a principal $\cF$-tangle of $S_k$.
    In particular, if $\sigma \subseteq \vS_k$ is a star with $X \subseteq \interior(\sigma)$, then $\sigma \subseteq \tau$.
\end{LEM}

\begin{proof}
    Since $X$ is a critical vertex set and thus infinitely connected in $G$, and because $|X| \geq k$, every separation in $S_k$ has a unique side which contains $X$. Hence, $\tau$ is an orientation of $S_k$, which for the same reason is consistent. Now let $\sigma$ be a star contained in $\tau$. Since $X \cap (B\setminus A) \neq \emptyset$ for all $(A,B) \in \sigma$, every component of $G - X$ whose neighbourhood in $G$ equals $X$ meets~$B$. Hence, as components are connected, each such component of~$G-X$ meets $\interior(\sigma)$. Since $X$ is a critical vertex set of $G$, this implies that $|\interior(\sigma)| = \infty$, so $\sigma$ is not in~$\cF$. In particular, $\tau$ avoids $\cU_k$ and is thus principal.
\end{proof}

We also need the following lemma, which describes a sufficient condition for a star to be home to a (non-principal) tangle.

\begin{LEM} \label{lem:TanglesAtSmallCritVS}
    Let $G$ be any graph, $k \in \N$, and let $\cF$ be a set of finite stars in $\vS_k$ all of which have finite interior. Further, let $\sigma \subseteq \vS_k$ be a star of separations that are tight on the small side, and let $X \subseteq \interior(\sigma)$ be of size~$<k$. Suppose that either~$X$ is a critical vertex set of $\torso(\sigma)$ or that infinitely many separations in $\sigma$ have separator~$X$. Then there is a (non-principal) $\cF$-tangle of $S_k$ that lives in $\sigma$.
\end{LEM}

\begin{proof}
    Set $\mathfrak{K} := \{\cK \subseteq \cC_X : |\cK| = 1 \text{ or } \bigcup \cK \subseteq G[A \setminus B] \text{ for some } (A,B) \in \sigma\}$. Note that by \cref{lem:TorsosAndCriticalVertexSets} if $X \in \crit(\torso(\sigma))$ or by the tightness on the left side of the separations in $\sigma$ if $\{(A,B) \in \sigma : A \cap B = X\}$ is infinite,~$X$ is critical in $G$. In particular, every $\cK \in F := \{\cC_X\} \cup \{\cC_X \setminus (\cK_1 \cup \dots \cup \cK_n) : n \in \N, \cK_1, \dots, \cK_n \in \mathfrak{K}\}$ still contains infinitely many components of $G-X$. 
    In particular, $\emptyset \notin F$. 
    Moreover, by definition, $F$ is closed in $\cC_X$ under taking supersets and under finite intersections. Thus, $F$ is a filter on~$\cC_X$.
    So by \cref{lem:UltrafilterInduceTangles} (and the ultrafilter lemma), there is an $\aleph_0$-tangle~$\tau$ in~$G$ such that $(V(G-(\bigcup \cK)), V\big(\bigcup \cK) \cup X) \in \tau$ for all $\cK \in F$. 
    By \cite{EndsAndTangles}*{Lemma~1.3 \& Corollary~1.5}, $\tau$ avoids $\cF$.
    Moreover, by \cref{lem:UltrafilterInduceTangles}, there exists for every $(A,B) \in \tau$ a collection $\cK \subseteq \cC_X$ such that $\bigcup \cK \subseteq G[B]$ 
    and $\cK \notin \mathfrak{K}$.  
    By the definition of $\mathfrak{K}$, this implies that $(B,A) \notin \tau$ for all $(A,B) \in \sigma$, and hence $\sigma \subseteq \tau$.
    So $\tau \cap \vS_k$ is as desired.
\end{proof}

\section{Refining inessential stars}\label{sec:RefiningInessStars}

In this section we prove \cref{lem:ReflemForInfGraphsInessStars}, which is one of the two main ingredients to the proof of \cref{main:TTDInfGraphs}.
Let us begin by giving a brief sketch of the proof of \cref{main:TTDInfGraphs}.
Similarly to the finite tangle-tree duality theorem, \cref{main:TTDFiniteGraphs}, it is not too difficult to show that not both, \ref{itm:TTDInf:Tangle} and \ref{itm:TTDInf:Tree} of \cref{main:TTDInfGraphs}, can hold at the same time. To see that at least one of the two assertions holds, we consider an arbitrary graph~$G$ without any $k$-tangles as in \ref{itm:TTDInf:Tangle}.
For the proof that $G$ then has an $S_k$-tree as in \ref{itm:TTDInf:Tree}, we need two ingredients.
The first one is a certain \td\ of $G$, whose existence follows from a result of the author, Jacobs, Knappe and Pitz (\cite{LinkedTDInfGraphs}; see \cref{thm:StartingTD} in \cref{sec:TTDInfGraphs}): 
$G$ admits a \td\ $(T, \cV)$ of adhesion $<k$ into finite parts such that every node $t$ of $T$ with $\sigma_t \notin \cU^\infty_k$ is inessential and has finite degree.

The second main ingredient to the proof of \cref{main:TTDInfGraphs} is \cref{lem:ReflemForInfGraphsInessStars} below.
This lemma ensures that, under mild additional assumptions on $(T, \cV)$, there exists, for every inessential node $t \in T$, a finite $S_k$-tree $(T^t, \alpha^t)$ over $\cT_k^* \cup \{\{\sv\} : \vs \in \sigma_t\}$ in which each $\vs \in \sigma_t$ appears as a leaf separation. In particular, all its non-leaves are associated with stars in $\cT^*_k$.
We then obtain a weakly exhaustive $S_k$-tree over $\cT^*_k \cup \cU^\infty_k$ by sticking the $S_k$-trees $(T^t, \alpha^t)$ together along $T$. More precisely, for every edge $e = \{t,s\}$ of $T$ we glue the trees $T^t$ and $T^{s}$ together along those leaf edges $f$ of $T^t$ and $f'$ of $T^{s}$ with $\alpha^t(f) = \{A,B\} = \alpha^{s}(f')$ where $\{A,B\}$ is the separation induced by the edge $e$ of $T$ (see \cref{constr:StickingSkTreesTogether}). 
\medskip

The idea of refining the inessential parts of a \td\ $(T, \cV)$ with $S_k$-trees as described above has its origin in \cite{JoshRefining}. There, Erde proved for finite graphs that if all edges of $T$ induce separations that efficiently distinguish two $k$-tangles, then such $S_k$-trees $(T^t, \alpha^t)$ exist for all inessential nodes $t$ of~$T$. The main result of this section generalizes his lemma (\cite{JoshRefining}*{Lemma~3.1}) not only to infinite graphs but also to certain \tds\ which no longer need to distinguish the $k$-tangles efficiently. To state this result formally, we need some further definitions.

First, we recall the definition of `closely related' from \cite{SARefiningInessParts}: Let $G$ be any graph and $k \in \N$. A separation $(A, B) \in \vS_k$ is \defn{closely related} to an orientation~$O$ of~$S_k$ if $(A,B) \in O$ and for every $(C,D) \in O$ we have $(A \cap C, B \cup D) \in \vS_k$.

\begin{proposition}{\cite{SARefiningInessParts}*{Proposition~3.4}}\label{prop:eff}
	Let $k \in \N$, and let $P$ and $P'$ be two $k$-profiles in a graph $G$. If a separation $(A,B) \in P$ distinguishes $P$ and $P'$ efficiently, then $(A,B)$ and $(B,A)$ are closely related to $P$ and~$P'$, respectively.
\end{proposition}

A finite-order separation $(A, B)$ of a graph $G$ is \defn{$\ell$-robust on the small side} for~$\ell \in \N$ if there exist a set $U \subseteq A$ of size~$\ell$ and a family~$\{P_x : x \in A \cap B\}$ of pairwise disjoint paths in~$G[A]$ such that~$P_x$ ends in~$x$ and there are~$\ell$ $U$--$P_x$ paths in~$G[(A \setminus B) \cup \{x\}]$ that meet at most in their endvertices in~$P_x$. An unoriented separation $\{A,B\}$ is \defn{$\ell$-robust} if both $(A,B)$ and $(B,A)$ are $\ell$-robust on the small side.

Note that the property `$\ell$-robust on the small side' is designed to mimic the presence of a highly connected substructure of $G$ on the small side of a separation. 
In fact, we defined `$\ell$-robust on the small side' precisely so that separations that are $\ell$-robust on the small side can be obtained from ends and critical vertex sets in the following way (cf.\ \cref{fig:EllRobust}):

\begin{figure}[ht]
    \centering
    \begin{subfigure}{0.4\linewidth}
        \centering
        \includegraphics[width=0.8\linewidth]{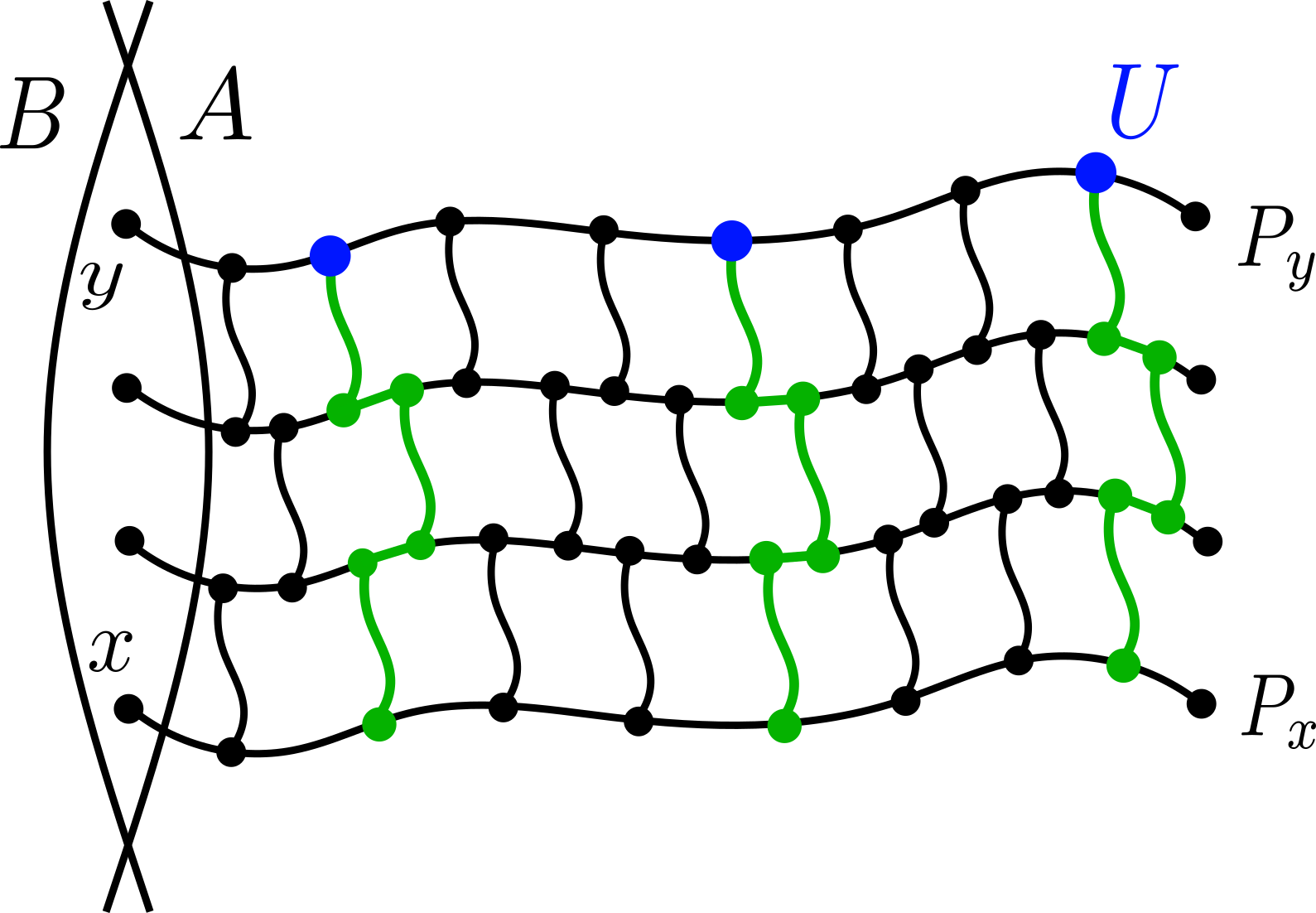}
        \caption{}
        \label{subfig:EllRobust:End}
    \end{subfigure}
    \begin{subfigure}{0.4\linewidth}
        \centering
        \includegraphics[width=0.4\linewidth]{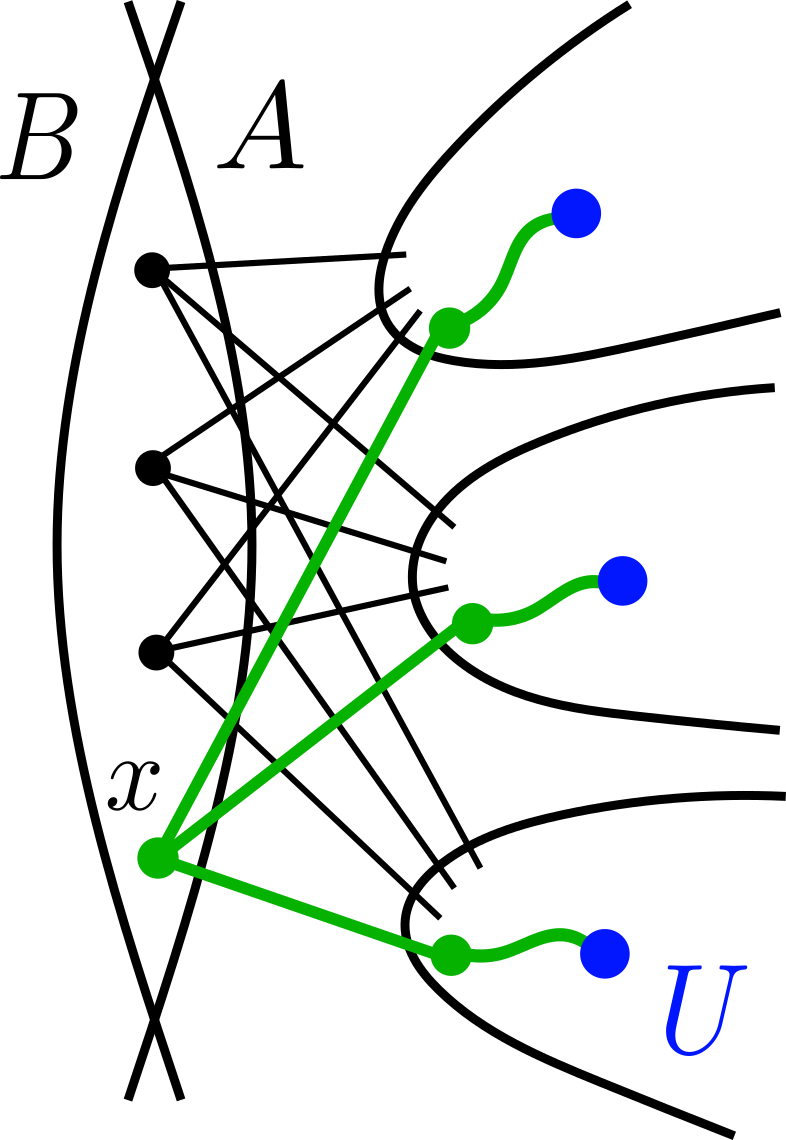}
        \caption{}
        \label{subfig:EllRobust:CritVS}
    \end{subfigure}
    \caption{Depicted are two separations $(A,B)$ that are $3$-robust on the small side. In the left figure, if $(A,B) \in \tau_\eps$ for some end $\eps$ of degree $4$, then the paths $P_x$ can be thought of as subpaths of disjoint $\eps$-rays. Then the green $U$--$P_x$ exist because all $\eps$-rays are equivalent.}
    \label{fig:EllRobust}
\end{figure}

\begin{lemma}{\cite{LinkedTDInfGraphs}*{Lemma~8.5}} \label{lem:ContractingEdgesEllRobust}
    Let $\eps$ be an end of a graph $G$ of finite combined degree, and let $((A_i, B_i))_{i \in \N}$ be a weakly exhaustive increasing sequence of separations in $\tau_\eps$ such that $\liminf_{i \in \N} |A_i \cap B_i| = \Delta(\eps)$. Then cofinitely many $(A_i, B_i)$ with $|A_i \cap B_i| = \Delta(\eps)$ are $\ell$-robust.
\end{lemma}

\begin{proposition} \label{prop:LeftRobustCrit}
    Let $\{A,B\}$ be a finite-order separation of a graph $G$, and suppose that $G[A\setminus B]$ contains infinitely many tight components of $G-X$ for some set $X \supseteq A \cap B$ of vertices of $G$. Then $(A,B)$ is $\ell$-robust on the small side for all $\ell \in \N$.
\end{proposition}

\begin{proof}
    This is witnessed by the trivial paths $P_x$ in $A \cap B$ and a set $U$ consisting of $\ell$ vertices that lie in pairwise distinct tight components of $G-X$ contained in $G[A\setminus B]$ (cf.\ \cref{subfig:EllRobust:CritVS}). 
\end{proof}

The definition of `$\ell$-robust on the small side' is tailored to the proof of \cref{main:TTDInfGraphs}: `$\ell$-robust on the small side' is defined precisely so that we can prove \cref{lem:ReflemForInfGraphsInessStars} below, and, at the same time, show that there exists a \td\ $(T, \cV)$ as described above whose `relevant' edges all induce separations that are $\ell$-robust on the small side (see \cref{thm:StartingTD}), and whose nodes are thus eligible for the application of \cref{lem:ReflemForInfGraphsInessStars}. 

A set $\cF$ of stars in $\vS_{\aleph_0}$ is \defn{$m$-bounded} for some $m \in \N$ if $\abs{\interior(\rho)} \leq m$ for all~$\rho \in \cF$. It is \defn{finitely bounded} if it is $m$-bounded for some $m \in \N$. 
The main result of this section then reads as follows:

\begin{LEM}\label{lem:ReflemForInfGraphsInessStars}
    Let $G$ be a graph, $k, m \in \N$, and let $\cF$ be an $m$-bounded, nice set of stars in~$\vS_k$. Set
    $\ell := \max\{3k-2, k(k-1)m + m\}$, and let $\sigma := \{\vs_1, \dots, \vs_n\} \subseteq \vS_k$ be a finite star with finite interior. Suppose that every separation in~$\sigma$ is either $\ell$-robust on the small side or has an inverse that is closely related to some $k$-profile in $G$ that avoids~$\cF$. 
    Set $\cF' := \cF \cup \{\{\sv_i\} : i \in [n]\}$.
    Then either there is an $\cF'$-tangle of~$S_k$ or there is a finite~$S_k$-tree over~$\cF'$ in which each $\vs_i$ appears as a leaf separation.
\end{LEM}

\noindent We remark that Erde \cite{JoshRefining} gave an example which shows that there need not exist an $S_k$-tree over $\cF'$ for every inessential star, even if $G$ is finite and $\cF = \cT_k^*$ for some~$k \in \N$. Thus, the additional assumptions on the separations in~$\sigma$ cannot be omitted. Moreover, \cref{ex:FinitelyBoundedIsNecessaryForAnSkTree} shows that we cannot omit the assumption that $\cF$ is finitely bounded.
\medskip

The remainder of this section is devoted to the proof of \cref{lem:ReflemForInfGraphsInessStars}, which we briefly sketch here.
The idea is to derive \cref{lem:ReflemForInfGraphsInessStars} from \cref{thm:TTDinASS}. In order to apply \cref{thm:TTDinASS}, we first reduce the problem to some finite separation system. 
For this, we define a subsystem~$S^\sigma_k \subseteq S_k$ that consists only of those separations of $G$ that are `relevant' for finding either an $\cF'$-tangle of~$S_k$ or an $S_k$-tree over~$\cF'$.
For a star $\sigma \subseteq \vS_k$, set
\[
\defnm{S^\sigma_k} := \{r \in S_k : \vs \leq \vr \text{ or } \vs \leq \rv \text{ for every } \vs \in \sigma\}.
\]

As we will see in a moment, $\vS^\sigma_k$ is finite and $\cF'$-separable if $\sigma$ is finite and has finite interior.
We can thus apply \cref{thm:TTDinASS} to $S^\sigma_k$ and $\cF'$, which yields either an $\cF'$-tangle of $S^\sigma_k$ or an $S^\sigma_k$-tree over~$\cF'$. By definition, an $S^\sigma_k$-tree over $\cF'$ is already an $S_k$-tree over~$\cF'$. 
The main part of the proof is then concerned with showing that every $\cF'$-tangle $\tau$ of $S_k^\sigma$ \defn{extends} to an $\cF'$-tangle of~$S_k$: that there exists an $\cF'$-tangle $\tau'$ of $S_k$ such that $\tau \subseteq \tau'$.

\begin{PROP}\label{prop:FiniteStarAndInteriorImpliesFiniteSkSigma}
    Given a graph $G$ and $k \in \N$, for every finite star $\sigma \subseteq \vS_k$ with finite interior the set~$S_k^\sigma$ is finite.
\end{PROP}

\begin{proof}
    By definition, every separation $\{A, B\} \in S_k^\sigma$ is nested with $\sigma$, and thus for every $(C,D) \in \sigma$ we have $C \cap D \subseteq A$ or $C \cap D \subseteq B$. Hence, every separation $\{A,B\} \in S^\sigma_k$ induces a separation $\{A \cap \interior(\sigma), B \cap \interior(\sigma)\}$ of $\torso(\sigma)$. As $\interior(\sigma)$ is finite, there are only finitely many separations of $\torso(\sigma)$. It thus suffices to show that only finitely many separations in $S_k^\sigma$ induce the same separation of $\torso(\sigma)$. 

    For this, let $\{A', B'\}$ be a separation of $\torso(\sigma)$. Then every separation of $G$ that induces $\{A', B'\}$ can be obtained from $\{A', B'\}$ by adding each component $K$ of $G - \interior(\sigma)$ to one side of $\{A',B'\}$ that contains~$N_G(K)$. It is straightforward to check that the arising separation will be in $S_k^\sigma$ if and only if we added for every separation $(C, D) \in \sigma$ all components of $G - \interior(\sigma)$ that are contained in $G[C \setminus D]$ to the same side of $\{A', B'\}$. It follows that at most $2^{\abs{\sigma}}$ separations in $S_k^\sigma$ induce the same separation of $\torso(\sigma)$. As $\sigma$ is finite, this concludes the proof.
\end{proof}

Given an $\cF'$-tangle $\tau$ of $S_k^\sigma$, we now inductively construct an $\cF'$-tangle $\tau'$ of $S_k$ with $\tau \subseteq \tau'$. For this, recall that $\sigma =: \{\vs_1, \dots, \vs_n\}$ is finite by assumption. So to define $\tau'$, we may proceed by extending~$\tau =: \tau_n$ step-by-step to an $\cF'$-tangle~$\tau_i$ of~$S_k^{\sigma_i}$ where $\sigma_i := \{\vs_1, \dots, \vs_i\}$ for $i \in [n]$ and~$\sigma_0 := \emptyset$. Then~$\tau_0$ will be the desired $\cF'$-tangle of~$S_k = S^{\sigma_0}_k$.

The following two lemmas describe how to obtain the tangle $\tau_{i-1}$ from $\tau_{i}$. We distinguish between two cases: whether $\sv_i$ is closely related to some $k$-profile that avoids $\cF$ or whether $\vs_i$ is $\ell$-robust on the small side. 

\begin{LEM}\label{lem:TangleInsideStar}
Let $G$ be a graph, $k \in \N$, and let $\cF$ be a set of stars in~$\vS_k$ that is strongly closed under shifting. Further, let~$\sigma \subseteq \vS_k$ be a star, and suppose there is some $\vs \in \sigma$ such that~$\sv$ is closely related to some $k$-profile in~$G$ that avoids~$\cF$. Set $\cF' := \cF \cup \{\{\rv\} : \vr \in \sigma\}$ and $\sigma' := \sigma \setminus \{\vs\}$. 
Then the following assertions hold:
    \begin{enumerate}[label=\rm{(\roman*)}]
        \item\label{itm:TangleInsideStar1} If $\tau'$ is an $\cF'$-tangle of $S^{\sigma'}_k$, then $\tau' \cap \vS^\sigma_k$ is an $\cF'$-tangle of $S^\sigma_k$.
        \item\label{itm:TangleInsideStar2} If $\tau$ is an $\cF'$-tangle of $S^{\sigma}_k$, then $\tau$ extends to an $\cF'$-tangle of $S^{\sigma'}_k$.
    \end{enumerate}
\end{LEM}

\begin{proof}
By definition it is clear that every $\cF'$-tangle $\tau$ of $S^{\sigma'}_k$ induces an $\cF'$-tangle $\tau \cap \vS^\sigma_k$ of $S_k^\sigma$.

For \ref{itm:TangleInsideStar2}, let $\tau$ be an $\cF'$-tangle of $S_k^\sigma$. We extend $\tau$ to an orientation $\tau'$ of $S^{\sigma'}_k$ as follows. Let~$r \in S^{\sigma'}_k$ be given, and first assume that $r$ is nested with $s$.
Then either $r \in S_k^\sigma$, and we then let~$\vr \in \tau'$ if and only if $\vr \in \tau$, or $r$ has an orientation that is smaller than $\vs$, and we then let~$\vr \in \tau'$ if and only if $\vr \leq \vs$. 
Second, assume that $r$ and $s$ cross, and fix an orientation $\vr$ of $r$. By the assumption on~$\vs$, its inverse~$\sv$ is closely related to some $k$-profile $P$ in $G$ that avoids $\cF$.
Since $r \in S_k$, it is oriented by $P$; we set 
$\vt := \vr \vee \vs$ if $\rv \in P$ and $\vt := \vr \wedge \sv$ if $\vr \in P$.
As $\sv$ is closely related to $P$, it follows that $t$ has order $<k$, and thus $t \in S^\sigma_k$ by \cref{lem:Fishlemma}. Hence, $\tau$ contains an orientation of $t$; we let $\vr \in \tau'$ if $\vt \in \tau$, and~$\rv \in \tau'$ otherwise.

By definition,~$\tau'$ is an orientation of~$S^{\sigma'}_k$ and $\tau \subseteq \tau'$; in particular,~$\tau$ extends to~$\tau'$ and $\sigma \subseteq \tau$.
Hence, we are left to show that $\tau'$ is an $\cF$-tangle of $S^{\sigma'}_k$.
For this, suppose for a contradiction that there is a set $\rho \subseteq \tau'$ which has one of the following two forms: either~$\rho = \{\vr_i : i \in I\}$ is a star in~$\cF$, or $\rho = \{\vr_1, \vr_2\}$ with $\rv_1 < \vr_2$.

As $P$ is consistent and avoids $\cF$, we have $\rho \not\subseteq P$, and thus $P$ contains the inverse $\rv_i$ of a separation $\vr_i \in \rho$, say $\rv_1 \in P$. 
If~$\rho$ is a star in $\cF$, then $\vr_i \in P$ for all $\vr_i \in \rho \setminus \{\vr_1\}$ since~$P$ is consistent.
As~$\sv$ is closely related to the $k$-profile~$P$, it follows from \cref{prop:ShiftingForbiddenStars} that $\rho' := \{\vr_1 \vee \vs\} \cup \{(\vr_i \wedge \sv) : \vr_i \in \rho\setminus \{\vr_1\}\}$ is a star in~$\cF$. But then $\rho' \subseteq \tau$ by the definition of~$\tau'$, which contradicts that~$\tau$ is an $\cF'$-tangle of~$S_k^\sigma$.

Otherwise, if $\rho = \{\vr_1, \vr_2\}$ with $\rv_1 < \vr_2$, then, since $\sv$ is closely related to $P$, we have
$\vr_1 \vee \vs, \vr_2 \vee \vs \in \vS_k$ if $\rv_2 \in P$, or $\vr_1 \vee \vs, \vr_2 \wedge \sv \in \vS_k$ if $\vr_2 \in P$. By the definition of~$\tau'$, it follows that $\{\vr_1 \vee \vs, \vr_2 \vee \vs\} \subseteq \tau$ or $\{\vr_1 \vee \vs, \vr_2 \wedge \sv\} \subseteq \tau$, respectively, which contradicts that $\tau$ is consistent because $(\vr_1 \vee \vs)^* \leq \rv_1 < \vr_2 \leq \vr_2 \vee \vs$ and $(\vr_1 \vee \vs)^* = \rv_1 \wedge \sv \leq \vr_2 \wedge \sv$.
\end{proof}

\begin{LEM}\label{lem:TangleInsideStarComplicated}
    Let $G$ be any graph, $k, m \in \N$, and let $\cF$ be an $m$-bounded, nice set of stars in~$\vS_k$.
    Set $\ell := \max\{3k-2, k(k-1)m + m\}$, let $\sigma \subseteq \vS_k$ be a star, and suppose that some $(C,D) \in \sigma$ is $\ell$-robust on the small side. Set $\cF' := \cF \cup \{\{\rv\} : \vr \in \sigma\}$ and $\sigma' := \sigma \setminus \{(C,D)\}$. 
    Then the following assertions hold:
    \begin{enumerate}[label=\rm{(\roman*)}]
        \item\label{itm:TangleInsideStarComplicated1} If $\tau'$ is an $\cF'$-tangle of $S^{\sigma'}_k$, then $\tau' \cap \vS^\sigma_k$ is an $\cF'$-tangle of $S^\sigma_k$.
        \item\label{itm:TangleInsideStarComplicated2} If $\tau$ is an $\cF'$-tangle of $S^{\sigma}_k$, then $\tau$ extends to an $\cF'$-tangle of $S^{\sigma'}_k$.
    \end{enumerate}
\end{LEM}

\begin{proof}
    By definition it is clear that every $\cF'$-tangle $\tau'$ of $S^{\sigma'}_k$ induces an $\cF'$-tangle $\tau' \cap \vS^\sigma_k$ of $S_k^\sigma$.

    For \ref{itm:TangleInsideStarComplicated2}, let~$\tau$ be an $\cF'$-tangle of~$S_k^\sigma$.  
    Fix a set $U \subseteq C$ of size $\ell$ and a family $\{P_x : x \in C \cap D\}$ of disjoint paths witnessing that $(C,D)$ is $\ell$-robust on the small side. 
    We define an orientation $\tau'$ of $S_k^{\sigma'}$ that extends~$\tau$, and then show that $\tau'$ is an $\cF'$-tangle. For this, let $\{A,B\} \in S^{\sigma'}_k$. If some $A$ meets every path~$P_x$ at least once, then we will let $\tau'$ orient $\{A,B\}$ in the same way as $\tau$ orients the corner $\{A \cup C, B \cap D\}$ of $\{A,B\}$ and $\{C,D\}$ (see \cref{fig:Reflemma}, and \eqref{eq:DefOfTau} below). Our first four claims show that this indeed defines an orientation of $S^{\sigma'}_k$.

    \begin{figure}[ht]
        \centering
        \begin{subfigure}{0.45\linewidth}
            \centering
            \includegraphics[width=0.9\linewidth]{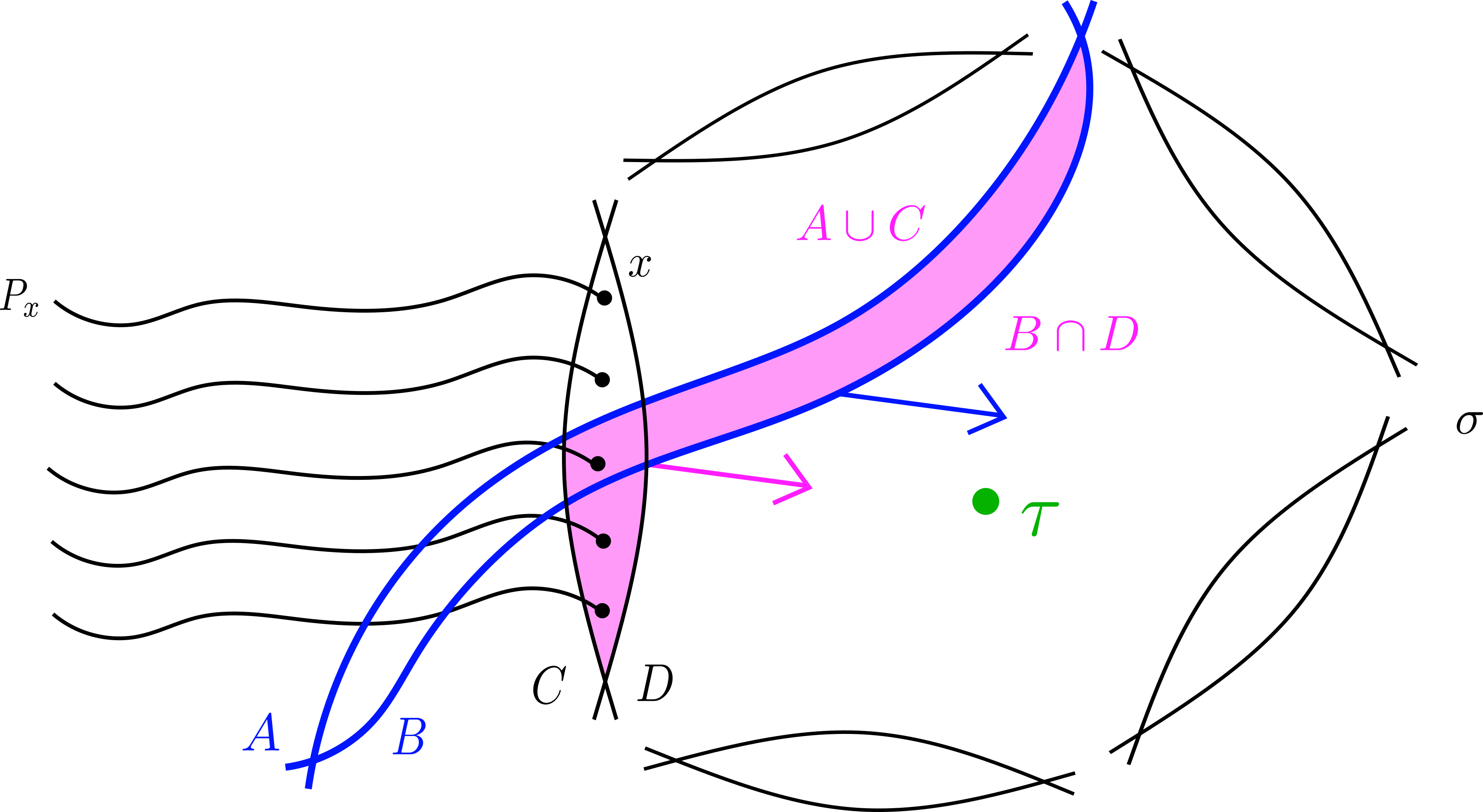}
            \caption{}
        \end{subfigure}
        \begin{subfigure}{0.45\linewidth}
            \centering
            \includegraphics[width=0.9\linewidth]{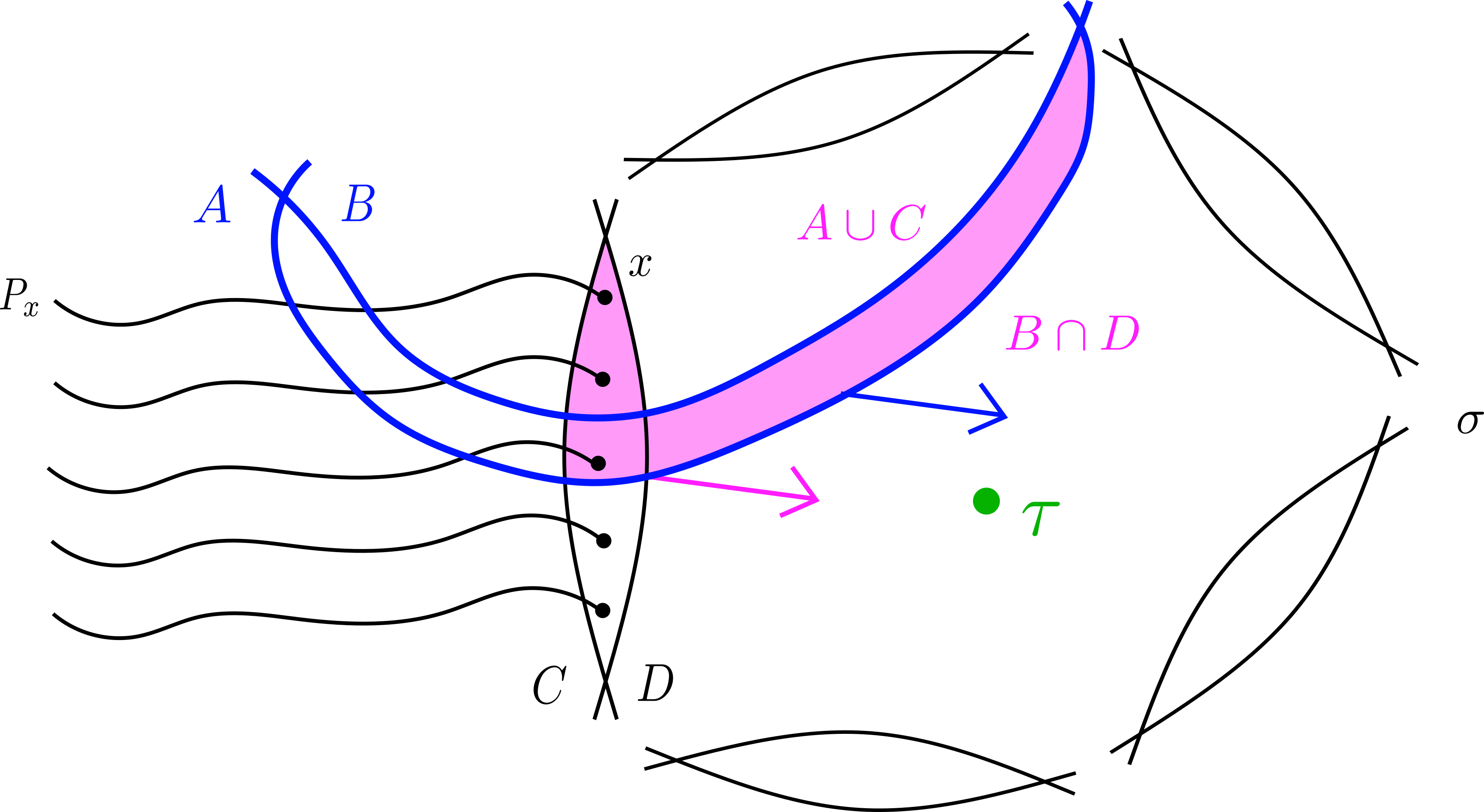}
            \caption{}
        \end{subfigure}
        \caption{Depicted is the definition of the orientation $\tau'$, as given by \eqref{eq:DefOfTau}. By \cref{claim:ExSideMeetingAllRays}, there is a side of $\{A,B\}$ that meets every path $P_x$ at least once, here $A$, and by \cref{claim:TauContainsAnOrient}, $\tau$ contains an orientation of the corner $\{A \cup C, B \cap D\}$ of $\{A,B\}$ and $\{C,D\}$ (depicted in pink). By \cref{claim:TriangleDefinesAnOrientation}, orienting $\{A,B\}$ in the same direction as $\tau$ orients $\{A \cup C, B \cap D\}$ yields an orientation $\tau'$ of $S_k^{\sigma'}$.}
        \label{fig:Reflemma}
    \end{figure}
    
    \begin{claim}\label{claim:ExSideMeetingAllRays}
        \emph{Every $\{A,B\} \in S_k$ has a side that meets every~$P_x$ at least once.}
    \end{claim}
    
    \begin{claimproof}
        Since $\ell \geq 3k-2$ and $|A \cap B| < k$, some strict side of $\{A,B\}$, say $A\setminus B$, contains a subset~$U'$ of~$U$ of at least $k$ vertices. 
        By the choice of $U$ and the $P_x$, there are at least $k$ pairwise internally disjoint $U'$--$P_x$ paths for every $x \in C \cap D$. Since $U' \subseteq A\setminus B$ and $|A \cap B| < k$, at least one of these paths is contained in~$A$. Thus, $A$ meets every~$P_x$ at least once.
    \end{claimproof}

    \begin{claim} \label{claim:CornerOfOrderLessThank}
        \emph{For every $\{A,B\} \in S_k$, if $A$ meets every $P_x$, then $\{A \cup C, B \cap D\}$ has order $<k$.}
    \end{claim}
    
    \begin{claimproof}
        By assumption, there are vertices $p_x \in A \cap V(P_x)$ for every $x \in C \cap D$. Since also $V(P_x) \subseteq C$ for every $x \in C \cap D$, we have $p_x \in A \cap C$, and thus every $P_x$ meets $A \cap C$. Additionally, as every~$P_x$ ends in $x \in C \cap D$, it also meets $B \cup D$. Since $P_x$ is connected and $\{A \cap C, B \cup D\}$ is a separation of $G$, it follows that $V(P_x) \cap ((A \cap C) \cap (B \cup D)) \neq \emptyset$ for all $x \in C \cap D$. Thus $\{A \cap C, B \cup D\}$ has order $\geq \abs{C \cap D}$, as the~$P_x$ are disjoint.
        By submodularity, this implies that $\{A \cup C, B \cap D\}$ has order $\leq \abs{A \cap B} < k$. 
    \end{claimproof}

    \begin{claim} \label{claim:TauContainsAnOrient}
        \emph{For every $\{A,B\} \in S^{\sigma'}_k$, if $A$ meets every $P_x$, then $\tau$ contains an orientation of $\{A \cup C, B \cap D\}$.}
    \end{claim}

    \begin{claimproof} 
        By \cref{claim:CornerOfOrderLessThank} and \cref{lem:Fishlemma}, $\{A \cup C, B \cap D\}$ is contained in $S_k^\sigma$, which clearly implies the assertion.
    \end{claimproof}

    \begin{claim}\label{claim:TriangleDefinesAnOrientation}
        \emph{If both sides $A$ and $B$ of some $\{A,B\} \in S_k^{\sigma'}$ meet every $P_x$ at least once, then either $(A \cup C, B \cap D)$, $(A \cap D, B \cup C) \in \tau$ or $(B \cup C, A \cap D), (B \cap D, A \cup C) \in \tau$.}
    \end{claim}

    \begin{claimproof}
        By \cref{claim:TauContainsAnOrient}, $\tau$ orients both separations, $\{A \cup C, B \cap D\}$ and $\{B \cup C, A \cap D\}$. 
        Since $\tau$ is consistent, it cannot contain both $(A \cup C, B \cap D)$ and $(B \cup C, A \cap D)$.
        We show that $D \cap (A \cup C) \cap (B \cup C)$ has size less than $k$. The assertion then follows as $\cP'_k \subseteq \cF$ because $\cF$ is nice, and $(C,D) \in \tau$ as well as $\{(C,D), (B \cap D, A \cup C), (A \cap D, B \cup C)\} \in \cP_k$. We have 
        \begin{equation*}
            D \cap (A \cup C) \cap (B \cup C) = D \cap ((A \cap B) \cup C) = (A \cap B \cap D) \cup (C \cap D) = (C \cap D) \dot\cup (A \cap B \cap (D \setminus C)),
        \end{equation*}
        and thus 
        \begin{equation*}
            \abs{D \cap (B \cup C) \cap (A \cup C)} = \abs{C \cap D} + (\abs{A \cap B} - \abs{A \cap B \cap C}) < \abs{C \cap D} + k - \abs{A \cap B \cap C}.
        \end{equation*}
        Since both $A$ and $B$ meet every~$P_x$, and because every $P_x$ is connected, also $A \cap B$ meets every~$P_x$. As all~$P_x$ are pairwise disjoint and contained in $G[C]$, it follows that $\abs{A \cap B \cap C} \geq |C \cap D|$. Hence, 
        \begin{equation*}
            \abs{D \cap (B \cup C) \cap (A \cup C)} < |C \cap D| + k - |C \cap D| = k.
        \end{equation*}
        This completes the proof of the claim.
    \end{claimproof}
    
    We now define an orientation $\tau'$ of $S_k^{\sigma'}$ as follows:

    \begin{equation}
    \begin{aligned}
        &\hspace{-3mm}\text{\emph{For every} } \{A,B\} \in S_k^{\sigma'}, \text{ \emph{if} } A \text{ \emph{meets every path} } P_x \text{ \emph{at least once, then we}}\\[-1mm]
        &\hspace{-3mm}\text{\emph{let} } (A,B) \in \tau' \text{ \emph{if} } (A \cup C, B \cap D) \in \tau, \text{ \emph{and} } (B,A) \in \tau' \text{ \emph{if} } (B \cap D, A \cup C) \in \tau.
    \end{aligned}
    \tag{$\triangle$} \label{eq:DefOfTau}
    \end{equation}
    
    \noindent By \cref{claim:ExSideMeetingAllRays,claim:TauContainsAnOrient,claim:TriangleDefinesAnOrientation}, $\tau'$ contains precisely one orientation of every separation in $S^{\sigma'}_k$. 
    We claim that~$\tau'$ is an $\cF'$-tangle of $S^{\sigma'}_k$ and that $\tau$ extends to $\tau'$, i.e.\ $\tau \subseteq \tau'$. 
    
    We first show the latter. For this, let $(A,B) \in \tau$ be given. Then $\{A,B\} \in S^{\sigma}_k$, which implies that either $(C, D) \leq (A,B)$ or $(C, D) \leq (B,A)$. In the first case, we have $V(P_x) \subseteq C \subseteq A$ for every $x \in C \cap D$. Moreover, $(A \cup C, B \cap D) = (A, B) \in \tau$, and thus $(A,B) \in \tau'$ by~\eqref{eq:DefOfTau}. Analogously, we find in the second case that $(B,A) \in \tau'$.

    Towards a proof that $\tau'$ is an $\cF'$-tangle of $S_k^{\sigma'}$, we first show that $\tau'$ is consistent.
    For this, suppose for a contradiction that there are $(A,B), (E,F) \in \tau'$ such that $(B,A) < (E,F)$. By \cref{claim:ExSideMeetingAllRays}, $\{A,B\}$ has a side that meets every $P_x$ at least once; we first assume that $B \cap V(P_x) \neq \emptyset$ for all $x \in C \cap D$. Then also $E \supseteq B$ meets every $P_x$, and thus, by~\eqref{eq:DefOfTau}, we have $(A \cap D, B \cup C), (E \cup C, F \cap D) \in \tau$, which contradicts that $\tau$ is consistent as $(B \cup C, A \cap D) \leq (E \cup C, F \cap D)$. The case that $F$ meets every~$P_x$ is symmetric, so we may assume that $A \cap V(P_x) \neq \emptyset \neq E \cap V(P_x)$ for every $x \in C \cap D$. Then $(A \cup C, B \cap D), (E \cup C, F \cap D) \in \tau$ by~\eqref{eq:DefOfTau}, which again contradicts that $\tau$ is consistent since $(B \cap D, A \cup C) \leq (B,A) < (E,F) \leq (E \cup C, F \cap D)$.
    
    It remains to show that $\tau'$ avoids $\cF'$. We have already seen that $\tau \subseteq \tau'$. Since $\tau$ avoids~$\cF'$, this implies that $\sigma \subseteq \tau \subseteq \tau'$.
    So suppose for a contradiction that there is a star $\rho \subseteq \tau'$ such that $\rho \in \cF$. 
    
    \begin{claim} \label{claim:SmallSideInForbiddenStarMeetingEveryPx}
        \emph{There is a separation $(A, B) \in \rho$ such that $A$ meets every path $P_x$ at least once.}
    \end{claim}
    
    \begin{claimproof}
        Let us first assume that there is a separation $(A,B) \in \rho$ whose strict small side $A\setminus B$ contains a set~$U'$ of $k$ vertices from $U$.
        By the choice of $U$ and the $P_x$, it follows that there are $k$ internally disjoint $U'$--$P_x$ paths, for every $x \in C \cap D$. 
        Since $U' \subseteq A\setminus B$ and $|A \cap B| < k$, at least one of these paths is contained in~$A$. Thus, $A$ meets every~$P_x$ at least once.

        Now suppose that no separation in $\rho$ contains more than $k-1$ vertices from $U$ in its strict small side. Since $|\interior(\rho)| \leq m$, it follows that $U \cap (A \setminus B) \neq \emptyset$ for at least $(\ell-m)/(k-1) = km$ separations $(A,B) \in \rho$. Let $\rho' \subseteq \rho$ be the set of these separations, and pick for each $(A,B) \in \rho'$ some $v_A \in U \cap (A \setminus B)$. Further, fix some $y \in C \cap D$. By the choice of $U$, there is a family $\{Q_{(A,B)} : (A,B) \in \rho'\}$ of internally disjoint paths such that $Q_{(A,B)}$ starts in~$v_A$ and ends in a vertex of $P_y$. As each $Q_{(A,B)}$ meets $A \cap B \subseteq \interior(\rho)$ in an internal vertex if $V(P_y) \cap A = \emptyset$, and since $|\interior(\rho)| \leq m$,~$P_y$ meets the small sides $A$ of at least $(k-1)m$ separations $(A,B) \in \rho'$. As $|C \cap D| < k$, iterating this argument for all $x \in C \cap D$ yields that the small side of some $(A,B) \in \rho' \subseteq \rho$ meets every~$P_x$.
    \end{claimproof}
    
    By \cref{claim:SmallSideInForbiddenStarMeetingEveryPx} there is a separation $(A,B) \in \rho$ whose small side $A$ meets every $P_x$ at least once in some vertex $p_x$. Since~$\rho$ is a star, we then have $p_x \in A \subseteq F$ for every $(E, F) \in \rho \setminus \{(A,B)\}$. Therefore, by~\eqref{eq:DefOfTau}, 
    \[
        \rho' := \{(A \cup C, B \cap D)\} \cup \{(E \cap D, F \cup C) : (E, F) \in \rho\setminus\{(A,B)\}\}
    \]
    is contained in $\tau$.
    
    We claim that $\rho' \in \cF$, contradicting that $\tau$ is an $\cF$-tangle of $S^{\sigma}_k$. For this, we show that $(C, D)$ emulates $(A \cap C, B \cup D)$ in $S_k$, from which the claim follows as $\cF$ is strongly closed under shifting and $\abs{(A \cap C) \cap (B \cup D)} \leq \abs{A \cap B} + \abs{C \cap D} \leq 2k - 2$.
    Indeed, let $\{E,F\} \in S_k$ such that $(A \cap C, B \cup D) \leq (E,F)$. 
    Then $p_x \in A \cap C \subseteq E$ for every $x \in C \cap D$, so $E$ meets every $P_x$. Then by \cref{claim:CornerOfOrderLessThank}, $\{E \cup C, F \cap D\}$ has order $<k$, which concludes the proof.
\end{proof}

For the proof of \cref{lem:ReflemForInfGraphsInessStars} it remains to show that $\vS^\sigma_k$ is $\cF'$-separable.

\begin{LEM} \label{lem:SkSigmaIsFSeparable}
    Let $G$ be a graph, $k \in \N$, and let $\sigma \subseteq \vS_k$ be a finite star. Suppose that every separation in $\sigma$ is either $(3k-2)$-robust on the small side or has an inverse that is closely related to some $k$-profile in $G$. If $\cF$ is a set of stars in $\vS_k$ and closed under shifting, then $\vS^\sigma_k$ is $\cF'$-separable where $\cF' := \cF \cup \{\{\sv\} : \vs \in \sigma\}$. 
\end{LEM}

\begin{proof}
    Since every $\sv$ with $\vs \in \sigma$ is maximal in $\vS^\sigma_k$ with respect to the partial order on $\vS_k$, we have $\{\sv \vee \vt\} = \{\sv\} \in \cF'$ for all $\vt$ that emulate some $\vr \leq \sv$ in $\vS^\sigma_k$. Hence, it suffices to show that $\vS^\sigma_k$ is $\cF$-separable.
    So let $\vr, \vr' \in \vS^\sigma_k$ be given such that $\vr < \rv'$, and pick a separation $s \in \vS^\sigma_k$ of minimal order such that $\vr \leq \vs \leq \rv'$. We claim that $\vs$ and $\sv$ emulate $\vr$ and $\vr'$, respectively, in $\vS^\sigma_k$ for $\cF$. This clearly implies the assertion.

    We show that $\vs$ emulates $\vr$ in $\vS^\sigma_k$ for $\cF$; the other case is symmetric. By the choice of $s$ and \cref{lem:Fishlemma}, $\vx \wedge \vs$ has order at least $|s|$ for all $x \in S^\sigma_k$ with $\vx \geq \vr$. So $\vs \vee \vx \in \vS^\sigma_k$ by submodularity and \cref{lem:Fishlemma}, which implies that $\vs$ emulates $\vr$ in $\vS^\sigma_k$. Since $\cF$ is closed under shifting in $\vS_k$, it thus suffices to show that $\vs$ also emulates $\vr$ in $\vS_k$. 
    So suppose for a contradiction that there is some $\vx \in \vS_k$ with $\vx \geq \vr$ such that $\vx \vee \vs$ has order $\geq k$; since $\sigma$ is finite, we may choose $\vx$ so that $x$ and $y$ are nested for as many $\vy \in \sigma$ as possible. 
    Let $\sigma' \subseteq \sigma$ consist of all those $\vy \in \sigma$ such that $y, x$ cross. As $\vs \geq \vr$ and $\vr \in \vS^\sigma_k$, we have $\vr \leq \yv$ for all $\vy \in \sigma'$. 
    If also $\vs \leq \yv$ for all $\vy \in \sigma'$, then $\vs \wedge \vx \in \vS^\sigma_{\aleph_0}$. Since $\vr \leq \vs \wedge \vx \leq \vs \leq \rv'$, it follows by the choice of $s$ that $|\vs \wedge \vx| \geq |s|$. By submodularity, this implies $|\vs \vee \vx| \leq |x| < k$ as desired.

    Hence, we may assume that $\vy \leq \vs$ for some $\vy \in \sigma'$. Then by \cref{claim:ExSideMeetingAllRays,claim:CornerOfOrderLessThank} in the proof of \cref{lem:ReflemForInfGraphsInessStars} if $\vy$ is $(3k-2)$-robust on the small side or by definition if $\yv$ is closely related to a $k$-profile in $G$, one of $\vx \vee \vy$ and $\vx \wedge \yv$ has order $\leq |x|$; let $\vt$ be that corner. By \cref{lem:Fishlemma}, $t$ is nested with more separations in $\sigma$ than $x$, so we have that $|\vs \vee \vt| < k$ by the choice of $x$. Since $|\vs \vee \vt| \geq |\vs \vee \vx|$, which it is straightforward to check, this concludes the proof.
\end{proof}

We are now ready to prove \cref{lem:ReflemForInfGraphsInessStars}:

\begin{proof}[Proof of \cref{lem:ReflemForInfGraphsInessStars}]
    Since~$\sigma$ is finite and has finite interior,~$S_k^\sigma$ is finite by \cref{prop:FiniteStarAndInteriorImpliesFiniteSkSigma}. Moreover, by \cref{lem:SkSigmaIsFSeparable}, $\vS^\sigma_k$ is $\cF'$-separable. 
    Hence, we can apply \cref{thm:TTDinASS} to $S_k^\sigma$ and $\cF'$, which yields either an $\cF'$-tangle of~$S^\sigma_k$ or an~$S_k^\sigma$-tree over~$\cF'$. Let us first assume that there is an $\cF'$-tangle~$\tau'$ of~$S^\sigma_k$. 
    Then inductively applying \cref{lem:TangleInsideStar}~\ref{itm:TangleInsideStar2} and~\ref{lem:TangleInsideStarComplicated}~\ref{itm:TangleInsideStarComplicated2} yields that $\tau'$ extends to an $\cF'$-tangle~$\tau$ of~$S_k$ as desired. 
    So we may assume that there is an $S^\sigma_k$-tree $(T, \alpha)$ over $\cF'$.
    By definition, $(T, \alpha)$ is also an $S_k$-tree over $\cF'$, so we are left to show that every $\vs \in \sigma$ appears as a leaf separation of $(T, \alpha)$. For this, let $\vs \in \sigma$ be given, and set $O := \{\vr \in \vS^\sigma_k: \vr \leq \sv\}$. As $\vs \in \sigma$, the set $O$ is a consistent orientation of $S^\sigma_k$. Moreover, since either $\vs$ is $\ell$-robust on the small side and $\cF$ is $(\ell-1)$-bounded, or $\sv$ is contained in an $\cF$-tangle,~$O$ avoids~$\cF$. Thus,~$O$ has to live at a leaf of $T$, which then has to be associated with $\{\sv\}$; so $\vs$ is a leaf separation of~$(T, \alpha)$.
\end{proof}

\section{Tangle-tree duality in infinite graphs}\label{sec:TTDInfGraphs}

In this section we prove Theorems \ref{main:TTDLocFinGraphs}, \ref{main:TTDCountableGraphs} and \ref{main:TTDInfGraphs}. In fact, we prove the following more general duality theorem for $\cF$-tangles:

\begin{mainresult}\label{main:TTDInfGraphsFTangles}
Let $G$ be a graph, and let $k \in \N$. Further, let $\cF$ be a finitely bounded, nice set of stars in~$\vS_k$. Then exactly one of the following assertions holds:
\begin{enumerate}[label=\rm{(\roman*)}]
    \item\label{itm:TTDInfGraphsFTanglesTangle} There exists a principal $\cF$-tangle of $S_k$ which is not induced by an end of combined degree $<k$. 
    \item\label{itm:TTDInfGraphsFTanglesTree} There exists a weakly exhaustive $S_k$-tree over $\cF \cup \cU^\infty_k$.
\end{enumerate}
\end{mainresult}

For the proof of \cref{main:TTDInfGraphs}, we first show two auxiliary statements. We will use the first of them to prove that in \cref{main:TTDInfGraphsFTangles} not both, \cref{itm:TTDInfGraphsFTanglesTangle} and \cref{itm:TTDInfGraphsFTanglesTree}, can hold.

\begin{LEM}\label{lem:TangleIsInducedByAnEnd}
    Let $G$ be a graph, $k \in \N$, and let $((A_i, B_i))_{i \in \N}$ be a weakly exhaustive increasing sequence in~$\vS_k$.
    Then every principal, consistent orientation~$\tau$ of $S_k$ with $(A_i, B_i) \in \tau$, for every $i \in \N$, is induced by an end of~$G$ of combined degree $<k$.
\end{LEM}

\begin{proof}
    Let $\cS$ be the set of all finite subsets of $V(G)$. A \defn{direction} of $G$ is a map $f$ with domain $\cS$ that maps every $S \in \cS$ to a component of $G - S$ such that $f(S) \subseteq f(S')$ whenever $S' \subseteq S$. 
    Now $\tau$ defines a direction $f$ of $G$ as follows. Let $S \in \cS$ be given. Since $((A_i, B_i))_{i \in \N}$ is weakly exhaustive, there exists some $i \in \N$ such that $S \subseteq A_i$. As $\tau$ is principle, there further exists a component $C_i$ of $G - (A_i \cap B_i)$ such that $(V(G-C_i), V(C_i) \cup (A_i \cap B_i)) \in \tau$; and since $\tau$ is consistent, we have $C_i \subseteq G[B_i \setminus A_i]$. Now let $f(S) := C$ where $C$ is the unique component of $G - S$ that contains $C_i$. It is straightforward to check that $f$ is a direction because $\tau$ is consistent. 
    
    By \cite{Ends}*{Theorem 2.2}, there exists an end $\eps$ of $G$ such that for all $\{A,B\} \in S_{\aleph_0}$ we have $(A,B) \in \tau_{\eps}$ if and only if $f(A \cap B) \subseteq G[B\setminus A]$. 
    We claim that $\tau$ is induced by $\eps$. Indeed, let $(A,B) \in \tau$ be given. Since $A \cap B$ is finite and $((A_i, B_i))_{i \in \N}$ is weakly exhaustive, there exists some $i \in \N$ such that $A \cap B \subseteq A_i$.
    As $C_i$ is connected and avoids $A \cap B \subseteq A_i$, either $C_i \subseteq G[B\setminus A]$ or $C_i \subseteq G[A\setminus B]$. In fact, because $\tau$ is consistent, we find $C_i \subseteq G[B\setminus A]$, and hence $(A,B) \leq (V(G-C_i), V(C_i) \cup (A_i \cap B_i))$. It follows that $(A,B) \in \tau_{\eps}$ since $\tau_{\eps}$ is consistent and $(V(G-C_i), V(C_i) \cup (A_i \cap B_i)) \in \tau_{\eps}$ by the choice of $\eps$.

    The assertion now follows since $\Delta(\eps) < k$ by \cref{prop:DegreeOfAnEndWitnessedBySeps}~\cref{itm:DegreeWitnessingSeq:SeqImpliesDeg}.
\end{proof}

In order to show that in \cref{main:TTDInfGraphsFTangles} at least one of \ref{itm:TTDInfGraphsFTanglesTangle} and \ref{itm:TTDInfGraphsFTanglesTree} holds, we will construct for every graph with no tangles as in \ref{itm:TTDInfGraphsFTanglesTangle} an $S_k$-tree as in \ref{itm:TTDInfGraphsFTanglesTree}. For this, we need the following theorem which follows easily from a result of the author, Jacobs, Knappe and Pitz \cite{LinkedTDInfGraphs}:

\begin{THM}\label{thm:StartingTD}
    Let $G$ be a graph, $\ell \in \N$, and suppose there exists some $k \in \N$ such that every end of $G$ has combined degree $<k$ and every critical vertex set of $G$ has size $<k$. Then $G$ has a \td~$(T, \cV)$ of adhesion $<k$ into finite parts such that
    \begin{enumerate}[label=\rm{(\roman*)}]
        \item\label{itm:StartingTDinfdegree} for every node $t$ of $T$ of infinite degree its associated star $\sigma_t$ is in $\cU^\infty_k$, and infinitely many $(A,B) \in \sigma_t$ are tight on the left side and satisfy $A \cap B = V_t$,
        \item\label{itm:StartingTDends} every end $\eta$ of $T$ is home to a unique end $\eps$ of $G$, and we have $\liminf_{i \in \N} |V_{r_i} \cap V_{r_{i+1}}| = \Delta(\eps)$ for every $\eta$-ray $R = r_1 r_2 \dots$ in $T$,
        \item \label{itm:StartingTDstronglyrelevant} for every edge $\ve = (t, s)$ in $T$, if
        $\deg_T(s) < \infty$,
        then the separation induced by~$\ve$ is $\ell$-robust on the small side, and
        \item \label{itm:StartingTD:otherProperties}
        $(T, \cV)$ is tight and displays the infinities\footnote{See \cite{LinkedTDInfGraphs}*{\S2} for definitions. We remark that it is not important what these properties actually mean, we only need them once in \cref{sec:RefiningEssStars} to apply \cite{LinkedTDInfGraphs}*{Lemma~8.6} to this \td.}.
    \end{enumerate}
\end{THM}

\begin{proof}
   Since all ends of $G$ have finite degree, $G$ contains no half-grid minor, and in particular no subdivision of $K_{\aleph_0}$. 
   So by \cite{halin78}*{Theorem 10.1} and \cite{LinkedTDInfGraphs}*{Theorem~2.2}, $G$ has finite tree-width, that is, admits a \td\ into finite parts.
   Thus, by \cite{LinkedTDInfGraphs}*{Theorem 4' \& Lemma~8.3}, $G$ admits a \td\ $(T, \cV)$ into finite parts which satisfies \ref{itm:StartingTDinfdegree} -- \ref{itm:StartingTD:otherProperties}. Indeed, \cref{itm:StartingTDstronglyrelevant} and \cref{itm:StartingTD:otherProperties} follow immediately by \cite{LinkedTDInfGraphs}*{Lemma~8.3}.
   \cref{itm:StartingTDinfdegree} and \cref{itm:StartingTDends} hold because all bags of $(T, \cV)$ are finite and $(T, \cV)$ \emph{displays the infinities of~$G$}\footnote{\label{note:Def}See \cite{LinkedTDInfGraphs}*{\S2} for a definition.}.  
   Moreover, by \cite{LinkedTDInfGraphs}*{Theorem 4' (L1)}, all adhesion sets of $(T, \cV)$ are either \emph{linked}\footnoteref{note:Def} to an end or a critical vertex set of $G$. Since all ends of $G$ have combined degree $<k$ and all critical vertex sets of $G$ have size~$<k$, this implies that all adhesion sets of $(T, \cV)$ have size $<k$.
\end{proof}

In the proof of \cref{main:TTDInfGraphsFTangles}, if a graph has no tangles as in \cref{itm:TTDInfGraphsFTanglesTangle}, we will apply \cref{lem:ReflemForInfGraphsInessStars} to the stars associated with the nodes of the \td\ $(T, \cV)$ from \cref{thm:StartingTD}. We then obtain an $S_k$-tree as in \cref{itm:TTDInfGraphsFTanglesTree} by sticking the $S_k$-trees obtained from \cref{lem:ReflemForInfGraphsInessStars} together along $T$ as follows:

\begin{construction} \label{constr:StickingSkTreesTogether}
    Let $(T, \cV)$ be a \td\ of a graph $G$ of adhesion $< k \in \N$, and let~$\cF$ be a set of stars in $\vS_k$. Further, let $U \subseteq V(T')$ be a set of nodes of $T$. Assume that for every node $t \in U$, we are given a weakly exhaustive $S_k$-tree $(T^t, \alpha^t)$ over $\cF \cup \{\{\sv\} : \vs \in \sigma_t\}$ in which each $\vs \in \sigma_t$ appears as a leaf separation. Assume further that the stars $\sigma_t$ associated in $T$ with nodes $t \in V(T)\setminus U$ are contained in $\cF$. 

    Set $T^t := T[N_{T}(t) \cup \{t\}]$ and $\alpha^t((s,t)) := (U_s,U_t)$ for all $t \in V(T) \setminus U$.
    We obtain the tree $T'$ from the disjoint union of the trees~$T^t$ by identifying for every edge $e = \{t_1, t_2\}$ of $T$ the nodes $s_1 \in T^{t_1}$ and $u_2 \in T^{t_2}$ as well as $s_2 \in T^{t_2}$ and $u_1 \in T^{t_1}$ where $(u_i, s_i)$ is the unique (leaf) edge of $T^{t_i}$ such that $\alpha^{t_i}(u_i, s_i) = (U_{t_i}, U_{t_{3-i}})$. 
    For each edge~$e$ of $T'$, we then set~$\alpha'(\ve)$ to be~$\alpha^t(\ve)$, where~$t$ is a node of~$T$ such that~$e \in E(T^t)$. It is straightforward to check that $(T', \alpha')$ is an $S_k$-tree over $\cF$. Moreover, since $(T, \cV)$ is a \td\ of $G$ and because each $(T^t, \alpha^t)$ is weakly exhaustive, $(T', \alpha')$ is weakly exhaustive.
\end{construction}

We are now ready to prove \cref{main:TTDInfGraphsFTangles}.

\begin{proof}[Proof of \cref{main:TTDInfGraphsFTangles}]
We first show that not both, \ref{itm:TTDInfGraphsFTanglesTangle} and \ref{itm:TTDInfGraphsFTanglesTree}, can hold for $G$. For this, suppose that there is a weakly exhaustive $S_k$-tree $(T, \alpha)$ over~$\cF$ as in \cref{itm:TTDInfGraphsFTanglesTree}, and let $\tau$ be any consistent orientation of $S_k$. We claim that $\tau$ is not a tangle as in \cref{itm:TTDInfGraphsFTanglesTangle}.
Indeed, since~$\tau$ is consistent, it induces via $\alpha^{-1}$ a consistent orientation~$O$ of~$E(T)$. It follows that $O$ either contains a sink or a directed ray. If~$O$ contains a sink, that is, if there is a node $t$ of~$T$ all whose incident edges are oriented inwards by~$O$, then $\sigma_t \subseteq \tau$. But~$T$ is over~$\cF \cup \cU^\infty_k$, and thus~$\tau$ is either not principal or not an $\cF$-tangle.
Otherwise, if $O$ contains a directed ray $R = r_1 r_2 ...$, then, since $(T, \alpha)$ is weakly exhaustive, $\tau$ contains an infinite weakly exhaustive increasing sequence $(\alpha(r_i, r_{i+1}))_{i \in \N}$ of separations of order $<k$. It follows, by \cref{lem:TangleIsInducedByAnEnd}, that $\tau$ is either not principal or induced by an end of $G$ of combined degree $< k$. 

We now show that at least one of \cref{itm:TTDInfGraphsFTanglesTangle} and \cref{itm:TTDInfGraphsFTanglesTree} holds. For this, suppose that \cref{itm:TTDInfGraphsFTanglesTangle} does not hold: that all principal $\cF$-tangles of~$S_k$ are induced by ends of combined degree $< k$. 
We show that then \cref{itm:TTDInfGraphsFTanglesTree} must hold: that there exists a weakly exhaustive $S_k$-tree over $\cF \cup \cU_k^\infty$.
By \cref{prop:EndsInduceFTangles,lem:TanglesAtLargeCritVS}, all ends of $G$ have combined degree $<k$, and all critical vertex sets of $G$ have size $<k$. 
We may thus apply \cref{thm:StartingTD}, which yields a \td~$(T, \cV)$ of $G$ of adhesion $<k$ into finite parts.
Let $t$ be a node of $T$ finite degree. Since $\interior(\sigma_t) = V_t$ is finite, and because of \cref{thm:StartingTD}~\ref{itm:StartingTDstronglyrelevant},~$\sigma_t$ satisfies the premise of \cref{lem:ReflemForInfGraphsInessStars}. 
Since $\sigma_t$ is not home to any ends as $\interior(\sigma_t) = V_t$ is finite, and since~$G$ does not contain any principal $\cF$-tangles in~$G$ of order~$k$ that are not induced by an end, the star~$\sigma_t$ cannot be home to any principal $\cF$-tangles of order $k$. 
Moreover, since $\cF$ is nice, and hence $\cP'_k \subseteq \cF$ as well as $\{(V(G), A)\} \in \cF$ for all sets $A$ of fewer than $k$ vertices, $\sigma_t$ can also not be home to any non-principal $\cF$-tangle by \cref{lem:NonPrincipalPkTanglesAreInducedByUFTangles}.
Hence, applying \cref{lem:ReflemForInfGraphsInessStars} to $\sigma_t$ yields a finite $S_k$-tree $(T^t, \alpha^t)$ over $\cF' := \cF \cup \{\{\sv\} : \vs \in \sigma_t\}$ in which each $\vs \in \sigma_t$ appears as a leaf separation.

Since $\sigma_t \in \cU_k^\infty$ for all infinite-degree nodes $t$ of $T$ by \cref{thm:StartingTD}~\cref{itm:StartingTDinfdegree}, applying \cref{constr:StickingSkTreesTogether} to $(T, \cV)$ and the $(T^t, \alpha^t)$ yields a weakly exhaustive $S_k$-tree over $\cF \cup \cU_k^\infty$, as desired.
\end{proof}

\begin{proof}[Proof of \cref{main:TTDInfGraphs}]
    By \cref{lem:UkStronglyClosedUnderShifting}, $\cT^*_k$ is nice. As the $\cT^*_k$-tangles of $S_k$ are precisely the $k$-tangles in~$G$ if $|G| \geq k$ \cite{TangleTreeGraphsMatroids}*{Lemma~4.2}, the assertion follows immediately by applying \cref{main:TTDInfGraphsFTangles} to the $(3k-3)$-bounded, nice set $\cT^*_k$.
\end{proof}

\begin{proof}[Proof of \cref{main:TTDLocFinGraphs}]
    By definition, $\cU^\infty_k$ is empty if $G$ is locally finite. Moreover, since every $k$-tangle is a $k$-profile, inductively applying the profile property yields that every $k$-tangle in a locally finite graph is principal (see also \cite{DiestelBook16noEE}*{Exercise~43 in Ch.\ 12}). 
    The assertion thus follows immediately from \cref{main:TTDInfGraphs}.
\end{proof}

\begin{proof}[Proof of \cref{main:TTDCountableGraphs}]
    By \cref{main:TTDInfGraphs}, it is enough to show that if \cref{itm:TTDInf:Tree} of \cref{main:TTDInfGraphs} holds, then also \cref{itm:TTDInf:Tree} of \cref{main:TTDCountableGraphs} holds. For this, assume that $(T, \alpha)$ is a weakly exhaustive $S_k$-tree over $\cT^* \cup \cU^\infty_k$.
    If $(T, \alpha)$ is even over $\cT^*$, then we are done. Otherwise we define an $S_k$-tree $(T', \alpha')$ as follows. 

    By pruning the tree $T$ if necessary, we may assume that $(T, \alpha)$ is \defn{irredundant}: for every node $t$ of $T$ and neighbours $t', t''$ of $t$ we have $\alpha(t', t) = \alpha(t'', t)$ if and only if $t' = t''$.\footnote{Pick any node $r$ of $T$ and for every separation $(A,B) \in \sigma_r$ a neighbour $t_{(A,B)}$ of $r$ such that $\alpha(t_{(A,B)}, r) = (A,B)$. Deleting from $T$ all components of $T - r$ that do not contain $t_{(A,B)}$ for any $(A,B) \in \sigma_r$ turns $(T, \alpha)$ into an $S_k$-tree $(T', \alpha\!\!\restriction_{T'})$ in which the neighbourhood of $r$ has changed but $\sigma_r$ has not, and neither has $\sigma_t$ for any other node $t$ of $T'$. So $(T', \alpha\!\!\restriction_{T'})$ is still a weakly exhaustive $S_k$-tree over $\cT^* \cup \cU^\infty_k$. Now think of $T'$ as rooted in $r$ and proceed along its levels.}
    Then all nodes in $T$ have countable degree. Indeed, let $t$ be a node of $T$ of infinite degree and consider $\sigma_t := \{\alpha(s, t) : \{s,t\} \in E(T)\}$. Then~$\sigma_t$ is a star in $\vS_k$ because $(T, \alpha)$ is an $S_k$-tree over $\cT^* \cup \cU_k^\infty$. Since~$G$ is countable, there are only countably many small separations of $G$ of the form $(A,V(G))$ for some set $A$ of fewer than $k$ vertices, and also $\sigma_t$ can contain at most countably many separations of the form $(A,B)$ with $B \neq V(G)$, as any such separation contains a vertex in its strict small side $A \setminus B$ that is not contained in the strict small side of any other separation in $\sigma_t$. Hence, $\sigma_t$ is countable, and thus $N_T(t)$ is countable since $(T, \alpha)$ is irredundant.

    Let $r$ be an arbitrary node of $T$, and for every infinite-degree node $t$ of $T$, let $\{s^t_i :  i \in \N_0\}$ be an enumeration of its neighbourhood such that $s^t_0$ is the unique vertex of $rTt$ that is incident with $t$. 
    Let~$F$ be the forest obtained from $T$ by deleting all edges $e$ of the form $e = \{t, s^t_i\}$ where $\deg(t) = \infty$ and $i \geq 2$. 
    Now the tree $T'$ is obtained from $F$ by simultaneously adding, for every infinite-degree node $t$ of $T$, a ray $R_t := r^t_2 r^t_3\dots$, the edge $\{t, r^t_2\}$, and all edges of the form $\{r^t_i, s^t_i\}$ for $i \geq 2$. 
    Further, let $\alpha': \vE(T') \rightarrow \vS_k$ be defined via $\alpha'(\ve) := \alpha(\ve)$ for all edges $e \in E(T') \cap E(T)$, and $\alpha'(s^t_i, r^t_i) := \alpha(s^t_i, t)$ and $\alpha'(r^t_i, r^t_{i+1}) := \bigvee_{j \geq i} \alpha(s^t_j, t)$ for all $i \geq 2$ as well as $\alpha'(t, r^t_2) := \alpha(s^t_0, t) \vee \alpha(s^t_1, t)$.
    
    Then $(T', \alpha')$ is again an $S_k$-tree, and it is straightforward to check that $(T', \alpha')$ is weakly exhaustive and over $\cT^* \cup \{\sigma \in \cU_k : |\sigma| = 3\}$. To turn $(T', \alpha')$ into an $S_k$-tree over $\cT^*$, we need to modify every star~$\sigma'_t$ associated with a node $t$ of $T'$ that is in $\cU_k$ (but not in $\cT^*$) to become a star in $\cT^*$. The easiest way to do this is to add the interior of $\sigma'_t$ to the separators of all separations in $\sigma'_t$. For this, we add a subdivision vertex $v_e$ to those edges $e = \{s,t\}$ of $T'$ whose endvertices $s$ and $t$ are both associated in $T'$ with stars $\sigma'_s$, $\sigma'_t$ in $\cU_k\setminus \cT^*$. We denote the arising tree with $T''$. To define $\alpha''$, let $e$ be an edge of $T''$, and first assume that $e = \{s,t\}$ for nodes $s,t$ of $T'$. If $\sigma'_s, \sigma'_t \in \cT^*$, then let $\alpha''(\ve) := \alpha'(\ve)$. Otherwise, if $\sigma'_t \notin \cT^*$, then let $\alpha''((s,t)) := \alpha'((s,t)) \wedge (V(G), \interior(\sigma'_t))$ (and $\alpha''((t,s))$ accordingly), i.e.\ $\alpha''((s,t))$ is obtained from $\alpha'((s,t))$ by adding the interior of $\sigma'_t$ to the small side (and hence the separator) of $\alpha'((s,t))$. Second, if $e = \{v_f, t\}$ where $f = \{s,t\} \in E(T')$, then let $\alpha''((v_f,t)) := \alpha'((s,t)) \wedge (V(G), \interior(\sigma'_t))$ (and $\alpha''((t,v_f))$ accordingly). Since $|\interior(\sigma'_t)| < k$ if $\sigma'_t \notin \cT^*$, the image of $\alpha''$ is contained in $\vS_k$; so $(T'', \alpha'')$ is again an $S_k$-tree, which by definition is over $\cT^*$.
\end{proof}

We conclude this section with an example that shows that \cref{main:TTDInfGraphs} fails for sets $\cF$ of stars that are nice but not finitely bounded.

\begin{example}\label{ex:FinitelyBoundedIsNecessaryForAnSkTree}
    Let $G = (V,E)$ be the graph with vertex set $V := \{v_{ij} : (i, j) \in [4] \times \N\}$ and edge set $E := \{\{v_{ij}, v_{i'j'}\} \in V(G) : j' \in \{j, j+ 1\}\}$ (see \cref{fig:CounterexNotFinBounded}). 
    Set $\cF' := \{\{(A_k, B_k)\} : k \in \N\}$ where $(A_k, B_k) := (\{v_{ij} : i \in [4], j \geq k\} , \{v_{ij} : i \in [4], j \leq k\})$ (see \cref{fig:CounterexNotFinBounded}), and let $\cF := \cF' \cup \cP'_5 \cup \{(V(G), A) : A \subseteq V(G), \abs{A} < 5\}$. Clearly, $\cF$ is strongly closed under shifting, and thus a nice set of stars in $\vS_5$.
    
        \begin{figure}[ht]
        \centering
        \definecolor{dgreen}{rgb}{0,0.8,0}
\scalebox{0.82}{%
\begin{tikzpicture}

\draw [->,line width=1pt] (5,3) -- (12,3);
\draw [->,line width=1pt] (5,0) -- (12,0);
\draw [line width=1pt] (5,3)-- (5,0);
\draw [line width=1pt] (7,3)-- (7,0);
\draw [line width=1pt] (9,3)-- (9,0);
\draw [line width=1pt] (11,3)-- (11,0);

\draw (5.6,1.8) node[anchor=north west] {\LARGE $K_8$};
\draw (7.6,1.8) node[anchor=north west] {\LARGE $K_8$};
\draw (9.6,1.8) node[anchor=north west] {\LARGE $K_8$};


\foreach \x in {1.5,3.5,5.5,7.5}{
           \draw [shift={(\x,1.5)},line width=1pt,color=dgreen]  plot[domain=-0.55:0.55,variable=\t]({4*cos(\t r)},{4*sin(\t r)});
        }
\foreach \x in {8.5,10.5,12.5,14.5}{
           \draw [shift={(\x,1.5)},line width=1pt,color=dgreen]  plot[domain=2.59:3.693,variable=\t]({4*cos(\t r)},{4*sin(\t r)});
        }
\draw [->,line width=1pt,color=dgreen] (4.88,-0.2) -- (4.2,-0.2);
\draw [->,line width=1pt,color=dgreen] (6.88,-0.2) -- (6.2,-0.2);
\draw [->,line width=1pt,color=dgreen] (8.88,-0.2) -- (8.2,-0.2);
\draw [->,line width=1pt,color=dgreen] (10.88,-0.2) -- (10.2,-0.2);


\draw (13,1.5) node {\Large $\eps$};

\draw (5.5,3.3) node {$v_{11}$};
\draw (5.4,2.25) node {$v_{21}$};
\draw (5.4,0.8) node {$v_{31}$};
\draw (5.5,-0.27) node {$v_{41}$};

\draw (7.5,3.3) node {$v_{12}$};
\draw (7.4,2.25) node {$v_{22}$};
\draw (7.4,0.8) node {$v_{32}$};
\draw (7.5,-0.27) node {$v_{42}$};

\draw (9.5,3.3) node {$v_{13}$};
\draw (9.4,2.25) node {$v_{23}$};
\draw (9.4,0.8) node {$v_{33}$};
\draw (9.5,-0.27) node {$v_{43}$};

\draw (11.5,3.3) node {$v_{14}$};
\draw (11.4,2.25) node {$v_{24}$};
\draw (11.4,0.8) node {$v_{34}$};
\draw (11.5,-0.27) node {$v_{44}$};

\draw[color=dgreen] (4.6,-0.65) node {$B_1$};
\draw[color=dgreen] (6.6,-0.65) node {$B_2$};
\draw[color=dgreen] (8.6,-0.65) node {$B_3$};
\draw[color=dgreen] (10.6,-0.65) node {$B_4$};

\draw[color=dgreen] (5.4,-0.65) node {$A_1$};
\draw[color=dgreen] (7.4,-0.65) node {$A_2$};
\draw[color=dgreen] (9.4,-0.65) node {$A_3$};
\draw[color=dgreen] (11.4,-0.65) node {$A_4$};

\begin{scriptsize}
\foreach \x in {5,7,9,11}{
            \foreach \y in {0,1,2,3}
            \draw [fill=black] (\x,\y) circle (2pt);
        }
        
\end{scriptsize}
\end{tikzpicture}
}%
        \vspace{-8mm}
        \caption{\cref{ex:FinitelyBoundedIsNecessaryForAnSkTree}}
        \label{fig:CounterexNotFinBounded}
    \end{figure}

    It is easy to check that $G$ has precisely one $\cF$-tangle of order $5$, the one induced by its end $\eps$. Indeed, any consistent orientation of $S_5$ that is not induced by $\eps$ has to contain some $(B_i, A_i)$, and is hence not an $\cF$-tangle.
    Thus, $G$ has no $\cF$-tangle of order~$5$ that is not induced by an end of combined degree $<5$. 
    
    But $G$ has no $S_5$-tree over $\cF \cup \cU^\infty_k$ either. Indeed, any such tree would have to contain a ray whose edges are associated with separations that form an increasing sequence in $\tau_{\eps}$. By the definition of $\cP'_k$, and because $\cU^\infty_k$ is empty since $G$ is locally finite, the nodes of that ray would eventually be associated with stars in $\cF'$, a contradiction because~$\cF'$ contains only singleton stars. 
\end{example}

\section{Bramble-treewidth duality: an application of the tangle-tree duality theorem} \label{sec:BrambleTWDuality}

A set $U$ of vertices of a graph $G$ is \defn{connected} if $G[U]$ is connected. A \defn{bramble} in $G$ is a set $\cB$ of mutually touching connected sets of vertices of $G$ where two sets of vertices are said to \defn{touch} if they have a vertex in common or if $G$ contains an edge between them. The \defn{order} of a bramble is the least number of vertices that \defn{cover} the bramble, in that they meet every element of it.

Seymour and Thomas proved the following duality between high-order brambles and small tree-width (see also \cite{DiestelBook16noEE}*{Theorem~12.4.3}): 

\begin{THM}{\cite{ST1993GraphSearching}}\label{thm:BrambleTreeDualityFiniteGraphs} Let $k \in \N$. A finite graph has tree-width $< k$ if and only if it contains no bramble of order $> k$.
\end{THM}

\cref{thm:BrambleTreeDualityFiniteGraphs} extends to infinite graphs with one adaptation.
For this, let us first note that every graph~$G$ with a ray contains a bramble of infinite order. Indeed, if $R = r_0r_1\dots$ is a ray in $G$, then $\cB := \{\{r_i : i \geq n\} : n \in \N\}$ is a bramble, and it clearly cannot be covered by finitely many vertices. However, the graph that consists of just a single ray has clearly tree-width $1$.

Thus, in order to ensure that brambles of high order force the tree-width of a graph up, we have to restrict the class of brambles we consider to those that are finite. 
Here, a bramble is \defn{finite} if all its elements are finite.
Note that, clearly, every bramble in a finite graph is finite. 

With this definition, \cref{thm:BrambleTreeDualityFiniteGraphs} extends to infinite graphs: 

\begin{THM} \label{thm:BrambleTreeDualityInfGraphs} Let $k \in \N$. A graph has tree-width $< k$ if and only if it contains no finite bramble of order~$> k$.
\end{THM}

The probably shortest way to prove \cref{thm:BrambleTreeDualityInfGraphs} is to use a result discovered by Thomas (see \cite{thomassen89}*{Theorem~14} for a proof) which says that an infinite graph~$G$ has tree-width $\leq k$ if and only if every finite subgraph of~$G$ has tree-width $\leq k$. Hence, if a graph~$G$ has tree-width $\geq k$ for some $k \in \N$, then some finite subgraph $H \subseteq G$ has tree-width $\geq k$. \cref{thm:BrambleTreeDualityFiniteGraphs} then yields a bramble $\cB$ in~$H$ of order $> k$. Since~$H$ is finite,~$\cB$ will be finite as well, and it is easy to see that $\cB$ is also a bramble of order $>k$ in $G$. The other implication can be proved similarly to the finite case. 

In this section we present an alternative proof of \cref{thm:BrambleTreeDualityInfGraphs}, which derives \cref{thm:BrambleTreeDualityInfGraphs} from \cref{main:TTDInfGraphsFTangles}. Though this proof is not as short as the one indicated above, it provides an example of how \cref{main:TTDInfGraphsFTangles} can be employed to obtain other duality theorems for infinite graphs.
Moreover, the proof we present in this section is direct in that it does not make use of the finite result (\cref{thm:BrambleTreeDualityFiniteGraphs}); but, of course, it easily implies it. Further, we prove \cref{thm:BrambleTreeDualityInfGraphs} by showing a more general duality that includes brambles and tree-width as well as $\cU_k$-tangles and $S_k$-trees over $\cU_k$.

Recall that $\defnm{\cU_k} := \big\{\sigma \subseteq \vS_{\aleph_0} : \sigma \text{ is a star with } |\interior(\sigma)| < k\big\}$ for $k \in \N$.
The main result of this section is \cref{main:TangleBrambleTreeDualityInfGraphs}, which we restate here for convenience:

\begin{customthm}{\cref*{main:TangleBrambleTreeDualityInfGraphs}} \label{thm:TangleBrambleTreeDualityInfGraphs}
\emph{The following assertions are equivalent for all graphs $G$ and $k \in \N$: 
\begin{enumerate}[label=\rm{(\roman*)}]
    \item\label{itm:TangleBrambleTreeDualityInfGraphsTangle} $G$ has a $\cU_k$-tangle of order $k$ that is not induced by an end of combined degree $<k$.
    \item\label{itm:TangleBrambleTreeDualityInfGraphsBramble} $G$ has a finite bramble of order at least $k$.
    \item\label{itm:TangleBrambleTreeDualityInfGraphsTree} $G$ has no weakly exhaustive $S_k$-tree over $\cU_k$.
    \item\label{itm:TangleBrambleTreeDualityInfGraphsTW} $G$ has tree-width at least $k - 1$.
\end{enumerate}}
\end{customthm}

\noindent  \cref{thm:TangleBrambleTreeDualityInfGraphs} generalizes a result of Diestel and Oum \cite{TangleTreeGraphsMatroids}*{Theorem~6.5} to infinite graphs; its proof is inspired by theirs.
\medskip

\begin{proof}[Proof of \cref{thm:TangleBrambleTreeDualityInfGraphs}]
    \ref{itm:TangleBrambleTreeDualityInfGraphsTangle} $\Leftrightarrow$ \ref{itm:TangleBrambleTreeDualityInfGraphsTree} is \cref{lem:UkStronglyClosedUnderShifting} and \cref{main:TTDInfGraphsFTangles}.
    \ref{itm:TangleBrambleTreeDualityInfGraphsTree} $\Leftrightarrow$ \ref{itm:TangleBrambleTreeDualityInfGraphsTW} is analogous to \cite{TangleTreeGraphsMatroids}*{Lemma~6.3}.

    It remains to show \ref{itm:TangleBrambleTreeDualityInfGraphsTangle} $\Leftrightarrow$ \ref{itm:TangleBrambleTreeDualityInfGraphsBramble}.
    For \ref{itm:TangleBrambleTreeDualityInfGraphsBramble} $\Rightarrow$ \ref{itm:TangleBrambleTreeDualityInfGraphsTangle} we closely follow the proof of \cite{TangleTreeGraphsMatroids}*{Lemma 6.4}. Let $\cB$ be a finite bramble of order at least $k$. For every $\{A,B\} \in S_k$, since $A \cap B$ is too small to cover $\cB$ but every two sets in $\cB$ touch and are connected, exactly one of the sets $A \setminus B$ and $B \setminus A$ contains an element of~$\cB$. Thus, $O := \{(A,B) \in \vS_k : B\setminus A \text{ contains an element of } \cB\}$ is an orientation of $S_k$, and it is consistent for the same reason.
    
    To show that $O$ avoids $\cU_k$, let $\sigma = \{(A_i,B_i) : i \in I\} \in \cU_k$ be given. Then $|\interior(\sigma)| < k$, so some $C \in \cB$ avoids $\interior(\sigma)$, and hence lies in the union of the sets $A_i \setminus B_i$. But these sets are disjoint, since $\sigma$ is a star, and they have no edges between them. Hence, $C$ lies in one of them, $A_j\setminus B_j$ say, which implies that $(B_j, A_j) \in O$. But then $(A_j, B_j) \notin O$, so $\sigma  \not\subseteq O$ as claimed.
    
    To see that $O$ is not induced by an end of~$G$ of combined degree~$<k$, let $\eps$ be an end of $G$. If $\cdeg(\eps) < k$, then in particular $\dom(\eps) < k$. Since $\cB$ has order at least $k$, there exists some $C \in \cB$ that avoids $\Dom(\eps)$. Moreover, again because $\Delta(\eps) < k$, there exists, by \cref{prop:DegreeOfAnEndWitnessedBySeps}~\cref{itm:DegreeWitnessingSeq:DegImpliesSeq}, a weakly exhaustive increasing sequence $((A_i, B_i))_{i \in \N}$ of separations in $\tau_\eps$ such that $(A_i \cap B_i) \cap (A_j \cap B_j) \subseteq \Dom(\eps)$ for $i \neq j \in \N$. Since $C \cap \Dom(\eps) = \emptyset$ and $C$ is finite, this implies that there is some $j \in \N$ such that $C \subseteq A_j\setminus B_j$. But then $(B_j, A_j) \in O$, and thus $O$ is not induced by $\eps$.

    For \ref{itm:TangleBrambleTreeDualityInfGraphsTangle} $\Rightarrow$ \ref{itm:TangleBrambleTreeDualityInfGraphsBramble} assume that $G$ has an $\cU_k$-tangle $\tau$ of order $k$ that is not induced by an end of combined degree $<k$. 

    \begin{claim}\label{claim:ProofOfBrambleTWDuality}
        \emph{For every separation $(A,B) \in \tau$ the set $\{(C,D) \in \tau : (A,B) \leq (C,D)\}$ has a maximal element.}
    \end{claim}

    \begin{claimproof}
        Suppose towards a contradiction that the set $S := \{(C,D) \in \tau : (A,B) \leq (C,D)\}$ has no maximal element, and let $S' \subseteq S$ consist of all separations $(C,D)$ in $S$ that are tight on the big side and have a connected strict big side $D\setminus C$.
        Since $\tau$ avoids $\cU_k$ and is hence principal, there exists for every $(C,D) \in \tau$ a component $K$ of $G-(A \cap B)$ such that $(V(G-K), V(K) \cup N_G(K)) \in \tau$. In particular, for every $(C,D) \in S$ there exists $(C', D') \in S'$ such that $(C,D) \leq (C', D')$.
        It follows that $S'$ is non-empty and that $S'$ has no maximal element either.
        Thus, $S'$ contains a strictly increasing sequence $((C_i, D_i))_{i \in \N}$ of separations.
        
        Since all $(C_i, D_i)$ have order $<k$, we may assume that $|C_i \cap D_i| = \ell$ for some $\ell < k$ and all $i \in \N$, by passing to a subsequence of $((C_i, D_i))_{i \in \N}$ if necessary. 
        We claim that $((C_i, D_i))_{i \in \N}$ is weakly exhaustive. Indeed, since all $D_i\setminus C_i$ are connected, all separators $C_i \cap D_i$ are distinct. Hence, $X := \bigcap_{i \in \N} \bigcup_{j \geq i} (C_j \cap D_j)$ has size $<\ell$; let $j \in \N$ such that $X \subseteq C_i \cap D_i$ for all $i \geq j$. Pick some $u \in (C_j \cap D_j) \setminus X$, and let $j' \in \N$ such that $u \in C_{j'}\setminus D_{j'}$. Since $(C_j, D_j)$ is tight on the big side and $D_j \setminus C_j$ is connected, there is for every $v \in D_j\setminus C_j$ a $v$--$u$ path in $G$ that avoids $X$. So if $D := \bigcap_{i \in \N} (D_i \setminus C_i)$ is non-empty, then there exists a finite $D$--$(C_{j'}\setminus D_{j'})$ path that avoids $X$. But this path has to meet all separators $C_i \cap D_i$ with $i \geq j'$ in vertices outside of~$X$, a contradiction.

        Thus, $((C_i, D_i))_{i \in \N}$ is weakly exhaustive, which by \cref{lem:TangleIsInducedByAnEnd} contradicts the assumption that $\tau$ is not induced by an end of combined degree $<k$.
    \end{claimproof}

    We now define a finite bramble $\cB$ as follows. Let $V(G)$ be equipped with a fixed well-ordering. Then for every non-empty set $U \subseteq V(G)$ there exists a unique element in $U$ which is least in the well-ordering; we denote this vertex with $v_U$.

    Now for every separation $(A,B) \in \tau$ that is maximal in $\tau$ (with respect to the partial order on $\tau$ induced by $\vS_k$), we pick a finite, connected set $U_{(A,B)} \subseteq B \setminus A$ which contains $v_{B \setminus A}$ and for every vertex in $A \cap B$ at least one of its neighbours. For this note that such a set exists since $(A,B)$ is maximal in $\tau$, and thus $G[B\setminus A]$ is connected and $N_G(B\setminus A) = A \cap B$. 
    We then put in $\cB$ precisely all sets $U_{(A,B)}$. By definition, all elements of $\cB$ are finite and connected. Moreover, $\cB$ has order at least $k$. Indeed, since~$\tau$ avoids~$\cU_k$, there exists for every set $U$ of at most $k-1$ vertices of $G$ a component $C$ of $G-U$ such that $(V(G-C), V(C) \cup N_G(C)) \in \tau$. Then for every maximal separation $(A,B)$ in $\tau$ with $(V(G-C), V(C) \cup N_G(C)) \leq (A,B)$ the set $U_{(A,B)}$ avoids $U$; and such an $(A,B)$ exists by \cref{claim:ProofOfBrambleTWDuality}.

    To conclude the proof, it remains to show that the sets in $\cB$ mutually touch. For this, let $U := U_{(A,B)}$, $U' := U_{(A', B')} \in \cB$ be given. If $v := v_{B\setminus A} = v_{B'\setminus A'} =: v'$, then $U$ and $U'$ intersect by construction. So one of $v$ and $v'$, say $v$, is strictly smaller in the well-ordering, which by the choice of $v'$ implies that $v \in A'$.
    We claim that $(A \cap B) \cap (B'\setminus A') \neq \emptyset$. Then the assertion follows. Indeed, let $u$ be any vertex in that set. Since $A\cap B \subseteq N_G(U)$ by the choice of $U$, there is a $v$--$u$ path in $G[U \cup \{u\}]$. But as $v \in A'$ and $u \in B'\setminus A'$, it follows that $U$ meets $A'\cap B'$. Hence $U$ and $U'$ touch since $A'\cap B' \subseteq N_G(U')$ by the choice of $U'$.

    To prove the claim suppose for a contradiction that $(A \cap B) \cap (B'\setminus A') = \emptyset$. If also $(B\setminus A) \cap (B'\setminus A') = \emptyset$, then it follows that $(B', A') \leq (A,B)$, which contradicts the consistency of~$\tau$ as $(A', B'), (A,B) \in \tau$. Hence, there is a vertex $u' \in (B\setminus A) \cap (B'\setminus A')$.
    Since both $(A', B')$ and $(A,B)$ are maximal separations in~$\tau$, we have $B'\setminus A' \not\subseteq B\setminus A$, and hence the exists a vertex $w \in (B' \setminus A') \cap A \neq \emptyset$. As $B'\setminus A'$ is connected, there exists a $u'$--$w$ path in $G[B'\setminus A']$. But this path has to meet $A \cap B$ since $u' \in B$ and $w \in A$, which concludes the proof.
\end{proof}

\begin{proof}[Proof of \cref{thm:BrambleTreeDualityInfGraphs}]
    This is \cref{itm:TBTDuality:Bramble} $\Leftrightarrow$ \cref{itm:TBTDuality:TW} of \cref{thm:TangleBrambleTreeDualityInfGraphs}.
\end{proof}

We conclude this section with the following corollary of \cref{lem:ReflemForInfGraphsInessStars} that the author, Jacobs, Knappe and Pitz use in \cite{LinkedTDInfGraphs}:

\begin{COR}
    Let $G$ be a graph of tree-width $\leq w \in \N$, and let~$\sigma$ be a finite star of separations of~$G$ of order~$\leq w+1$ whose interior is finite.
    Suppose that all separations in~$\sigma$ are $\ell$-robust on the small side for $\ell := (w+1)^2(w+2)+w+1$.
    Then $\torso(\sigma)$ has tree-width $\leq w$.
\end{COR}

\begin{proof}
    By definition, $\cU_{w+2}$ is $(w+1)$-bounded, and by \cref{lem:UkStronglyClosedUnderShifting}, $\cU_{w+2}$ is nice, so we may apply \cref{lem:ReflemForInfGraphsInessStars}, which yields that there is either an $\cU_{w+2}$-tangle $\tau$ of $G$ with $\sigma \subseteq \tau$ or a finite $S_{w+2}$-tree $(T, \alpha)$ over $\cU_{w+2} \cup \{\{\sv\} : \vs \in \sigma\}$ in which each $\vs \in \sigma$ appears as a leaf separation. 
    Suppose first that the former holds. Since the interior of $\sigma$ is finite and $\sigma \subseteq \tau$, the $\cU_{w+2}$-tangle $\tau$ cannot be induced by an end of $G$.
    But since~$G$ has tree-width $\leq w$, it has no $\cU_{w+2}$-tangles of order $w+2$ that are not induced by an end by \cref{main:TangleBrambleTreeDualityInfGraphs}, a contradiction. 
    So we may assume the latter. 
    It is easy to check that $(T, \alpha)$ induces a \td\ $(T, \cV)$ of $G$ (cf.\ \cite{TangleTreeGraphsMatroids}*{Lemma 6.3}) whose bags have size $\leq w+1$, unless they are associated with leaves of ~$T$ whose incident edge induces a separation in $\sigma$. By restricting the bags in $\cV$ to $\interior(\sigma)$, we obtain a \td\ $(T, \cV')$ of $G[\interior(\sigma)]$ of width $\leq w$. In fact, since $(T, \alpha)$ contains all $\vs \in \sigma$ as leaf separations, $(T, \cV')$ is even \td\ of $\torso(\sigma)$. Thus, $\torso(\sigma)$ has tree-width $\leq w$.
\end{proof}

\section{Refining trees of tangles}\label{sec:RefiningEssStars}

Besides the tangle-tree duality theorem, Robertson and Seymour \cite{GMX} proved the \emph{tree-of-tangles theorem}, which asserts that for every $k \in \N$ every finite graph has a \td\ such that its $k$-tangles live at different nodes of the tree. Erde \cite{JoshRefining} combined this theorem and the tangle-tree duality theorem into one, by constructing a single \td\ such that every node either accommodates a single $k$-tangle or is too small to accommodate one, in that it is associated with a star in $\cT_k$. In fact, he showed that such a \td\ can be obtained from any given one that efficiently distinguishes all the $k$-tangles, by refining its inessential parts.

The author \cite{SARefiningEssParts} improved Erde's result by constructing further refinements of the essential parts of that \td, yielding a \td\ that has the additional property that all its essential bags are as small as possible. In this section, we extend this result to infinite graphs. We then obtain \cref{main:TTDPlusToT} as a simple corollary.

To state the main result of this section, we first need some further definitions.
Following \cite{ToTinfOrder}, we call two regular $k$-profiles $\tau, \tau'$ in a graph $G$ \defn{combinatorially distinguishable} if at least one of them is principal or they are both non-principal but such that there exists a set $X \subseteq V(G)$ such that $(V(K) \cup X, V(G-K)) \in \tau$ for all $K \in \cC_X$ and such that $(V(G-K), V(K) \cup X) \in \tau'$ for a component $K \in \cC_X$.

A set $\cF$ of stars in $\vS_k$ is \defn{profile-respecting} if every $\cF$-tangle of $\vS_k$ is a $k$-profile in $G$.
A $k$-profile in~$G$ is \defn{bounded} if it does not extend to an $\aleph_0$-profile\footnote{Equivalently, a principal $k$-profile $\tau$ in $G$ is bounded if and only if it is neither induced by an end nor of the form $\{(A,B) \in \vS_k : X \subseteq B\}$ for a set $X \in \crit(G)$ of size~$\geq k$ (cf.\ \cite{EndsAndTangles}*{Theorem~3}). Moreover, every non-principal $k$-profile is unbounded.}.

Given some set $S \subseteq S_{\aleph_0}$, a star $\sigma \subseteq \vS$ is \defn{exclusive} for some set $\cO$ of consistent orientations of $S$ if it is contained in exactly one orientation in $\cO$. If $O \in \cO$ is that orientation, we say that $\sigma$ is \defn{$O$-exclusive (for $\cO$)}.
Similarly, a bag $V_t$ of a \td\ of $G$ is \defn{exclusive} (for $\cO$) if $\sigma_t$ is exclusive for $\cO$. 

If~$(T, \cV)$ and $(\tilde{T}, \tilde{\cV})$ are both \tds\ of~$G$, then $(T, \cV)$ \defn{refines} $(\tilde{T}, \tilde{\cV})$ if the set of separations induced by the edges of~$T$ is a superset of the set of separations induced by the edges of~$\tilde{T}$.

\begin{THM} \label{thm:RefToTsInfGraphs}
Let $G$ be a graph, $k \in \N$, and let $\cF$ be a finitely bounded, profile-respecting, nice set of finite stars in~$S_k$. Further, let $(\tilde{T}, \tilde{\cV})$ be a \td\ of $G$ which distinguishes all combinatorially distinguishable $\cF$-tangles of order $k$ such that every separation induced by an edge of $\tilde{T}$ distinguishes a pair of $\cF$-tangles of order $k$ efficiently. Then there exists a \td\ $(T, \cV)$ of $G$ which refines $(\tilde{T}, \tilde{\cV})$ and which is such that
\begin{enumerate}[label=\rm{(\roman*)}]
    \item\label{itm:RefToTsInfGraphsEnds} every end of $G$ of combined degree $<k$ lives in an end of $T$;
    \item\label{itm:RefToTsInfGraphsEnds2} if every end of $\tilde{T}$ is home to an end of $G$, then also every end of $T$ is home to an end of $G$;
    \item\label{itm:RefToTsInfGraphsCrit} every non-principal $\cF$-tangle of order $k$ which does not live in an end of $\tilde{T}$ lives at a node $t$ of $T$ with $\sigma_t \in \cU^\infty_k$;
    \item\label{itm:RefToTsInfGraphsIness} for every inessential node $t$ of $T$ we have $\sigma_t \in \cF$; and
    \item\label{itm:RefToTsInfGraphsBoundedTangle} every bag $V_t$ of $(T, \cV)$ that is home to a bounded $\cF$-tangle of order $k$ is of smallest size among all the exclusive bags of \tds\ of $G$ that are home to the $\cF$-tangle living in $V_t$.
\end{enumerate}
\end{THM}

\noindent We remark that if $G$ is locally finite, then there exists by \cite{CanonicalTreesofTDs}*{Theorem~7.3} for every $k \in \N$ a \td~$(\tilde{T}, \tilde{\cV})$ of $G$ which efficiently distinguishes all $k$-tangles in $G$, and which thus satisfies the premise of \cref{thm:RefToTsInfGraphs}.  
\medskip

In the remainder of this section we prove \cref{thm:RefToTsInfGraphs}. For this, we need two more refining lemmas. The first one lets us refine stars which are home to a bounded tangle, and the second one generalizes our refining lemma for inessential stars, \cref{lem:ReflemForInfGraphsInessStars}, to stars whose interior is infinite.

To show the first lemma, we need the following result of \cite{SARefiningEssParts}:

\begin{LEM}{\cite{SARefiningEssParts}*{Proof of Lemma 4.3}}\label{lem:StarWithMinIntNestedClRel}
    Let $k \in \N$, let $\cQ$ be some set of $k$-profiles in a graph\footnote{In \cite{SARefiningEssParts} the assertion of \cref{lem:StarWithMinIntNestedClRel} is shown only for finite graphs. However, the same proof works for infinite graphs as long as~$\sigma$ is finite and has finite interior.} $G$, and let~$\tau \in \cQ$. Further, let~$\sigma \subseteq \tau$ be a finite star with finite interior, and suppose that every separation in~$\sigma$ efficiently distinguishes some pair of $k$-profiles in $\cQ$. 
    Then there exists a star $\rho \subseteq \tau$ with $\sigma \leq \rho$ whose interior is of smallest size among all stars in~$\tau$ that are exclusive for $\cQ$ and which has the further property that all the separations in $\rho$ are closely related to~$\tau$.
\end{LEM}

The following lemma refines essential stars that are home to a bounded tangle in a similar way as \cref{lem:ReflemForInfGraphsInessStars} refines inessential stars. It generalizes \cite{SARefiningEssParts}*{Lemma 4.3} to infinite graphs.

\begin{LEM}\label{lem:ReflemForInfGraphsBoundedEssStars}
    Let $G$ be a graph, $k \in \N$, and let $\cF$ be a finitely bounded, profile-respecting, nice set of finite stars in $\vS_k$. Let $\sigma \subseteq \vS_k$ be a star, and suppose that every separation in $\sigma$ efficiently distinguishes some pair of $\cF$-tangles of $S_k$. Further, suppose there is a unique $\cF$-tangle $\tau$ of $S_k$ that satisfies $\sigma \subseteq \tau$.
    If $\tau$ is bounded, then there exists a star~$\sigma' \subseteq \tau$ whose interior is of smallest size among all exclusive stars in~$\tau$, and a finite $S_k$-tree over $\cF \cup \{\sigma'\} \cup \{\{\sv\} : \vs \in \sigma\}$ in which each~$\vs \in \sigma$ appears as a leaf separation.
\end{LEM}

\noindent Note that we show in the proof that the interior of every exclusive star in $\tau$ is finite. Thus, `of smallest size' is well-defined.

\begin{proof}
    Since $\tau$ is bounded and thus not induced by an end, and because $\sigma$ is $\tau$-exclusive, no end of $G$ lives in~$\sigma$ by \cref{prop:EndsInduceFTangles}. Hence, by \cref{prop:RayInTorsoExtends} and because all separations in $\sigma$ are tight on the small side by \cref{lem:EffYieldsTight}, $\torso(\sigma)$ is rayless.
    Moreover, by \cref{lem:TanglesAtLargeCritVS,lem:TanglesAtSmallCritVS}, $\torso(\sigma)$ is tough, and so by \cref{prop:RaylessToughGraphsAreFinite}, $\interior(\sigma)$ is finite. In particular, again by \cref{lem:TanglesAtSmallCritVS}, $\sigma$ is finite.
    
    We can thus apply \cref{lem:StarWithMinIntNestedClRel} to $\sigma$ and~$\tau$, which yields a star $\sigma' \subseteq \tau$ with $\sigma \leq \sigma'$ which is closely related to~$\tau$ and whose interior is of smallest size among all exclusive stars in~$\tau$. 
    As all separations in~$\sigma$ distinguish some pair of $\cF$-tangles in $G$ efficiently, and are hence closely related to some $\cF$-tangle by \cref{prop:eff}, it follows that all the stars $\rho_{s} := \{\sv\} \cup \{\vr : \vr \leq \vs, \vr \in \sigma\}$ for $\vs \in \sigma'$ satisfy the premise of \cref{lem:ReflemForInfGraphsInessStars}. We thus obtain, for every~$\rho_{s}$, a finite $S_k$-tree $(T^{s}, \alpha^{s})$ over $\cF \cup \{\{\rv\} : \vr \in \rho_{s}\}$ in which each $\vr \in \rho_{s}$ appears as a leaf separation. 

    Let $T$ be the tree obtained from the disjoint union of the trees $T^{s}$ by identifying their leaves $v_{s}$ where~$v_{s}$ is the unique leaf of $T^{s}$ whose incident edge induces $\vs$. Then $(T, \alpha)$ with $\alpha(\ve) = \alpha^{s}(\ve)$ where $T^{s}$ is the unique tree containing $e$ is an $S_k$-tree over $\cF \cup \{\sigma'\} \cup \{\{\sv\} : \vs \in \sigma\}$.
\end{proof}

The next lemma generalizes \cref{lem:ReflemForInfGraphsInessStars} to stars with infinite interior:

\begin{LEM}\label{lem:ReflemInfInessStars}
    Let $G$ be a graph, $k \in \N$, and let $\cF$ be a finitely bounded, nice set of stars in~$\vS_k$. 
    Let $\sigma := \{\vs_i : i \in I\} \subseteq \vS_k$ be a star, and suppose that every separation in~$\sigma$ efficiently distinguishes some pair of $k$-profiles in $G$ that avoid $\cF$.
    Set $\cF' := \cF \cup \{\{\sv_i\} : i \in I\}$.
    Then either there is a principal $\cF'$-tangle of~$S_k$ that is not induced by an end of $G$ of combined degree $<k$, or there is a weakly exhaustive~$S_k$-tree $(T, \alpha)$ over~$\cF' \cup \cU^\infty_k$ in which each $\vs_i$ appears as a leaf separation.

    In particular, $(T, \alpha)$ can be chosen so that every end of $T$ is home to an end of $G$.
\end{LEM}

If $(T,\cV)$ is a \td\ of a graph $G$, then a tree $T'$ obtained from $T$ by edge contractions induces the \td\ $(T',\cV')$ of $G$ whose bags are $V'_t = \bigcup_{s \in t} V_s$ for every $t \in T'$, where we denote the vertex set of $T'$ as the set of branch sets. Recall that $V_e := V_s \cap V_t$ for every edge $e=st$ of~$T$.

To prove \cref{lem:ReflemInfInessStars} we need the following lemma:

\begin{lemma}{\cite{LinkedTDInfGraphs}*{Lemma~8.6}} \label{lem:ContractingEdgesStillFinite}
    Let $G$ be a graph, and let $(T, \cV)$ be the \td\ of $G$ that satisfies \cref{itm:StartingTDends} and \ref{itm:StartingTD:otherProperties} of \cref{thm:StartingTD}. Let $F$ be some set of edges of $T$ such that no edge in $F$ is incident with a node of infinite degree and such that 
    for every end $\eta$ of $T$, the set $F$ avoids cofinitely many edges $e$ of the $\eta$-ray $R$ in $T$ with $|V_e| = \Delta(\eps)$ where $\eps$ is the end of $G$ that lives in $\eta$. Then the \td\ obtained from $(T, \cV)$ by contacting all edges in $F$ has still finite parts.
\end{lemma}

\begin{proof}[Proof of \cref{lem:ReflemInfInessStars}]
    Let us assume that there is no $\cF'$-tangle that is as desired and show that there then exists a weakly exhaustive $S_k$-tree over $\cF' \cup \cU^\infty_k$.
    Set $\ell := \max\{3k-2, k(k-1)m + m\}$ where $m \in \N$ is such that $\cF$ is $m$-bounded.
    
    Assume first that $\sigma$ has finite interior; in this case, we allow that separations in $\sigma$ do not efficiently distinguish two $\cF$-tangles if they are $\ell$-robust on the small side. We show that there then exists an $S_k$-tree $(T, \alpha)$ as desired such that $T$ is rayless.  Set $\cX := \{X \subseteq V(G) : X = A \cap B \text{ for infinitely many } (A,B) \in \sigma\}$ and $\rho := \{(A,B) \in \sigma : A \cap B \notin \cX\} \cup \{(A_X, B_X) : X \in \cX\}$ where $(A_X, B_X)$ is the supremum over all $(A,B) \in \sigma$ with $A \cap B = X$. Clearly, we have $A_X \cap B_X = X$, and hence $\rho$ is included in $\vS_k$. 
    We claim that $\rho$ satisfies the premise of \cref{lem:ReflemForInfGraphsInessStars}.
    Indeed, $\rho$ is still a star, and it is finite since $\interior(\sigma)$ is finite. Moreover, every separation in $\rho \cap \sigma$ is either $\ell$-robust on the small side or has an inverse that is closely related to some $\cF$-tangle of order $k$ by \cref{prop:eff}. Further, every separation $(A,B) \in \rho\setminus \sigma$ is $\ell$-robust on the small side by \cref{prop:LeftRobustCrit}, as $G[A\setminus B]$ contains infinitely many tight components of $G-(A \cap B)$ by \cref{lem:EffYieldsTight}. 
    
    Hence, we may apply \cref{lem:ReflemForInfGraphsInessStars} to $\rho$. As $\rho$ is finite and has finite interior, it cannot be home to any $\cF$-tangles that are non-principal or induced by an end by \cref{lem:NonPrincipalPkTanglesAreInducedByUFTangles} and \cref{prop:EndsInduceFTangles}. Since $\rho$ is also not home to any other principal $\cF$-tangles by assumption, we thus obtain a finite $S_k$-tree $(T, \alpha)$ over $\cF \cup \{\{\sv\} : \vs \in \rho\}$ in which each $(A,B) \in \rho$ appears as a leaf separation; let us denote the respective leaf of~$T$ with $t_{(A,B)}$. We now obtain the desired rayless $S_k$-tree over $\cF \cup \cU_k^\infty$ by adding, for every $(A,B) \in \sigma \setminus \rho$ a leaf $u_{(A,B)}$ and the edge $\{u_{(A,B)}, t_{(A_X, B_X)}\}$ where $X = A \cap B$. 
    \smallskip

    We now turn to the case that $\sigma$ has infinite interior. By \cref{prop:EndsInduceFTangles}, no end of $G$ of combined degree $\geq k$ lives in $\sigma$, and hence, by \cref{lem:DegreeOfEndsInTorso}, all ends of $\torso(\sigma)$ have combined degree $<k$. Further, by \cref{lem:TorsosAndCriticalVertexSets,lem:TanglesAtLargeCritVS}, all critical vertex sets of $\torso(\sigma)$ have size $<k$. Hence, we may apply \cref{thm:StartingTD} to $\torso(\sigma)$, which yields a \td\ $(T', \cV')$ of $\torso(\sigma)$ of adhesion $<k$.
    Since every separator $A \cap B$ of a separation $(A,B) \in \sigma$ is complete in $\torso(\sigma)$, there exists for every such $A \cap B$ a node $t_{A \cap B}$ of $T'$ such that $A \cap B \subseteq V'_{t_{A \cap B}}$. We then obtain a \td\ $(T'', \cV'')$ of $G$ by adding for every $(A,B) \in \sigma$ a node $u_{(A,B)}$ and the edge $\{u_{(A,B)}, t_{A \cap B}\}$ to $T'$ and assigning the bag $A$ to $u_{(A,B)}$.
    
    Let $F$ be the set of edges $e = \{s,t\}$ of $T''$ with $\sigma''_s, \sigma''_t \notin \cU^\infty_k$ whose induced separation is neither $\ell$-robust nor in $\sigma$.
    Let $(\tilde{T}'', \tilde{\cV}'')$ and $(\tilde{T}', \tilde{\cV}')$ be the \tds\ obtained from $(T'', \cV'')$ and $(T', \cV')$, respectively, by contracting all edges in $F$. 
    Then all bags $\tilde{V}''_t$ for nodes $t$ of $\tilde{T}''$ not of the form $u_{(A,B)}$ are finite: By \cref{thm:StartingTD}~\ref{itm:StartingTDends} every end $\eta$ of $T'$ is home to an end $\eps'$ of $\torso(\sigma)$ with $\liminf_{e \in R} |V_e| = \Delta(\eps')$ where $R$ is any $\eta$-ray. Then by \cref{lem:DegreeOfEndsInTorso} the same holds true for the ray $R$ in $T'' \supseteq T'$ and some end $\eps$ of $G$ with $\Delta(\eps) = \Delta(\eps')$.
    Thus, by \cref{lem:ContractingEdgesEllRobust} applied to $\eps$ and the separations of $G$ induced by the ray $R$ in $T''$, the set $F$ and $(T', \cV')$ satisfy the premise of \cref{lem:ContractingEdgesStillFinite}, and thus $(\tilde{T}', \tilde{\cV}')$ has still finite parts. As $\tilde{V}''_t = \tilde{V}'_t$ for all nodes $t$ of $\tilde{T}''$ not of the form $u_{(A,B)}$, these bags are finite. 

    Now let $t$ be a node of $\tilde{T}''$ which is neither a leaf of the form $u_{(A,B)}$ for some $(A,B) \in \sigma$ nor associated with a star $\tilde{\sigma}''_t$ in $\cU_k^\infty$. Then all separations in $\tilde{\sigma}''_t$ are either $\ell$-robust on the small side or efficiently distinguish two $\cF$-tangles. Indeed, by the definition of $\tilde{T}''$, no edge $e=st$ of $\tilde{T}''$ is in $F$, which implies that the separation $(U_s,U_t)$ induced by $\ve$ is either $\ell$-robust, or in $\sigma$, or we have $\tilde{\sigma}''_s \in \cU^\infty_k$. In the second case, $(U_s, U_t)$ efficiently distinguishes two $\cF$-tangles, and in the third case $(U_s, U_t)$ is $\ell$-robust on the small side by \cref{prop:LeftRobustCrit}: infinitely many separations in $\sigma'_s$ are tight on the small side by \cref{thm:StartingTD}~\ref{itm:StartingTDinfdegree}, and hence, as all separations in $\sigma$ are tight on the small side, also infinitely many separations in $\sigma''_s$ are tight on the small side; so $U_s$ contains infinitely many tight components of $G-V_e$. Since $\interior(\tilde{\sigma}''_t)$ is finite, we may apply the first case to $\tilde{\sigma}''_t$. As $\sigma$ is not home to any principal $\cF$-tangles that are not induced by an end of combined degree $<k$, $\tilde{\sigma}''_t$ is not home to any such $\cF$-tangles either. Hence, we obtain a rayless $S_k$-tree $(T^t, \alpha^t)$ over $\cF \cup \cU^\infty_k \cup \{\{\sv\} : \vs \in \tilde{\sigma}''_t\}$ in which each $\vs \in \tilde{\sigma}''_t$ appears as a leaf separation.
    Applying \cref{constr:StickingSkTreesTogether} to $(\tilde{T}'', \tilde{\cV}'')$ and the $(T^t, \alpha^t)$ yields a weakly exhaustive $S_k$-tree $(T, \alpha)$ over $\cF \cup \cU_k^\infty \cup \{\{\sv\} : \vs \in \sigma\}$. Moreover, by construction, each~$\vs_i \in \sigma$ appears as a leaf separation of $(T, \alpha)$. For the `in particular'-part, note that every end of $T'$ is home to an end of $\torso(\sigma)$ by \cref{thm:StartingTD}~\cref{itm:StartingTDends}, and hence every end of $\tilde{T}''$ is home to an end of $G$ by \cref{lem:DegreeOfEndsInTorso}. Since all $T^t$ are rayless, the assertion follows.
\end{proof}

With \cref{lem:ReflemForInfGraphsBoundedEssStars,lem:ReflemInfInessStars} at hand, we are ready to prove the main result of this section.

\begin{proof}[Proof of \cref{thm:RefToTsInfGraphs}]
    We define for every node $t$ of $\tilde{T}$ an $S_k$-tree $(T^t, \alpha^t)$ as follows.
    If $t$ is home to a bounded $\cF$-tangle, then let $(T^t, \alpha^t)$ be the finite $S_k$-tree obtained from applying \cref{lem:ReflemForInfGraphsBoundedEssStars} to $\tilde{\sigma}_t$.
    If~$t$ is inessential or only home to $\cF$-tangles of order $k$ that are either non-principal or induced by ends of combined degree $<k$, then let $(T^t, \alpha^t)$ be the $S_k$-tree obtained from applying \cref{lem:ReflemInfInessStars} to $\tilde{\sigma}_t$.

    Then applying \cref{constr:StickingSkTreesTogether} to $(\tilde{T}, \tilde{\cV})$ and the $(T^t, \alpha^t)$ yields a weakly exhaustive $S_k$-tree $(T, \alpha)$. 
    It is now straightforward to check that $(T, \cV)$ with $V_t := \interior(\sigma_t)$ is a \td\ of $G$ (cf.\ \cite{confing}*{\S4}); in particular, every edge $\ve$ of $T$ induces the separation $\alpha(\ve)$. Then $(T, \cV)$ satisfies \cref{itm:RefToTsInfGraphsEnds} to \cref{itm:RefToTsInfGraphsBoundedTangle}: 
    By construction, no end of combined degree $<k$ can live at a node of $T$, and hence they have to live in ends of $T$; so \cref{itm:RefToTsInfGraphsEnds} holds. \cref{itm:RefToTsInfGraphsEnds2} follows by the `in particular' part of \cref{lem:ReflemInfInessStars}.
    Property \cref{itm:RefToTsInfGraphsCrit} holds because non-principal $\cF$-tangles that live at nodes of $\tilde{T}$ have to live at nodes of~$T$ by the `in particular' part of \cref{lem:ReflemInfInessStars}, but they cannot live at nodes that are associated with stars in $\cF$. \cref{itm:RefToTsInfGraphsIness} and \cref{itm:RefToTsInfGraphsBoundedTangle} follow by \cref{lem:ReflemInfInessStars,lem:ReflemForInfGraphsBoundedEssStars}.
\end{proof}

We conclude this section with the proof of \cref{main:TTDPlusToT}. For this, we need the following theorem,
which is immediate from the proof of \cite{LinkedTDInfGraphs}*{Corollary~5'}:

\begin{theorem} \label{thm:ToTInfTangles}
    Every graph $G$ without half-grid minor admits a \td\ $(T, \cV)$ which efficiently distinguishes all combinatorially distinguishable $\aleph_0$-tangles in $G$. 
    
    Moreover, $(T, \cV)$ can be chosen so that every end of $T$ is home to an end of $G$, and every non-principal $\aleph_0$-tangle lives at a node $t$ of $T$.
\end{theorem}

\begin{proof}[Proof of \cref{main:TTDPlusToT}]
    By assumption, $G$ has no end of combined degree $\geq k$, and hence no half-grid minor. Let $(T', \cV')$ be the \td\ of $G$ from \cref{thm:ToTInfTangles}. Since~$G$ has no bounded $k$-tangle by assumption, $(T', \cV')$ in particular distinguishes all combinatorially distinguishable $k$-tangles. By contracting all edges of~$T$ whose induced separations do not efficiently distinguish some pair of $k$-tangles, we obtain a \td\ $(\tilde{T}, \tilde{\cV})$ that satisfies the premise of \cref{thm:RefToTsInfGraphs}. Let $(T, \cV)$ be the \td\ obtained from applying \cref{thm:RefToTsInfGraphs} to $(\tilde{T}, \tilde{\cV})$. Then $(T, \alpha)$ is an $S_k$-tree, where $\alpha(t_0, t_1) := (U_{t_0}, U_{t_1})$ for all edges $(t_0, t_1) \in \vE(T)$. In particular, $(T, \alpha)$ is over $\cT^* \cup \cU^\infty_k$ by \cref{itm:RefToTsInfGraphsEnds}, \cref{itm:RefToTsInfGraphsCrit} and \cref{itm:RefToTsInfGraphsIness} of \cref{thm:RefToTsInfGraphs} and because all principal $k$-tangles in $G$ are induced by ends of combined degree $<k$. 

    We claim that $(T, \alpha)$ is as desired. Indeed, it satisfies~\cref{itm:TTDPlusToT:Ends} by \cref{thm:RefToTsInfGraphs}~\ref{itm:RefToTsInfGraphsEnds} and~\ref{itm:RefToTsInfGraphsEnds2} and because every end of $T'$, an hence every end of $\tilde{T}$, is home to an end of $G$.
    Moreover, $(T, \alpha)$ satisfies \cref{itm:TTDPlusToT:Ultrafilter} by \cref{thm:RefToTsInfGraphs}~\cref{itm:RefToTsInfGraphsCrit} and because every non-principal $k$-tangle in $G$ lives at a node of $T'$ and hence of $\tilde{T}$.
    Finally, $(T, \alpha)$ satisfies~\cref{itm:TTDPlusToT:ToT} because $(T, \cV)$ refines $(\tilde{T}, \tilde{\cV})$, and $(\tilde{T}, \tilde{\cV})$ distinguishes all combinatorially distinguishable $k$-tangles in~$G$.
\end{proof}

\bibliographystyle{amsplain}
\bibliography{localbib_new.bib}

\end{document}